\documentclass[a4paper]{article}

\usepackage{amsmath,amssymb,mathtools,amsthm}
\usepackage{xcolor,graphicx,float,enumitem}
\usepackage{stackengine}
\usepackage{comment}

\usepackage{bm}
\usepackage[all]{xy}
\graphicspath{{figures/}}

\usepackage[default,scale=1]{opensans}
\usepackage[T1]{fontenc}

\usepackage[utf8]{inputenc}
\DeclareUnicodeCharacter{00A0}{ }
\usepackage{geometry}
\newcommand{\ee}{\varepsilon}

\usepackage{bm}

\newcommand{\R}{{\mathbb R}}
\newcommand{\Rd}{{\R^d}}

\newcommand{\Ev}{\Upsilon}

\newcommand{\changevariables}{\mathcal{I}}
\newcommand{\diffEv}{\mathfrak{E}}

\DeclareMathOperator{\diver}{div}
\DeclareMathOperator{\trace}{trace}

\DeclareMathOperator{\supp}{supp}

\newcommand*\diff{\mathop{}\!\mathrm{d}}

\usepackage[hidelinks]{hyperref}
\usepackage[nameinlink]{cleveref}

\newtheorem{theorem}{Theorem}[section]
\newtheorem{proposition}[theorem]{Proposition}%
\newtheorem{corollary}[theorem]{Corollary}%
\newtheorem{lemma}[theorem]{Lemma}%

\theoremstyle{definition}
\newtheorem{definition}[theorem]{Definition}%
\newtheorem{remark}[theorem]{Remark}%

\numberwithin{equation}{section}

\makeatletter
\renewcommand*{\@fnsymbol}[1]{\ensuremath{\ifcase#1\or \star \or \dagger\or \ddagger\or
		\mathsection\or \mathparagraph\or \|\or **\or \dagger\dagger
		\or \ddagger\ddagger \else\@ctrerr\fi}}
\makeatother

\usepackage{todonotes}
\DeclareUnicodeCharacter{2212}{-}

\let\OLDthebibliography\thebibliography
\renewcommand\thebibliography[1]{
  \OLDthebibliography{#1}
  \setlength{\parskip}{0.5pt}
  \setlength{\itemsep}{2.5pt plus 0.5ex}
}

\begin{document}
\title
{
    Partial mass concentration
	for \\ 
    fast-diffusions with non-local aggregation terms
}
\author{
    Jos\'e A. Carrillo%
     \thanks{Mathematical Institute, University of Oxford, Oxford OX2 6GG, UK.  \href{mailto:carrillo@maths.ox.ac.uk}{carrillo@maths.ox.ac.uk}} %
    \and 
    Alejandro Fernández-Jiménez%
     \thanks{Mathematical Institute, University of Oxford, Oxford OX2 6GG, UK.  \href{mailto:alejandro.fernandezjimenez@maths.ox.ac.uk}{alejandro.fernandezjimenez@maths.ox.ac.uk}} %
    \and 
    David Gómez-Castro%
     \thanks{Inst. de Matemática Interdisciplinar \& Fac. de Matemáticas, U. Complutense de Madrid.  \href{mailto:dgcastro@ucm.es}{dgcastro@ucm.es}}
}
\maketitle

\begin{abstract}
    We study well-posedness and long-time behaviour of  aggregation–diffusion  equations of the form $\frac{\partial \rho}{\partial t} = \Delta \rho^m + \nabla \cdot (\rho ( \nabla V+  \nabla W \ast \rho ))$ in the fast-diffusion range, $0<m<1$, and $V$ and $W$ regular enough. We develop a well-posedness theory, first in a ball and then in $\Rd$, and characterise
    the long-time asymptotics in the space $W^{-1,1}$ for radial initial data. In the radial setting and for the mass equation,
    viscosity solutions are used to prove partial mass concentration asymptotically as $t\to \infty$, i.e. the limit as $t\to \infty$ is of the form $\alpha \delta_0 + \widehat \rho \, dx$ with $\alpha \geq 0$ and $\widehat\rho\in L^1$. Finally, we give instances of $W \ne 0$ showing that partial mass concentration does happen in infinite time, i.e. $\alpha > 0$. 
    \smallskip 

    \smallskip

\end{abstract}

\section{Introduction}

Nonlinear aggregation-diffusion equations of the form 
\begin{equation}\label{ec:The_equation}
	\frac{\partial \rho}{\partial t} = \Delta \rho^m + \nabla \cdot (\rho \nabla V ) + \nabla \cdot ( \rho \nabla W \ast \rho ),
\end{equation}
are frequent in continuous descriptions of density populations. Here, $\rho_t$ corresponds to a time-dependent probability measure over the whole space $\mathbb{R}^d$. The constant $m>0$ is the diffusion exponent leading to three cases: slow diffusion if $m > 1$, linear diffusion if $m=1$, or fast diffusion if $0<m<1$. The nonlinear diffusion equation corresponding to $V=W=0$ is well-known, see \cite{Vaz06}. The potentials $V(x)$ and $W(x)$ respectively model the confinement behaviour and the interaction kernel describing attraction/repulsion between the individuals of the population. In this work, $V$ and $W$ are assumed to be bounded below, so we can restrict without lost of generality to $V, W \geq 0$.

This family of equations with linear diffusion first appeared to explain a biological phenomenon known as chemotaxis cell movement \cite{KS70,Pat53}. Chemotaxis refers to the natural phenomenon under which cells diffuse in space while secreting some chemical substances attracting other/same cells. The well-known Keller-Segel equation models an aggregation-diffusion behaviour of cell populations. Nonlinear slow diffusions were introduced to cope with overcrowding effects \cite{CC06,HP09}.  Surveys about the classical Keller-Segel model and more general aggregation-diffusion equations can be found in \cite{Hor03,CCY19}. A lot of work has been devoted to this family of equations since the 90s to classify the different cases in which diffusion dominates over aggregation or viceversa \cite{JL92,HV96, BDP06, BCM08, BCL09, CD14, CCV15,CHMV18,CHVY19}. Apart from applications in mathematical biology \cite{TBL06,HP09}, this family of equations also has applications in gravitational collapse or statistical mechanics among others \cite{SC02,SC08}.

The time-dependent equation \eqref{ec:The_equation} is the (formal) $2$-Wasserstein gradient flow of a free-energy functional \cite{O01,CJMTU01,CMV03,AMS08,San15} defined for probability densities $\rho \in L^1 \cap\mathcal{P} ( \mathbb{R}^d )$ by
\begin{equation}\label{ec:The_Free_Energy}
	\mathcal{F} [ \rho ] := - \tfrac{1}{1 - m} \int_{\mathbb{R}^d} \rho (x)^m \, dx + \int_{\mathbb{R}^d} V(x)  \rho(x) dx + \frac{1}{2} \int \int_{\mathbb{R}^d \times \mathbb{R}^d} W (x-y) \, \rho (x) \rho (y) \, dx \, dy\,,
\end{equation}
for $m \in (0,1)$ with $W$ assumed to be symmetric, $W(x)=W(-x)$.
Notice that the suitable extension to probability measures $\mu \in \mathcal{P} ( \mathbb{R}^d )$ is done by taking 
\begin{equation*}
    \tilde{\mathcal{F}} [ \mu ] := - \tfrac{1}{1 - m} \int_{\mathbb{R}^d} \mu_{\mathrm{ac}} (x)^m \, dx + \int_{\mathbb{R}^d} V(x)   d \mu (x) + \frac{1}{2} \int \int_{\mathbb{R}^d \times \mathbb{R}^d} W (x-y) \, d\mu (x) \, d\mu (y) 
\end{equation*}
where $\mu_{\mathrm{ac}}$ is the absolutely continuous part of $\mu$ with respect to the Lebesgue measure \cite{DT84}.

Nonlinear equations of the form \eqref{ec:The_equation} show challenging behaviours with regard to their long-time behaviour and the properties of their steady states. By now, the case of the porous medium diffusion $m > 1$ is well understood \cite{CJMTU01,BCL09,KL10,CCV15,CHMV18,CHVY19}, and so is the case of linear diffusion $m=1$ \cite{AMTU01,MV00,OV00,DP04,BCM08,BCC12}. However, for the case $0 <m < 1$ much less work has been done. Recently, a work by Carrillo, Hittmeir, Volzone, and Yao \cite{CHVY19} studies the equation 
\begin{equation}
	\label{eq:ADE V = 0}
	\frac{\partial \rho}{\partial t} = \Delta \rho^m + \nabla \cdot ( \rho \nabla W \ast \rho )
\end{equation}
and show that all its stationary solutions, with no restriction on $m > 0$, are radially decreasing up to translation. Another work by Carrillo, Delgadino, Dolbeault, Frank and Hoffmann \cite{CDDFH19} shows that, under certain conditions on the potential $W$, the energy minimiser of the corresponding energy functional over probability measures with zero center of mass is likewise split as
\begin{equation*}
	\widehat{\mu} = \left(1- \| \widehat{\rho} \|_{L^1 ( \mathbb{R}^d)}\right) \delta_0 + \widehat{\rho} \, dx,
\end{equation*}
where $\delta_0$ denotes the Dirac delta at $0$ and $\widehat{\rho} (x) > 0$ is the density of the absolutely continuous part of the measure. Furthermore, the presence of the concentrated point measure is known for specific choices of $W$, see \cite{CDFL22}. To the best of our knowledge there are no results in the literature showing that solutions of the parabolic problem \eqref{eq:ADE V = 0} actually converge to these type of minimisers with partial mass concentrated at a Dirac delta at the origin.

This work shows that partial mass concentration occurs in the long time asymptotics for the case $0 < m < 1$. More precisely, we find conditions on the initial data $\rho_0$ and the potentials $V$ and $W$ so that
\begin{itemize}
	\item a notion of solution of the Cauchy problem is provided leading to global in time solutions.
	\item as $t \rightarrow \infty$, the solution experiments a one-point blow up of the form,
	\begin{equation*}
		\widehat{\mu} = \left(\| \rho_0 \|_{L^1 ( \mathbb{R}^d)} - \| \widehat{\rho} \|_{L^1 ( \mathbb{R}^d)}\right) \delta_0 + \widehat{\rho} \, dx,
	\end{equation*}
	with $\widehat \rho \in L^1(\Rd)$. 
\end{itemize}
A previous result by Carrillo, Gómez-Castro, and Vázquez in \cite{CGV22} shows a similar asymptotic behavior for the easier case of no interaction $W = 0$ and $V$ regular enough. In this work, we will further refine the techniques from \cite{CGV22} in order to improve the results given there and to include the case $W \geq 0$. In contrast to \cite{CGV22}, our main results regarding partial mass concentration for long times hold for general radial $L^1 \cap L^\infty$ initial data. Notice the authors in \cite{CGV22} can only prove partial mass concentration for a particular choice of initial data, and any other initial datum above it.

The fast-diffusion case involves new difficulties. It is known that Dirac measures are invariant by the fast-diffusion equation $\frac{\partial u}{\partial t} = \Delta u^m$ when $0 < m < \frac{d-2}{d}$ but they are not generated from $L^1$ initial data \cite{BF81}. However, if we add a drift term, the aggregation caused by it can be strong enough to overcome the fast-diffusion repulsion and produce infinite-time concentration \cite{CGV22}. The range $ \frac{d-2}{d} < m < 1$ is better understood. For example, for the quadratic confinement potential its long-time asymptotics are given by integrable stationary solutions (even if the initial data is given by a Dirac delta), see for example \cite{BBDGV09} and its references. In order to analyse the Cauchy problem for \eqref{ec:The_equation}, we will take advantage from the \textit{a priori} estimates of $\rho$ that comes from its structure and the formal interpretation of \eqref{ec:The_equation} as the $2$-Wasserstein gradient flow associated to the free energy \eqref{ec:The_Free_Energy} leading to a control of the dissipation of the free energy.

One of the key ingredients to analyse partial pass concentration in these models will be to study the equation \eqref{ec:The_equation} in mass variables for radial solutions under the assumptions of radial symmetry of $V$ and $W$. Let us introduce the spatial variable $v = |x|^n |B_1|$ and consider the mass function
\begin{equation*}
	M (t, v) = \int_{\widetilde{B_v}} \rho_t (x) \, dx ,
\end{equation*}
where $\widetilde B_v$ is the ball centred at $0$ such that $|\widetilde{B_v}| = v$. 
The choice of the volume variable is so that $\rho = \frac{\partial M}{\partial v}$.
For convenience let us define the radius $R_v = R^n |B_1|$. We will prove that $M$ satisfies the following nonlinear aggregation-diffusion equation in the viscosity sense
\begin{equation}\label{eq:mass equation bounded domain}
	\frac{\partial M}{\partial t} = \kappa (v)^2 \frac{\partial }{\partial v}  \left( \frac{\partial M}{\partial v} \right)^m + \kappa (v)^2 \frac{\partial M}{\partial v} \frac{\partial}{\partial v}  \left(V + W*\frac{\partial M}{\partial v} \right)   ,
\end{equation}
where $\kappa (v) = n \omega_n^{\frac{1}{n}} v^{\frac{n-1}{n}}$. Equation \eqref{eq:mass equation bounded domain} is a Hamilton-Jacobi type problem amenable to be treated by means of the notion of viscosity solution \cite{CIL92}. 
Notice that in \eqref{eq:mass equation bounded domain} the diffusion is non-linear of $p$-Laplace type with $p = m+1$. We will demonstrate that the notion of viscosity solution is the natural way to study the formation of a Dirac delta comprising part of the mass of the initial data. Linking the fact that \eqref{ec:The_equation} has a gradient flow structure with the notion of viscosity solution in mass variable will be crucial in order to obtain our main result of partial mass concentration for long times for generic initial data.

In the radial setting, the formation of a Dirac delta at $0$ is equivalent to the loss of the Dirichlet boundary condition $M(t,0)=0$. There are very few results of loss of the Dirichlet boundary condition in finite or infinite time in the literature of parabolic equations. One example of this is the previous work \cite{CGV22}, in which with similar techniques the authors prove concentration in infinite time. There are some other examples like \cite{ARS04, SV06}, where they study equations of the type $u_t = u_{xx} + |u_x|^p$, or \cite{BdL04, PS17, PS20, MS20} for a family of viscous Hamilton-Jacobi equations. 

\section{Main results}

The aim of this section is to present the main results of this paper, the key ideas of the proof, and the structure of the rest of the paper. In the sequel, we denote the dependence on time of $\rho$ with a subindex, referring to it like $\rho_t$. Notice that \eqref{ec:The_equation} is in divergence form and one expects conservation of mass under suitable assumptions. We will carry on the dependence of the mass from the initial data instead of normalizing it to one as it is customary in the gradient flow literature. 

\subsection{The case of the ball of radius \texorpdfstring{$R$}{R}} The overall picture we have of the problem is better when we focus on the problem posed in a ball $B_R$. 
\begin{subequations}
\label{eq:Fast-Diffusion_Problem_BR}
We take a more general kernel $K$ with $K(x,y)=K(y,x)$ and consider the aggregation-diffusion problem
\begin{align}
    		&\frac{\partial \rho}{\partial t} = \Delta \rho^m + \nabla \cdot (\rho \nabla V) +  \nabla \cdot \left(\rho \nabla \int_{B_R} K(\cdot ,y) \rho(y) \, dy \right) \qquad  \text{in } \, (0, \infty ) \times B_R, \\
      		&\rho (0, x) = \rho_0 (x), \qquad  \text{for } x \in B_R,
\end{align}
with no-flux condition on the boundary of the ball
\begin{equation}
\left( \nabla  \rho^m + \rho \nabla V + \rho \nabla \int_{B_R} K(\cdot ,y) \rho(y) \, dy \right) \cdot \nu (x) = 0
        \qquad \text{on } \, (0, \infty ) \times \partial B_R\,.
\end{equation}
\end{subequations}
As a convenient assumption, we require that $V$ does not produce flux across the boundary
\begin{equation}\label{eq:V has no flux}
    \nabla V (x) \cdot \nu (x) = 0, \quad \text{on } \partial B_R.
\end{equation}
In order to pass to the limit as $R \to \infty$ we are particular interested in kernels of the form 
\begin{equation*}
	K(x,y) = \eta(x) W(x-y) \eta (y),
\end{equation*}
where $\eta$ is any function in $C_c^{\infty} (B_R)$. In this case, $K$ does not produce any flux across the boundary. This problem is the formal $2$-Wasserstein gradient-flow of the free energy
\begin{equation}\label{ec:Free_energy_FR}
	\mathcal{F}_R [\rho] = \frac{1}{m-1} \int_{B_R} \rho (x)^m \, dx + \int_{B_R} V(x) \rho (x) \, dx + \frac{1}{2} \int_{B_R}\int_{B_R} K(x,y) \rho (y) \rho(x) \, dy \, dx.
\end{equation}
The corresponding equation for the mass is
\begin{equation}\label{eq:FDE_Mass_equation}
	\frac{\partial M}{\partial t} = \kappa (v)^2 \frac{\partial }{\partial v}  \left( \frac{\partial M}{\partial v} \right)^m + \kappa (v)^2 \frac{\partial M}{\partial v} \diffEv [\rho], \quad \kappa (v) = n \omega_n^{\frac{1}{n}} v^{\frac{n-1}{n}}, \quad \text{in } (0, \infty) \times (0,R_v),
\end{equation}
where $R_v = |B_R|$, and slightly abusing the notation, we define
\begin{align*}
\begin{split}
		\diffEv[
		\mu] = \frac{\partial V}{\partial v} +   \frac{\partial}{\partial v} \left( \int_{B_R} K( \cdot , y) \diff \mu (y) \right).
		\end{split}
	\end{align*}
The precise statement will be given below in subsection \ref{sec:Mass equation regularized problem}.

We devote \cref{sec:BR} to the well-posedness theory. 
The aim is to prove a well-posedness result for \eqref{eq:Fast-Diffusion_Problem_BR} for the following notion of solution.
\begin{definition}\label{def:Strong solutions V K}
	We say that $\rho$ is a weak $L^1$ solution of the problem \eqref{eq:Fast-Diffusion_Problem_BR}  in $(0,T) \times B_R$ if $\rho \in C([0,T]; L^1(B_R)), \rho^m \in L^1(0,T; W^{1,1} (B_R))$ and
	\begin{equation*}
			\int_{B_R} \rho_0 \varphi (0) + \int_0^T \int_{B_R} \rho_t \frac{\partial \varphi}{\partial t} = \int_0^T \int_{B_R} \nabla \rho^m \nabla  \varphi + \int_0^T \int_{B_R} \rho_t \left( \nabla V + \nabla \int_{B_R} K(\cdot,y) \rho(y) \diff y \right)  \nabla \varphi.
	\end{equation*}
    for all $\varphi \in C^\infty( [0,T] \times B_R) $ with $\varphi(T,x) = 0$.
	Furthermore, we say that it is a 
	strong solution  if it is a weak solution and
	\begin{enumerate}
	    \item $\rho \in L^{\infty} \left( (0, T) \times  B_R \right)$,
		\item $\rho^m \in L^2 (0, T ; H^{2} (B_R) )$,
		\item $\frac{\partial \rho}{\partial t} \in L^2 ((0, T) \times  B_R )$. 
	\end{enumerate}
\end{definition}
The well-posedness result is the following.
\begin{theorem}
	\label{th:Banach_fixed_point}
	Assume the basic hypotheses  on the initial data $\rho_0$ and the potentials $V$ and $K$
	\begin{align}\label{hyp:Basic BR}
		\tag{H0}
		\begin{split}
			& \rho_0 \in L^{\infty} (B_R), \\
			& V \in W^{2, \infty} (B_R), \, V \geq 0, \,  \nabla V (x) \cdot \nu (x) = 0 \, \text{ on } \partial B_R,\\
			& K \in W^{2,\infty} (B_R \times B_R), \, K(x,y)=K(y,x) \ge 0 \text{ for all } x,y \in B_R, \\
   &\qquad \nabla_x K (x, y) \cdot \nu (x) =0 \text{ for all } (x,y) \in (\partial B_R \times B_R).
		\end{split}
	\end{align}
	 Then, there exists a unique strong solution for the problem \eqref{eq:Fast-Diffusion_Problem_BR}
	 for all $T > 0$.
\end{theorem}
The proof is constructive, in subsection \ref{sec:The_equation_in_BR} we replace the fast diffusion by a uniformly elliptic operator and fixed time-dependent drift. Using the classical theory \cite{LSU68} we prove well-posedness of classical solutions for this problem. In subsection \ref{sec:FD+know_transport} we replace the uniformly elliptic operator by the fast diffusion and with compactness arguments we prove the existence of strong solutions. Uniqueness follows from an $L^1$ comparison principle. Finally, using a fixed-point argument on the drift we show well-posedness for the problem \eqref{eq:Fast-Diffusion_Problem_BR}. The proof can be found in subsection \ref{sec:FD+V+K} using the previous sections.

Once we have shown well-posedness for the problem, we prove the so-called free-energy dissipation formula.
\begin{proposition}\label{prop:Convergence_of_the_free_energies}
	Assume the basic hypotheses  \eqref{hyp:Basic BR} on the initial data $\rho_0$ and the potentials $V$ and $K$, then strong solutions $\rho$ of the problem \eqref{eq:Fast-Diffusion_Problem_BR} satisfy the identity
	\begin{equation}\label{eq:flow of the free energy}
		\int_{t_1}^{t_2} \int_{B_R} \rho \left| \nabla \left( U'(\rho) + V + \int_{B_R} K( \cdot , y ) \rho (y) \, dy \right) \right|^2 \, dt = \mathcal{F}_R [\rho_{t_1}] - \mathcal{F}_R [\rho_{t_2}].
	\end{equation}
	In particular, $\mathcal{F}_R [\rho_t]$ is a non-increasing function of time.
\end{proposition}
We prove this result in subsection \ref{sec:Free_energy}. We use this estimate to prove long-time asymptotics up to a subsequence in the following sense. Let us take $\left\lbrace t_n \right\rbrace$ a sequence of times such that $t_n \rightarrow \infty$. Let $\rho$ be a weak solution of the problem \eqref{eq:Fast-Diffusion_Problem_BR}, so we define the sequence of functions 
\begin{align}\label{seq:time_steps}
	\begin{split}
		\rho^{[n]}\colon& [0,1] \times B_R  \longrightarrow  \mathbb{R} \\
		& \qquad (t,x) \quad \mapsto \, \, \rho (t + t_n, x).
	\end{split}
\end{align}
In \cref{sec:Long time asymptotic} we prove the following:
\begin{theorem}\label{th:Convergence_stationary_state}
	Assume the basic hypotheses  \eqref{hyp:Basic BR} on the initial data $\rho_0$ and the potentials $V$ and $K$. Let be the sequence $\rho^{[n]}$ be defined as in \eqref{seq:time_steps}. Then, there exists $\widehat \mu \in W^{-1,1} (B_R)$ such that, up to a subsequence, $\rho^{[n]} \rightarrow \widehat \mu$ in $C([0,1]; W^{-1,1} (B_R))$. 
\end{theorem}

A key difficult point is to show that the limit $\widehat \mu$ in the previous theorem does not depend on time. We we will show later that it is a solution of the stationary problem in the sense of mass. 
In  \cite{CGV22} the authors can only show that $\widehat \mu$ does not depend on time for a specific family of explicit initial data (which are singular at $0$), for which the mass function point-wise increasing with respect to $t$. 
We can make the characterisation for general initial data in $L^1 \cap L^\infty$ with radial symmetry.

Our method relies on taking advantage of the negative Sobolev space $W^{-1,1}(B_R)$ defined as the distributions with finite norm
\begin{equation}\label{def:norm W -1 1}
    \| \mu \|_{W^{-1,1}(B_R)} := \inf_{\mu = \diver (F)} \| F \|_{L^1(B_R)}.
\end{equation}
Notice that elements in $W^{-1,1} (B_R)$ can be tested against $W^{1,\infty}_0 (B_R)$ and hence $\mathcal{M}(B_R) \subset W^{-1,1} (B_R)$ with compact embedding.
The key reason to use this space is the following \textit{a priori} estimate
\begin{align}
		\int_0^1  \left\| \frac{\partial \rho_t^{[n]}}{\partial t} \right\|^2_{W^{-1, 1}(B_R)} \hspace{-2ex} 
		&= \int_0^1 \left\|  \diver \left( \rho_t^{[n]} \nabla \left( \frac{m}{m-1} \left( \rho_t^{[n]} \right)^{m-1} +  V + \int_{B_R} K( \cdot, y) \rho_t^{[n]} (y) \, dy   \right)  \right)  \right\|^2_{W^{-1, 1}(B_R)} \nonumber \\
		& \le \int_0^1 \left\|  \rho_t^{[n]} \nabla \left( \frac{m}{m-1} \left( \rho_t^{[n]} \right)^{m-1} +  V + \int_{B_R} K( \cdot, y) \rho_t^{[n]} (y) \, dy  \right)   \right\|^2_{L^{1}(B_R)} \nonumber\\
		& \leq \int_{t_n}^{t_n +1} \left\|  \rho_t^{\frac{1}{2}}  \right\|^2_{L^{2}(B_R)} \left\|  \rho_t^{\frac{1}{2}} \nabla \left( \frac{m}{m-1} \rho_t^{m-1} +  V +  \int_{B_R} K ( \cdot , y)  \rho_t (y) \, dy   \right)   \right\|^2_{L^{2}(B_R)} \nonumber\\
		& \leq \| \rho_0 \|_{L^{1}(B_R)} \left( \mathcal{F}_R [\rho_{t_n}] - \mathcal{F}_R [ \rho_{t_{n}+1} ] \right),
        \label{ec:inequality_about_drho/dt_W-1,1}
\end{align}
which involves only the energy. Hence, it follows that
\begin{align}\label{ec:asymptotics_equicontinuity}
	\begin{split}
		\| \rho_t^{[n]} - \rho_s^{[n]} \|_{W^{-1,1}(B_R)} 
		\leq \int_s^t \left\| \frac{\partial \rho_{\sigma}^{[n]}}{\partial \sigma} \right\|_{W^{-1,1} (B_R)} \!\!\!\!\! d \sigma 
		& \leq \| \rho_0 \|^{\frac{1}{2}}_{L^1(B_R)} \left( \mathcal{F}_R [\rho_{t_n}] - \mathcal{F}_R [ \rho_{t_{n}+1} ] \right)^{\frac{1}{2}} |t-s|^{\frac{1}{2}}.
	\end{split}
\end{align}
We have shown that $\mathcal F_R[\rho_t]$ is non-increasing in time, and we will also show that the free energy is bounded below. Hence, there is a limit as $t \to \infty$ for $\mathcal F_R[\rho_t] \to L$.

We prove \Cref{th:Convergence_stationary_state} using the estimate \eqref{ec:asymptotics_equicontinuity} as follows. To prove the convergence claim of a subsequence of $\rho^{[n]}$ we use the Ascoli-Arzelá theorem.   
Passing to the limit in \eqref{ec:asymptotics_equicontinuity} we also show that $\| \widehat \mu_t - \widehat \mu_s \|_{W^{-1,1} (B_R)} = 0$, so $\widehat \mu$ does not depend on time. The details of the proof are given in subsection \ref{sec:Convergence_to_stationary_state}.

\begin{remark}\label{rem:Benamou-Brenier}
The Benamou-Brenier formula \cite{San15} and the $2$-Wasserstein gradient flow structure of the problem can be used to prove that the convergence of the full sequence $\rho^{\infty, 1, [n]}$ also holds in $C([0,1]; \mathcal{P} (B_R))$ endowed with the $2$-Wasserstein distance. We will not include details of this proof since the asymptotic result is essentially analogous.
\end{remark}

We present the mass equation and some of its properties in subsection \ref{sec:Mass equation regularized problem}. Once we have presented the problem, in subsection \ref{sec:viscosity solutions} we prove through the construction of $\rho$ that the mass $M$ is a viscosity solution of \eqref{eq:FDE_Mass_equation} for any finite time in the precise sense given in \Cref{def:mass viscosity BR}.

After that we will study the asymptotic in time of $M$ and link it to the asymptotic in time of the solution of the problem \eqref{eq:Fast-Diffusion_Problem_BR}.
In the same sense as for $\rho$  in \eqref{seq:time_steps} we define the sequence in $[0,1] \times [0,{R_v}]$ given by
\begin{equation}\label{ec:Mass_tn_tn+1}
	M^{[n]} (t, v )= \int_{\widetilde{B_v}} \rho^{[n]}_t (x) \, dx,
\end{equation}
which are the solutions of the mass equation \eqref{eq:General_Mass_equation} for the time interval $[t_n, t_n +1]$ after a translation to $[0,1]$.

 \begin{theorem}\label{prop:M^[n]_converge_uniformly}
 	Assume the basic hypotheses \eqref{hyp:Basic BR}, and that the initial datum $\rho_0$ and potentials $V$ and $K$ are radially symmetric, and let the sequence $M^{[n]}$ be defined as in \eqref{ec:Mass_tn_tn+1}. Then, there exists $\widehat{M} \in C((0, R_v])$ such that, if we define $\widehat{M}(0)=0$, then the following properties hold:
 	\begin{itemize}
 		\item The function $\widehat{M}$ is non-decreasing in $v$,
 		\item up to a subsequence, $M^{[n]} \rightarrow \widehat{M}$ in $ C_{loc}([0,1] \times (0 , R_v])$ and point-wise in $[0,1] \times [0,R_v]$,
 		\item $\frac{\partial \widehat{M}}{\partial v} =\widehat{\mu}$ belongs to $\mathcal M ([0,R_v])$ in the precise sense given in \Cref{rem:dM/dv identify with mu},
 		\item $\widehat{M}$ is a stationary viscosity solution of \eqref{eq:mass equation bounded domain} in the sense given in \Cref{def:mass viscosity BR}.
 	\end{itemize}
 \end{theorem}
 
 We present its proof in subsection \ref{sec:Convergence of the mass}.
 
 \begin{remark}
 	Notice that the convergence happens point-wise in $[0,R_v]$ so that, due to the no-flux condition on $\rho$, we have that $\widehat{M}(R_v) = M(t,R_v) = \| \rho_0 \|_{L^1 (B_R)}$.
 \end{remark}

\begin{remark}
	Notice that if $\widehat{M} \in W^{1,\infty}_{loc} ( (0,R_v] )$ then there exists $\widehat \rho \in L^1(B_R)$ such that
	\begin{equation*}
		\widehat{\mu} = \alpha \delta_0 + \widehat{\rho} \, dx,
	\end{equation*}
	with $\widehat{\rho}$ integrable and $\alpha = \| \rho_0 \|_{L^1 (B_R)} - \| \widehat{\rho} \|_{L^1 (B_R)} \geq 0$.
\end{remark}

There are two examples where we can prove this regularity. We devote subsection \ref{sec:Regularity M} to this goal. The first example is 
\begin{proposition}
    \label{th:Concentration_Phenomena_ball_viscosity_solutions}
    Assume the basic hypotheses \eqref{hyp:Basic BR}, and that the initial datum $\rho_0$ and potentials $V$ and $K$ are radially symmetric. Let the sequence $M^{[n]}$ be defined as in \eqref{ec:Mass_tn_tn+1} and $\widehat{M}$ its limit obtained in \Cref{prop:M^[n]_converge_uniformly}. 
    Assume, furthermore, that 
    $$\left. \frac{\partial }{\partial v} \left( V(x) + \int_{B_R} K(x , y )  \diff \widehat \rho (y) \right)\right|_{   v = |B_1| |x|^d  }   \geq 0.$$
    Then, $\widehat{M}$ is $C^2 ((0,R_v))$.
\end{proposition}
We proof this result in subsection \ref{sec:Regularity M 1} and follows directly from the regularity theory of Caffarelli and Cabré \cite{CC95}. 
A much more difficult result is the following
\begin{theorem}\label{th:Regularity_of_Minfty}
    Assume  the basic hypotheses \eqref{hyp:Basic BR}, that the initial datum $\rho_0$ and potentials $V$ and $K$ are radially symmetric,  and furthermore, that  $V$ and $K$ have compact support. 
    Then, $\widehat M $ is linear in an interval $[R_v - b,R_v]$ for some $b > 0$, $\frac{\partial \widehat{M}}{\partial v} (R_v ) > 0$, and for a.e. $v \in (0,R_v)$
    \begin{equation}
        \label{eq:mass ball stationary characterisation}
        \frac{\partial \widehat{M}}{\partial v} (v) 
        =  
        \left(   
        \left( \frac{\partial \widehat{M}}{\partial v} (R_v ) \right)^{m-1} 
        +\tfrac{1-m}{m}
        \left. \left( V(x) + \int_{B_R} K(x,y) \diff \widehat \rho (y) \right) \right|_{   v = |B_1| |x|^d  }
        \right)^{-\frac 1 {1-m} }
    \end{equation}
    and, in particular, $\widehat{M} \in W^{2, \infty}_{loc} ((0 , R_v])$.
\end{theorem}
The proof can be found in subsection \ref{subsec:Regularity_Minfty} and uses inf- and sup-convolutions. Hence, we have characterised $\widehat \rho$, the absolutely continuous part of the steady state $\widehat \mu$, as
$$
	\widehat \rho (x) =  
	\left(   
	h
	+\tfrac{1-m}{m}
	 \left( V(x) + \int_{B_R} K(x,y) \diff \widehat \rho (y) \right) 
	\right)^{-\frac 1 {1-m} }.
$$
for some $h > 0$.

\subsection{The problem in \texorpdfstring{$\mathbb{R}^d$}{Rd}}
Under certain assumptions, we extend the results we obtain in the ball $B_R$ to the whole space $\Rd$ to study the problem 
\eqref{ec:The_equation}. In order to do this, we define
\begin{equation*}
    \mathcal{F}_{R, \eta} [\rho] := \frac{1}{m-1} \int_{B_R} \rho (x)^m \, dx + \int_{B_R} V_R(x) \rho(x) \, dx + \frac{1}{2} \int_{B_R} \int_{B_R} \eta (x) W(x-y) \eta (y) \rho (y) \rho(x) \, dy \, dx,
\end{equation*}
the free energy associated to the problem \eqref{eq:Fast-Diffusion_Problem_BR} when $K(x,y)=\eta(x) W(x-y) \eta (y)$. Whenever we want to refer to the free energy associated to the problem in the whole space, we will write $\mathcal{F}_{\infty , \eta}$. We keep this notation for the free energy in \cref{sec:Rd}.

Notice that $\mathcal F_{R,\eta}$ is a generalisation of the free energy studied in \cite{CGV22}, which corresponds to $W = 0$ (i.e.
$
    \mathcal{F}_{\infty, 0} 
$)
whereas the free energy for \eqref{ec:The_equation} corresponds to $\mathcal F_{\infty,1}$.

First we extend the theory from \cref{sec:BR} and \cref{sec:Long time asymptotic} for $\rho$. 
For the treatment of $\Rd$ we rely heavily on the control of the free energy. For this we assume that, even when $W = 0$ the free energy is bounded below
\begin{equation}
	\tag{H1}
    \label{ec:Free_energy_is_bounded_below}
		\underline{\mathcal{F}_{0}} := \inf_{\rho \in \mathcal{P}_{ac} (\mathbb{R}^d )} \mathcal{F}_{\infty, 0} [\rho] > - \infty,
\end{equation}
This allows us to control the second-order moment. 
We also require the control of a higher order moment, and we are only able to work in the range
\begin{equation}
	\tag{H2}
    \label{eq:sharp m extension}
		m > \frac{1+d- \sqrt{2d +1}}{d}.
\end{equation}
\begin{remark}\label{rem:sharp m}
    Assumption \eqref{eq:sharp m extension} is the sharp value of $m$ for the Carlson-Levin inequality \cite[Lemma 5]{CDDFH19} for the $p$-moment $p=\frac{2}{1-m}$. Carlson type inequalities appear in the literature after \cite{Car34} and its sharp form is established by Levin in \cite{Lev48}. We need to take into account the sharp value of $m$ at the Step 4a of the proof of \Cref{th:locally_strong_solution_Rn} in order to control the tail of the diffusive part of the free energy.
\end{remark}

We require that the initial datum has controlled moment of certain order higher than $2$
\begin{equation}
	\tag{H3}
    \label{eq:p moment bdd for rho0}
	\int_{\Rd} |x|^{\frac{2}{1-m}} \rho_0  < \infty
\end{equation}
We will also assume that the growth of $V$ and $W$ is no faster than the quadratic in the sense that
	\begin{gather}
		\tag{H4}
        \label{eq:V_quadratic}
		|\nabla V (x)| \leq C(1+|x|) \qquad \text{for all } x \in \mathbb{R}^d,
	\\
		\tag{H5}
		\label{eq:W_quadratic}
		|\nabla W (x)| \leq C(1+|x|) \qquad \text{for all } x \in \mathbb{R}^d,
    \\
    	\tag{H6}
    	\label{eq:Delta V and W bounded}
    \Delta V, \Delta W \in L^{\infty} (\mathbb{R}^d ).
	\end{gather}
Let us denote by $\rho^{R, \eta}$ the unique strong solution for the problem \eqref{eq:Fast-Diffusion_Problem_BR} in the bounded domain $(0,T) \times B_R$ for every $R$ and every kernel $K$ of the form, $K(x,y) = \eta(x) W(x-y) \eta (y)$ with $\eta\in C_c^{\infty} (\mathbb{R}^d)$ a cut-off function. The solution $\rho^{R, \eta }$ is such that
\begin{equation*}
	\rho^{R, \eta} \in L^{\infty} (0,T; H^1 (B_R)).
\end{equation*}
We construct the extension by zero of $\rho^{R, \eta}$ to the whole space $\mathbb{R}^d$
$
	\rho^{R, \eta, \ast} = \rho^{R, \eta} \chi_{B_R}.
$

In order to pass to the limit we build a sequence $\eta_j$ as follows. We take a smooth radially-symmetric and non-increasing function $\eta_1$ such that
\begin{equation*}
	\eta_1 (x) 
	= 
	\begin{dcases}
		1 & \text{if }|x| \leq \frac{1}{2}, \\
		0 & \text{if }|x| \ge 1,
	\end{dcases}
\end{equation*}
and define $\eta_j (x) = \eta_1 ( x/j )$. In this way, $\eta_j$ is such that $\eta_j \nearrow 1$ when $j \rightarrow \infty$. We are able to prove the following result.
\begin{theorem}\label{th:locally_strong_solution_Rn}
	Assume $\rho_0 \in L^1 ( \mathbb{R}^d ) \cap L^{\infty} (\mathbb{R}^d)$, and the following technical assumptions \eqref{ec:Free_energy_is_bounded_below}-\eqref{eq:Delta V and W bounded}. First, for $j$ fixed we prove there exists a sequence $R_{ij} \to \infty$ as $i \to \infty$ such that
	\begin{align}\label{eq:Local convergence from BR to Rd}
		\rho^{R_{ij} , \eta_j, \ast} & \rightarrow \rho^{\infty , \eta_j} \quad \text{in } C_{loc}([0,\infty ); L^2_{loc}(\mathbb{R}^d)) \text{ as } i \to \infty,\\
		\label{eq:Convergence free energy ball to Rd}
		\mathcal{F}_{R_{ij}, \eta_j}[\rho^{R_{ij}, \eta_j}_t] & \rightarrow \mathcal{F}_{\infty, \eta_j} [\rho^{\infty, \eta_j}_t] \quad \text{for every } t \in [0,\infty)  \text{ as } i \to \infty, 
	\end{align}
	and $\rho^{\infty, \eta_j}$ is a locally-strong solution of 
	\begin{equation*}
		\frac{\partial \rho}{\partial t} = \Delta \rho^m + \nabla \cdot (\rho \nabla V) +  \nabla \cdot \left(\rho \nabla \int_{B_R} K_{\eta_j}(\cdot ,y) \rho(y) \, dy \right) \qquad \mathrm{in } \, (0, \infty ) \times \Rd
	\end{equation*}
	where $K_{\eta_j}(x,y) = \eta_j(x) W(x-y) \eta_j(y)$
	in the sense of \Cref{def:locally-strong solution}. Then, as $j \to \infty$ we have that
	\begin{align}\label{eq:Local convergence from BR to Rd 2}
		\rho^{\infty , \eta_j} &\rightarrow \rho^{\infty , 1} \quad \text{in } C_{loc}([0,\infty ); L^2_{loc}(\mathbb{R}^d)) \text{ as } j \to \infty,\\
		\label{eq:Convergence free energy ball to Rd 2}
		\mathcal{F}_{\infty, \eta_j}[\rho^{\infty, \eta_j}_t] &\rightarrow \mathcal{F}_{\infty, 1} [\rho^{\infty, 1}_t] \quad \text{for every } t \in [0,\infty)  \text{ as } j \to \infty, 
	\end{align}
	and $\rho^{\infty, 1}$ is a locally-strong solution of the problem \eqref{ec:The_equation} in the sense of \Cref{def:locally-strong solution}.
\end{theorem}
We prove the result in subsection \ref{sec:short-time well-posedness Rd}. In \Cref{prop:extend mass solution to Rd} we show that these short-time limits also hold for $M$ in the sense of viscosity solutions. We require the additional assumptions on the tails
\begin{equation}
	 \tag{H7}
	 \label{eq:Growth of V compare to W}
	\lim_{\sigma \to \infty} \sup_{|x| > \sigma} \frac{|\nabla W(x)|}{V(x)} = 0.
\end{equation}
and the following uniformity on the tails of $V$: for every $K \Subset \Rd$ we have that
\begin{equation}
	\tag{H8}
	\label{eq:Control of the tails} 
	C(K) = \sup_{\substack{x \in K \\ y \in \Rd}} \frac{V(y-x)}{1 + V(y)} < \infty .
\end{equation}
For the long-time asymptotics we can prove that
	\begin{theorem}\label{th:mu infty Rd}
	Assume all the hypothesis from \Cref{th:locally_strong_solution_Rn}. Then, there exists $\widehat{\mu} \in W^{-1,1}_{loc} (\mathbb{R}^d)$ and a subsequence such that 
	$$
	\rho^{\infty, 1, [n]} \to \widehat\mu \quad  \text{in } C([0,1];W^{-1,1}_{loc}(\mathbb{R}^d)) \qquad \text{and} \qquad \rho^{\infty , 1, [n]} \rightharpoonup \widehat\mu \quad \text{weak}-\ast \text{ in } L^{\infty}(0,1; \mathcal{M}(\mathbb{R}^d)).
	$$ 
	Furthermore, if we also assume \eqref{ec:Free_energy_bounded_below} and that $\inf_{x \in \Rd \backslash B_{\sigma}} V(x) \rightarrow \infty$ when $\sigma \rightarrow \infty$, $\widehat\mu$ is such that $\widehat\mu (\Rd) = \| \rho_0 \|_{L^1 (\Rd)}$.
\end{theorem}
\begin{remark}
    An analogous improvement to  \Cref{rem:Benamou-Brenier} using the Benamou-Brenier formula and the $2$-Wasserstein gradient flow structure of the problem can also be done in $\R^d$ to show convergence of the full sequence in $C([0,1]; \mathcal{P}_2 (\R^d))$ endowed with the $2$-Wasserstein distance. 
\end{remark}

We prove the result in subsection \ref{sec:Long time asymptotics Rd}. In \Cref{prop:limit Mn properties} we prove that the long-time asymptotics also hold for the mass $M$ in the sense of viscosity solutions. 
Once we can link $\widehat\mu$ to $\widehat M$, the asymptotic in time of $M$, we are able to characterise $\widehat \mu$.
\begin{remark}\label{rem:Potential concentration Rd}
    If $\widehat M \in W^{1, \infty}_{loc} ((0, \infty))$ and the mass \textit{is not escaping through infinity} (i.e. $\|\rho_t\|_{L^1} = \|\rho_0\|_{L^1}$), then there exists $\widehat{\rho} \in L^1 (\mathbb{R}^d)$ such that
	\begin{equation*}
		\widehat\mu = \alpha \delta_0 + \widehat{\rho} \, dx,
	\end{equation*}
	with $\widehat{\rho}$ integrable and $\alpha = \| \rho_0 \|_{L^1(\mathbb{R}^d)} - \| \widehat{\rho} \|_{L^1 (\mathbb{R}^d)} \geq 0$.
\end{remark}
In $\Rd$ we are only able to make the extension of the simpler result \Cref{th:Concentration_Phenomena_ball_viscosity_solutions} from the $B_R$ case, which we state below as \Cref{prop:R=infty eta= 1 t=infty regularity in terms of flux}.

We conclude the paper with \cref{sec:Example of concentration} where we construct a family of examples of $V$ and $W$ where the presence of a Dirac delta at the origin for $\widehat \mu$ can be shown.
We provide two appendixes with technical results that are needed in the proofs. In \cref{ap:Maximum principle viscosity solutions} we present a maximum principle for viscosity solutions, and in \cref{ap:a priori estimates Rd} we show \textit{a priori} estimates for the $\Rd$ case.

\section{Well-posedness for the aggregation-diffusion equation in \texorpdfstring{$B_R$}{BR}}
\label{sec:BR}

We restrict ourselves to the bounded domain $B_R$, adding a no-flux boundary condition. Through this section we will work with a fixed time-dependent drift $E$. First, in subsection \ref{sec:The_equation_in_BR} we review the classical theory for uniformly elliptic diffusion $\Phi$ and we obtain some \textit{a priori} estimates that will be useful to us on different parts of this work. After that, in subsection \ref{sec:FD+know_transport} we consider the fast diffusion case, $\Phi (s) = s^m$, $0<m<1$. We obtain uniqueness by an $L^1$ continuous dependence argument and existence follows from a convergence argument, using the solutions from the previous subsection.

\subsection{Uniformly elliptic diffusion with time-dependent drift in \texorpdfstring{$B_R$}{BR}}\label{sec:The_equation_in_BR}

In this subsection we consider the problem
\begin{equation}\label{eq:General_elliptic_Problem}
	\left\lbrace  \begin{array}{lll}
		\frac{\partial \rho}{\partial t} = \Delta \Phi(\rho) + \nabla \cdot (\rho E) \quad & \mathrm{in} \, (0, \infty ) \times B_R  \\
		\left( \nabla  \Phi(\rho) + \rho E \right) \cdot \nu (x) = 0 & \mathrm{on} \, (0, \infty ) \times \partial B_R, \\
		\rho (0, x) = \rho_0 (x). &
	\end{array} \right.
\end{equation}
where $E = E_t ( x)$ is assumed to be smooth, more precisely $E \in C^{\infty} ([0, \infty) \times \overline {B_R} )$, and $\Phi \in C^1$ is uniformly elliptic, in the sense that there exist constants such that
\begin{equation}\label{ec:unif_ellipticity}
	0 < c_1 \leq \Phi ' (\rho ) \leq c_2 < \infty .
\end{equation}
Furthermore, we assume,
\begin{equation}\label{ec:extra_boundary_condition}
	E_t( x) \cdot \nu (x) = 0, \quad \text{on }   \partial B_R, \, \forall t \in (0, \infty ).
\end{equation}
As we explained in \cite{CGV22}, under these assumptions, existence, uniqueness and the maximum principle hold from the classical theory \cite{LSU68, GT01, Ama90, Yin05}. Classical solutions to \eqref{eq:General_elliptic_Problem} with initial data $\rho_0 \in C^2 (\overline{B_R})$ satisfy
\begin{equation*}
	\rho \in C^1 \left( (0,T); C (\overline{B_R}) \right) \cap C \left( (0, \infty) ; C^2 (\overline{B_R}) \right) \cap C \left( [0, \infty ) \times \overline{B_R} \right) 
\end{equation*}
as obtained in \cite{Ama90}.
Let us now discuss further properties of the solution of \eqref{eq:General_elliptic_Problem}. First, we will present some {\em a priori} estimates that were already introduced in \cite{CGV22}, where $E$ did not depend on time $t$. However, some results in \cite{CGV22} also hold for time-dependent $E$. In particular, we can state the following a priori estimates.

\begin{proposition}[$L^p$ estimates, \cite{CGV22}]
	Assume \eqref{ec:unif_ellipticity} and $E_t(x) \cdot \nu (x) \geq 0$. Then, the unique classical solution of \eqref{eq:General_elliptic_Problem} satisfies that
	\begin{equation}\label{ec:Lp_estimate}
		\| (\rho_t)_{\pm} \|_{L^p( B_R)} \leq e^{\frac{p-1}{p} \int_0^T \| \nabla \cdot E_t \|_{L^{\infty} (B_R)}} \|(\rho_0)_{\pm} \|_{L^p ( B_R)}.
	\end{equation}
\end{proposition}

\begin{remark}
	If we assume $\rho_0(x) \geq 0$, the maximum principle implies that $\rho_t (x) \geq 0$. Then, due to the no-flux boundary condition, for $p=1$ we also have
	\begin{equation}\label{ec:L1_norm_is_preserved}
		\| \rho_t \|_{L^1 (B_R)} = \| \rho_0 \|_{L^1(B_R)},
	\end{equation}
	because, along solutions of \eqref{eq:General_elliptic_Problem} we have
	\begin{equation*}
		\frac{d}{dt} \| \rho_t \|_{L^1 (B_R)} = \int_{B_R} \nabla \cdot ( \nabla \Phi (\rho_t) + \rho_t E_t) = \int_{\partial B_R} (\nabla \Phi (\rho_t) + \rho_t E_t ) \cdot \nu (x) = 0.
	\end{equation*}
\end{remark}

Next, in order to formulate the {\em a priori} estimates on $\nabla \Phi (\rho )$ and $\nabla \rho$ we define,
\begin{equation}\label{ec:Primitive_of_Phi}
	\Psi (s) =2 \int_0^s \Phi (\sigma ) \, d \sigma .
\end{equation}
\begin{lemma}[\textit{a priori} estimates on $\nabla \Phi (\rho )$, \cite{CGV22}]\label{lem:a_priori_estimates_Phi}
    Assume \eqref{ec:unif_ellipticity} and \eqref{ec:extra_boundary_condition}. Then, the unique classical solution of \eqref{eq:General_elliptic_Problem} satisfies that
	\begin{equation}\label{ec:a_priori_estimate_Phi}
		\int_0^T \int_{B_R} | \nabla \Phi (\rho_t) |^2 \leq  \int_{B_R} \Psi (\rho_0 ) + \int_0^T  \| E_t \|^2_{L^{\infty}(B_R)} \| \rho_t\|_{L^2(B_R)}^2 .
	\end{equation}
\end{lemma}

From this point we will include the proofs for the following \textit{a priori} estimates because the time dependence becomes relevant and, therefore, there are significant differences with respect to the results shown in \cite{CGV22}.

In order to obtain an \textit{a priori} estimate for $\nabla \rho$ we define an auxiliary function $G_{\Phi}$, given by the conditions
\begin{equation}\label{def:G_Phi}
	G_{\Phi}''(s) = \frac{1}{\Phi ' ( s )}, \qquad G_{\Phi}'(0) = G_{\Phi}(0) = 0.
\end{equation}
Due to \eqref{ec:unif_ellipticity} and the regularity of $\Phi$, $G$ is well defined and $C^2$.

\begin{lemma}[\textit{a priori} estimates on $\nabla \rho$]\label{cor:a_priori_estimates_rho}
	Assume \eqref{ec:unif_ellipticity} and \eqref{ec:extra_boundary_condition}. Then, the unique classical solution of \eqref{eq:General_elliptic_Problem} satisfies that
	\begin{equation}\label{ec:a_priori_estimate_rho}
		\int_{B_R} G_{\Phi}(\rho_T) + \frac 1 2 \int_0^T \int_{B_R} | \nabla \rho_t |^2 \leq  \int_{B_R} G_{\Phi}(\rho_0) + \frac{1}{2}  \int_0^T \left\| \frac{\rho_t}{\Phi ' (\rho_t )} \right\|_{L^{\infty} (B_R)}^2 \| E_t \|_{L^2 (B_R)}^2.
	\end{equation}
\end{lemma}
\begin{proof}
	We can compute that
	\begin{equation*}
		\frac{\partial}{\partial t} \int_{B_R}  G_{\Phi} (\rho ) = - \int_{B_R} \nabla G_{\Phi}' (\rho )  \cdot (\nabla \Phi (\rho )  + \rho E_t) = - \int_{B_R} | \nabla \rho |^2 - \int_{B_R} \frac{\rho}{\Phi ' (\rho )} \nabla \rho \cdot E_t .
	\end{equation*}
	Then, we have
	\begin{equation*}
		\int_{B_R} \frac{\partial}{\partial t} G_{\Phi} (\rho ) + \int_{B_R} | \nabla \rho |^2 \leq \left \| \frac{\rho}{\Phi ' (\rho )} \right\|_{L^{\infty} (B_R)} \| \nabla \rho \|_{L^2 (B_R)} \| E_t \|_{L^2 (B_R)}.
	\end{equation*}
	Applying Young's inequality we obtain
	\begin{equation*}
		\int_{B_R} \frac{\partial}{\partial t} G_{\Phi} (\rho ) + \frac{1}{2} \int_{B_R} | \nabla \rho |^2 \leq \frac{1}{2} \left \| \frac{\rho}{\Phi ' (\rho )} \right \|_{L^{\infty} (B_R)}^2 \| E_t \|_{L^2 (B_R)}^2,
	\end{equation*}
	obtaining the desired result if we integrate in time.
\end{proof}

We apply a similar argument to obtain \textit{a priori} estimates for the time derivative.

\begin{lemma}[\textit{a priori} estimates on $\frac{\partial \rho}{\partial t}$]\label{lem:a_priori_d/dt_rho_t_is_L2}
	Assume \eqref{ec:unif_ellipticity}, \eqref{ec:extra_boundary_condition}, and $\Phi (\rho_0 ) \in H^1 (B_R)$. Then, the unique classical solution of \eqref{eq:General_elliptic_Problem} satisfies that
	\begin{align}\label{ec:a_priori_estimate_drho/dt_L2}
		\begin{split}
			\frac{1}{2}\int_0^T \int_{B_R} \Phi ' (\rho_t ) \left| \frac{\partial \rho_t}{\partial t} \right|^2 + \frac{1}{2} \int_{B_R} | \nabla \Phi (\rho_T) |^2 & \leq   \frac{1}{2} \int_{B_R} | \nabla \Phi (\rho_0 ) |^2  + \frac{1}{2}\int_0^T \int_{B_R} \Phi ' (\rho_t ) | \nabla \rho_t |^2 | E_t |^2 \\
			& \quad + \int_0^T \int_{B_R} \rho_t \nabla \Phi (\rho_t) \cdot \frac{\partial E_t}{\partial t} \\
			& \quad - \int_{B_R} \zeta( \rho_0 ) \nabla \cdot  E_0  + \int_{B_R} \zeta (\rho_T) \nabla  \cdot  E_T ,
		\end{split}
	\end{align}
	where $\zeta: [0, \infty ) \rightarrow \mathbb{R}$ is a function such that $ \zeta' (s) = s  \Phi' (s)$.
\end{lemma}
\begin{proof}
	We will use the notation $w = \Phi (\rho )$. When $\rho$ is smooth, we can take $\frac{\partial w}{\partial t}$ as a test function and integrate in $B_R$. Notice that $\frac{\partial w_t}{\partial t} = \Phi ' (\rho_t ) \frac{\partial \rho_t}{\partial t}$, so
	\begin{equation*}
		\int_{B_R} \Phi ' (\rho_t ) \left| \frac{\partial \rho_t}{\partial t} \right|^2 = \int_{B_R} \frac{\partial w_t}{\partial t}  \cdot \left(\Delta w_t  +  \nabla \cdot ( \rho_t E_t) \right).
	\end{equation*}
	We can integrate by parts to recover
	\begin{equation*}
		\int_{B_R} \Phi ' (\rho_t ) \left| \frac{\partial \rho_t}{\partial t} \right|^2 = - \frac 1 2 \frac{d}{dt} \int_{B_R} | \nabla w_t |^2  - \int_{B_R} \rho_t \nabla \frac{\partial w_t}{\partial t} \cdot E_t.
	\end{equation*}
	Integrating in $[0, T ]$ we have 
	\begin{equation*}
		\int_0^T \int_{B_R} \Phi' (\rho_t ) \left| \frac{\partial \rho_t}{\partial t} \right|^2 + \frac 1 2\int_{B_R} | \nabla w_T |^2 = \frac 1 2\int_{B_R} | \nabla w_0 |^2 - \int_0^T \int_{B_R} \rho_t \nabla \left( \frac{\partial w_t}{\partial t} \right) \cdot E_t.
	\end{equation*}
	Integrating by parts in time the second integral on the RHS,
	\begin{align*}
		\int_0^T \int_{B_R} \Phi' (\rho_t ) \left| \frac{\partial \rho_t}{\partial t} \right|^2 + \frac 1 2\int_{B_R} | \nabla w_T |^2 & =  \frac 1 2\int_{B_R} | \nabla w_0 |^2 + \int_0^T \int_{B_R}  \nabla w_t \cdot \left[ \frac{\partial \rho_t}{\partial t} E_t + \rho_t \frac{\partial E_t}{\partial t} \right]  \\
		& \quad + \int_{B_R} \rho_0 \nabla w_0 \cdot E_0 - \int_{B_R} \rho_T \nabla w_T \cdot E_T .
	\end{align*}
	At this point we introduce $\zeta$. Then, integrating by parts we get
	\begin{align*}
		\int_0^T \int_{B_R} \Phi' (\rho_t ) \left| \frac{\partial \rho_t}{\partial t} \right|^2 + \frac 1 2\int_{B_R} | \nabla w_T |^2 & =  \frac 1 2\int_{B_R} | \nabla w_0 |^2 + \int_0^T \int_{B_R}  \nabla w_t \cdot \left[ \frac{\partial \rho_t}{\partial t} E_t + \rho_t \frac{\partial E_t}{\partial t} \right]  \\
		& \quad - \int_{B_R} \zeta( \rho_0 ) \nabla \cdot  E_0  + \int_{B_R} \zeta (\rho_T) \nabla  \cdot  E_T  .
	\end{align*}
	Notice that,
	\begin{equation*}
		\frac{\partial \rho_t}{\partial t} \nabla w_t = \Phi ' (\rho_t )^{\frac{1}{2}} \frac{\partial \rho_t}{\partial t} \Phi ' (\rho_t )^{\frac{1}{2}} \nabla \rho_t.
	\end{equation*}
	Therefore, applying Young's inequality, we recover,
	\begin{align*}
		\frac{1}{2} \int_0^T \int_{B_R} \Phi' (\rho_t ) \left| \frac{\partial \rho_t}{\partial t} \right|^2 + \frac{1}{2} \int_{B_R} | \nabla w_T |^2 \leq & \, \frac 1 2 \int_{B_R} | \nabla w_0 |^2  + \frac{1}{2} \int_0^T \int_{B_R} \Phi' (\rho_t) | \nabla \rho_t |^2 |E_t|^2 \\
		& + \int_0^T \int_{B_R} \rho_t \nabla w_t \cdot \frac{\partial E_t}{\partial t} \\
		&  - \int_{B_R} \zeta( \rho_0 ) \nabla \cdot  E_0  + \int_{B_R} \zeta (\rho_T) \nabla  \cdot  E_T , 
	\end{align*}
	where all the terms on the RHS are bounded. This is because $E \in W^{1, \infty} ([0,T] \times \overline{B_R})$, $\Phi$ is uniformly elliptic, the $L^p$ estimate \eqref{ec:Lp_estimate}, the \textit{a priori} estimate \eqref{ec:a_priori_estimate_Phi} on $\| \nabla \Phi (\rho) \|_{L^2((0,T)\times B_R)}$ and the \textit{a priori} estimate \eqref{ec:a_priori_estimate_rho} on $\| \nabla \rho \|_{L^2((0,T) \times B_R)}$.
\end{proof}
\begin{remark}
	Let us observe that, in particular, because $\Phi'$ is non-negative, we have the inequality,
	\begin{align}\label{ec:drho/dt_L2_is_bounded}
		\begin{split}
			\frac{1}{2} \left(
			\min_{ \substack{ t \in (0,T) \\ x \in  B_R} } 
			\Phi' (\rho_t ) \right) \int_0^T \int_{B_R} \left| \frac{\partial \rho_t}{\partial t} \right|^2 & \leq \frac{1}{2} \int_{B_R} | \nabla \Phi (\rho_0 ) |^2 - \frac{1}{2} \int_{B_R} | \nabla \Phi (\rho_T ) |^2   \\
			& \quad + \frac{1}{2}\int_0^T \int_{B_R} \Phi ' (\rho_t ) | \nabla \rho_t |^2 | E_t |^2 + \int_0^T \int_{B_R} \rho_t \nabla \Phi (\rho_t) \cdot \frac{\partial E_t}{\partial t} \\
			& \quad  - \int_{B_R} \zeta( \rho_0 ) \nabla \cdot  E_0  + \int_{B_R} \zeta (\rho_T) \nabla  \cdot  E_T.
		\end{split}
	\end{align}
\end{remark}

\subsection{Fast Diffusion with time-dependent drift in \texorpdfstring{$B_R$}{BR}}\label{sec:FD+know_transport}

In this subsection, we will focus in the Fast Diffusion problem with known transport,
\begin{equation}\label{eq:FDE_Problem_+_known_transport}
	\left\lbrace  \begin{array}{lll}
		\frac{\partial \rho}{\partial t} = \Delta \rho^m + \nabla \cdot (\rho E)  \qquad & \mathrm{in} \, (0, \infty ) \times B_R  \\
		\left( \nabla  \rho^m + \rho   E \right) \cdot \nu (x) = 0 & \mathrm{on} \, (0, \infty ) \times \partial B_R, \\
		\rho (0, x) = \rho_0 (x). &
	\end{array} \right.
\end{equation}
For convenience in this section, we will denote 
$\Phi(s) = s^m.$
Let us define our notions of solutions for the problem \eqref{eq:FDE_Problem_+_known_transport}.

\begin{definition}[Weak solution]\label{def:distributional_solution_FD+known_transport}
	$\rho$ is said to be a weak solution of the problem \eqref{eq:FDE_Problem_+_known_transport} in $(0,T) \times B_R$ if it is $C([0,T]; L^1(B_R)), \Phi(\rho) \in L^1( 0,T; W^{1,1} (B_R))$ and for all $\varphi \in X = \left\lbrace \varphi \in C^{\infty} ([0,T] \times \overline{B_R}) \, : \, \varphi(T)=0 \right\rbrace$, it satisfies,
	\begin{equation*}
		\int_{B_R} \rho_0 \varphi (0) + \int_0^T \int_{B_R} \rho_t \frac{\partial \varphi}{\partial t} = \int_0^T \int_{B_R} \nabla \Phi( \rho ) \nabla \varphi + \int_0^T \int_{B_R} \rho_t E_t \nabla \varphi.
	\end{equation*}
\end{definition}

\begin{definition}[Strong solution]\label{def:strong_solution_FD+known_transport}
	$\rho$ is said to be a strong solution of the problem \eqref{eq:FDE_Problem_+_known_transport}  in $(0,T) \times B_R$ if it is a weak solution such that
	\begin{enumerate}
		\item $\Phi (\rho) \in L^1 (0, T ; H^{2} (B_R) )$; 
		\item $\frac{\partial \rho}{\partial t} \in L^2 ((0, T) \times  B_R )$.
	\end{enumerate}
\end{definition}
\begin{remark}
    Since $\frac{\partial \rho}{\partial t} \in L^2 ((0,T) \times B_R)$ we have that,
    \begin{equation*}
        \int_{B_R} | \rho_{s_2} - \rho_{s_1} | \leq \int_{s_1}^{s_2} \int_{B_R} \left| \frac{\partial \rho_s (x)}{\partial t} \right| \, dx \, ds \leq |s_2-s_1|^{\frac{1}{2}} \Big\| \frac{\partial \rho}{\partial t} \Big\|_{L^2((0,T) \times B_R)}
    \end{equation*}
    and $\rho$ is such that $\rho \in C^{\frac{1}{2}} ([0,T] ; L^1 (B_R))$.
\end{remark}

\subsubsection{Uniqueness}

The uniqueness result can be obtained by continuous dependence in $L^1$ with respect $\rho_0$ and $E$. Given two strong solutions $\rho$ and $\overline{\rho}$ to the problem \eqref{eq:FDE_Problem_+_known_transport} with given drifts $E$ and $\overline{E}$ respectively, we follow an argument similar to \cite[Theorem 2.10]{CGV22}.
\begin{theorem}[$L^1$ continuous dependence]
	\label{th:L1_continuous_dependence}
	Let $\rho$ and $\overline{\rho}$ be two strong solution of \eqref{eq:FDE_Problem_+_known_transport}  with given smooth drifts $E$ and $\overline{E}$ respectively. If $E = \overline E$ then we have the $L^1$ contraction principle
	\begin{equation*}
		\int_{B_R} [ \rho_T - \overline{\rho}_T ]_+ \leq  \int_{B_R} [ \rho_0 - \overline{\rho}_0]_{+} .
	\end{equation*}
	If $E \not \equiv \overline E$ and $\rho \in L^\infty ( (0,T) \times B_R)$ then we have that
	\begin{equation}\label{ec:L1 continuous dependence}
		\int_{B_R} [ \rho_T - \overline{\rho}_T ]_+ \leq  \int_{B_R} [ \rho_0 - \overline{\rho}_0]_{+} + \|\rho_0 \|_{L^1(B_R)} \int_0^T \| \nabla \cdot (E_t - \overline{E}_t ) \|_{L^{\infty}(B_R)} 
		+ C_1 \left( \int_0^T \| E_t - \overline{E}_t \|_{L^2 (B_R)}^2 \right)^{1/2} ,
	\end{equation}
	where,
	\begin{equation*}
		C_1 = \left( \int_{B_R} G_{\Phi}(\rho_0)  + \frac{1}{2}  \int_0^T \left\| \frac{\rho_t}{\Phi ' (\rho_t )} \right\|_{L^{\infty} (B_R)}^2 \| E_t \|_{L^2 (B_R)}^2 \right)^{1/2},
	\end{equation*}
	and $G_{\Phi}$ is defined in \eqref{def:G_Phi}.
\end{theorem}

\begin{proof}
	Denote $w_t = \Phi ( \rho_t ) - \Phi (\overline{\rho}_t)$. Let $j$ be convex and denote $p = j'$. We have
	\begin{align*}
		\int_{B_R}  \frac{\partial}{\partial t} (\rho_t - \overline{\rho}_t) p (w_t) & = \int_{B_R} p (w_t) \nabla \cdot \left( \nabla w_t + (\rho_t E_t - \overline{\rho}_t \bar{E}_t ) \right) \\
		& = - \int_{B_R} p'(w_t) | \nabla w |^2 - \int_{B_R} \nabla p (w_t)  \cdot ( \rho_t E_t - \overline{\rho}_t \overline{E}_t).
	\end{align*}
	Then,
	\begin{align*}
		\int_{B_R}  \frac{\partial}{\partial t} (\rho_t - \overline{\rho}_t) p (w_t) & \leq -  \int_{B_R} \nabla p (w_t) \cdot  ( \rho_t E_t - \overline{\rho}_t \overline{E}_t) \\ 
		& =  - \int_{B_R} \nabla p (w_t) \cdot \rho_t (E_t - \overline{E}_t )   - \int_{B_R} \nabla p (w_t) \cdot \overline{E}_t (\rho_t - \overline{\rho}_t ) \\
		& = \int_{B_R} p(w_t) \left( \nabla \rho_t \cdot ( E_t - \overline{E}_t) + \rho_t \nabla \cdot (E_t - \overline{E}_t ) \right)   \\
		&  \quad + \int_{B_R} p(w_t) \left( \nabla (\rho_t - \overline{\rho}_t ) \cdot \overline{E}_t + (\rho_t - \overline{\rho}_t ) \nabla \cdot \overline{E}_t \right).
	\end{align*}
	Now, we take the limit $p(s) \rightarrow \mathrm{sign}_0^+ (s)$, the positive part function such that $\mathrm{sign}_0^+ (0) = 0$, and get the bound
	\begin{align*}
		\int_{B_R} [ \rho_T - \overline{\rho}_T ]_+  & \leq \int_{B_R} [ \rho_0 - \overline{\rho}_0]_{+} + \int_0^T \int_{B_R} \mathrm{sign}_0^+ (w_t) \left( \nabla \rho_t \cdot ( E_t - \overline{E}_t) + \rho_t \nabla \cdot (E_t - \overline{E}_t ) \right) \\
		& \quad + \int_0^T \int_{B_R} \nabla (\rho_t - \overline{\rho}_t )_+ \cdot \overline{E}_t + (\rho_t - \overline{\rho}_t )_+ \nabla \cdot \overline{E}_t \\
		& \leq  \int_{B_R} [ \rho_0 - \overline{\rho}_0]_{+} \\
		& \quad + \| \nabla \rho_t \|_{L^2 (B_R)} \| E_t -\overline{E}_t \|_{L^2(B_R)} +  \| \rho_0 \|_{L^1 (B_R)}  \| \nabla \cdot (E_t - \overline{E}_t ) \|_{L^{\infty}(B_R)} \\
		& \quad + \int_0^T \int_{B_R} \nabla \cdot \left( [ \rho_t - \overline{\rho}_t ]_+ \overline{E}_t\right),
	\end{align*}
	due to a direct application of the Dominated Convergence Theorem using the properties of the strong solutions. The last term vanishes since $\overline{E}_t (x) \cdot \nu (x) = 0$ on $\partial B_R$. If $E = \overline E$ then, the $L^1$ contraction is obvious. If $E \neq \overline E$, considering the \textit{a priori} estimate \eqref{ec:a_priori_estimate_rho}, we recover the desired inequality. 
\end{proof}

\subsubsection{Existence} 
We proceed to prove existence for the problem \eqref{eq:FDE_Problem_+_known_transport}, by constructing solutions for regular $\Phi$. We consider the following sequence $\Phi_k$ of functions satisfying \eqref{ec:unif_ellipticity} given by $\Phi_k (0) = 0$ and  
\begin{equation}\label{ec:Phik}
	\Phi'_k (s) 
	\sim 
	\begin{dcases}
		m k^{m-1} & s > k, \\
		m s^{m-1} & s \in [k^{-1}, k], \\
		m k^{1-m} & s < k^{-1},
	\end{dcases}
\end{equation}
up to a smoothing. The sequence $\Phi_k$ is such that, $\Phi_k ( s) \rightarrow \Phi (s)$ for all $s \geq 0$. We take the sequence 
$\rho^{(k)}$ given as the unique solutions of the problems 
\begin{equation}\label{eq:FDE_Problem_+_known_transport_3}
	\begin{dcases}
		\frac{\partial {\rho^{(k)}} }{\partial t}  = \Delta \Phi_k({\rho^{(k)}}) + \nabla \cdot ({\rho^{(k)}} {E^{(k)}}) \qquad & \mathrm{in} \, (0, \infty ) \times B_R,  \\
		\left( \nabla  \Phi_k(\rho^{(k)}) + \rho^{(k)}   E^{(k)} \right) \cdot \nu (x) = 0 &  \mathrm{on} \, (0, \infty ) \times \partial B_R, \\
		{\rho^{(k)}} (0, x) = {\rho}^{(k)}_0 (x)& x \in B_R,
	\end{dcases} 
\end{equation}
where $\rho^{(k)}_0 \in C^2( \overline {B_R} )$ such that $\rho_0^{(k)} \to \rho_0$ in $L^\infty (B_R)$ and $E^{(k)}$ smooth such that $E^{(k)} \rightarrow E$ in $W^{1, \infty}([0,T] \times B_R)$. 

We will prove existence of solutions for the problem \eqref{eq:FDE_Problem_+_known_transport} using a compactness type argument. 

\begin{theorem}[Existence]
	\label{th:Existence_FD_known_transport_(0-T)xBR} 
	Assume $\rho_0 \in L^{\infty}(B_R)$, $E \in W^{1, \infty} ([0,T] \times B_R)$, and $\rho^{(k)}$ the strong solutions of \eqref{eq:FDE_Problem_+_known_transport_3}. Then, $\rho^{(k)}$ converges weakly in $H^1((0,T) \times B_R )$, strongly in $L^2 ((0,T) \times B_R)$ and in $C([0,T] ;L^1(B_R))$ to $\rho$, the unique strong solution of \eqref{eq:FDE_Problem_+_known_transport}.
\end{theorem}

\begin{proof}
	\textbf{Step 1. $ \rho^{(k)}$ is uniformly bounded in $L^2 (0,T; H^1 (B_R))$.} Due to the $L^p$ estimate \eqref{ec:Lp_estimate},
	\begin{align}\label{eq:L infty estimate E smooth fixed}
		\|\rho^{(k)} \|_{L^\infty ((0,T) \times B_R)} & \leq  \exp \left(\int_0^T \| \nabla \cdot E_t^{(k)} \|_{L^{\infty}(B_R)}  \right) \| \rho_0^{(k)} \|_{L^\infty (B_R)}  \leq \Lambda_T < \infty,
	\end{align}
	since $\rho_0^{(k)} \to \rho_0$ in $L^\infty (B_R)$. 
	Because of the {\em a priori} estimate \eqref{ec:a_priori_estimate_rho},
	\begin{equation}\label{eq:nabla_rhok_uniformly_bdd_intermiadiate_step}
		\int_0^T \int_{B_R} | \nabla \rho^{(k)}_t |^2  \leq \int_{B_R} G_{\Phi_k}( \rho_0^{(k)} ) + \frac{1}{2} \int_0^T \left\| \frac{\rho^{(k)}_t}{\Phi'_k (\rho^{(k)}_t)} \right\|_{L^{\infty}(B_R)} \| E_t^{(k)} \|_{L^2 (B_R)},
	\end{equation}
	where $G_{\Phi_k}$ is defined in \eqref{def:G_Phi}. Since $G_{\Phi_k}$ are non-decreasing functions we have that
	\begin{equation*}
		0 = G_{\Phi_k} (0) \le G_{\Phi_k} ( \rho_0^{(k)} (x) ) \leq G_{\Phi_k} ( \| \rho_0^{(k)} \|_{L^{\infty} (B_R)}).
	\end{equation*}
	Moreover, $G_{\Phi_k} (\sigma) \rightarrow \frac{\sigma^{3-m}}{m  (2-m)  (3-m)}$ for all $\sigma \geq 0$. Therefore, for $k$ big enough we have that,
	\begin{equation*}
		G_{\Phi_k} (\| \rho_0^{(k)} \|_{L^{\infty} (B_R)} ) \leq 2 \frac{\| \rho_0^{(k)} \|_{L^{\infty}(B_R)}^{3-m}}{m  (2-m)  (3-m)}.   
	\end{equation*}
	Then, $G_{\Phi_k}(\rho_0^{(k)}(x))$ is uniformly bounded. The second term of the RHS of the estimate \eqref{eq:nabla_rhok_uniformly_bdd_intermiadiate_step} is also uniformly bounded, 
	\begin{equation*}
		\left| \frac{\rho^{(k)}_t}{\Phi'_k (\rho^{(k)}_t)} \right| 
		\leq 2 
		\begin{dcases}
			\frac 1 m |\rho^{(k)}_t|^{2-m} & \text{if } |\rho^{(k)}_t| \ge k^{-1},\\
			\frac 1 m k^{m-2}  & \text{if } |\rho^{(k)}_t| < k^{-1}.
		\end{dcases}
	\end{equation*}
	where the constant $2$ accounts for the regularisation of $\Phi_k'$ on the corners.
	Since $\| \rho^{k} \|_{L^\infty ( (0,T) \times B_R )}$ is uniformly bounded, so is this quantity. $\| E_t^{(k)} \|_{L^2 (B_R)}$ is uniformly bounded by $C \| E_t \|_{L^2 (B_R)}$ because of the convergence $E^{(k)} \rightarrow E$ in $W^{1, \infty} ([0,T] \times B_R)$. In particular,  
	\begin{equation*}
		\sup_k \int_0^T \int_{B_R} | \nabla \rho^{(k)}_t |^2 < \infty.
	\end{equation*}
	Therefore, $\rho^{(k)}$ is uniformly bounded in $L^2 (0, T; H^1 (B_R))$.
	
	\textbf{Step 2. $\nabla \Phi_k (\rho^{(k)})$ is uniformly bounded in $L^2 ((0,T) \times B_R)$.} 
	Due to \eqref{ec:a_priori_estimate_Phi}, it follows that
	\begin{equation*}
		\int_0^T \int_{B_R} | \nabla \Phi_k (\rho_t^{(k)}) |^2  \leq  \int_{B_R} \Psi_k (\rho_0^{(k)} ) + \int_0^T \| E_t^{(k)} \|_{L^{\infty}(B_R)}^2 \| \rho_t \|_{L^2(B_R)}^2,
	\end{equation*}
	where $\Psi_k (s) = 2 \int_0^s \Phi_k (\sigma ) \, d \sigma$.
	Hence, for $k > \Lambda_T$, $\Phi_k (\rho^{(k)}) \leq \Phi (\rho^{(k)})$, so $\Psi_k(\rho^{(k)}) \leq \Psi (\rho^{(k)})$ we have,
	\begin{align*}
		\int_0^T \| \nabla \Phi_k (\rho_t^{(k)}) \|_{L^2 (B_R)}^2 & \leq \int_{B_R} \Psi (\rho_0^{(k)} )  + \int_0^T C \| E_t \|_{L^{\infty}(B_R)}^2 \| \rho_t \|_{L^2(B_R)}^2 \\
		& \leq \frac{2}{m+1} |B_R| \|\rho_0 \|_{L^{\infty}(B_R)}^{m+1} + C\int_0^T \| E_t \|_{L^{\infty}(B_R)}^2 \| \rho_t \|_{L^2(B_R)}^2. 
	\end{align*}
	
	\textbf{Step 3. $ \frac{\partial \rho^{(k)}}{\partial t}$ is uniformly bounded in $L^2 ((0,T) \times B_R)$.} 
	From \eqref{ec:drho/dt_L2_is_bounded} we get
	\begin{align*}
		\frac{1}{2}\int_0^T \int_{B_R} \Phi_k ' (\rho_t^{(k)} ) \left| \frac{\partial \rho_t^{(k)}}{\partial t} \right|^2 & \leq   \frac{1}{2} \int_{B_R} | \nabla \Phi_k (\rho_0^{(k)} ) |^2  + \frac{1}{2}\int_0^T \int_{B_R} \nabla \Phi_k  (\rho_t^{(k)} )  \nabla \rho_t^{(k)}  | E_t^{(k)} |^2 \\
		&\quad  + \int_0^T \int_{B_R} \rho_t \nabla \Phi_k (\rho_t^{(k)}) \cdot \frac{\partial E_t^{(k)}}{\partial t} \\
		& \quad - \int_{B_R} \zeta_k( \rho_0^{(k)} ) \nabla \cdot  E_0^{(k)}  + \int_{B_R} \zeta_k (\rho_T^{(k)}) \nabla  \cdot  E_T^{(k)}.
	\end{align*}
	Furthermore, because $\Phi_k'$ is non-increasing we also have
	\begin{align*}
		\frac{1}{2}\int_0^T \int_{B_R}  \left| \frac{\partial \rho_t^{(k)}}{\partial t} \right|^2 & \leq  \frac{1}{\Phi_k' (\| \rho^{(k)}_t \|_{L^{\infty}((0,T) \times B_R)})}  \Bigg( \frac{1}{2} \int_{B_R} | \nabla \Phi_k (\rho_0 ) |^2  + \frac{1}{2}\int_0^T \int_{B_R} \nabla \Phi_k  (\rho_t^{(k)} )  \nabla \rho_t^{(k)}  | E_t^{(k)} |^2 \\
		&\qquad  + \int_0^T \int_{B_R} \rho_t \nabla \Phi_k (\rho_t^{(k)}) \cdot \frac{\partial E_t^{(k)}}{\partial t}  - \int_{B_R} \zeta_k( \rho_0^{(k)} ) \nabla \cdot  E_0^{(k)}  + \int_{B_R} \zeta_k (\rho_T^{(k)}) \nabla  \cdot  E_T^{(k)} \Bigg).
	\end{align*}
	Then, $\Phi'_k (\| \rho^{(k)}_t \|_{L^{\infty}} ) \geq \Phi'_k (\Lambda_T) \geq \Gamma_T > 0$ since $\Phi_k' (\Lambda_T) \rightarrow \Phi' (\Lambda_T) > 0$. Let us now apply Hölder and Young inequalities,
	\begin{align*}
		\frac{1}{2}\int_0^T \int_{B_R}  \left| \frac{\partial \rho_t^{(k)}}{\partial t} \right|^2 & \leq  \frac{1}{\Gamma_T}  \Bigg( \frac{1}{2} \int_{B_R} | \nabla \Phi_k (\rho_0^{(k)} ) |^2  + \frac{1}{2}\int_0^T \| \nabla \Phi_k  (\rho_t^{(k)} ) \|_{L^{2} (B_R)} \| \nabla \rho_t^{(k)} \|_{L^2 (B_R)} \|  E_t^{(k)} \|^2_{L^{\infty}(B_R)} \\
		& \qquad  + \int_0^T  \| \nabla \Phi_k (\rho_t^{(k)}) \|_{L^2(B_R)} \| \rho_t \|_{L^2 (B_R)} \Big\| \frac{\partial E_t^{(k)}}{\partial t} \Big\|_{L^{\infty} (B_R)} \\
		& \qquad - \int_{B_R} \zeta_k( \rho_0^{(k)} ) \nabla \cdot  E_0^{(k)}  + \int_{B_R} \zeta_k (\rho_T^{(k)}) \nabla  \cdot  E_T^{(k)} \Bigg) \\
		& \leq   \frac{1}{2\Gamma_T} \int_{B_R} | \nabla \Phi_k (\rho_0^{(k)} ) |^2  \\
        &\qquad + \frac{C}{4\Gamma_T} \|  E \|^2_{L^{\infty}((0,T) \times B_R)}\int_0^T \left( \| \nabla \Phi_k  (\rho_t^{(k)} ) \|_{L^{2} (B_R)}^2 + \| \nabla \rho_t^{(k)} \|_{L^2 (B_R)}^2 \right)  \\
		& \qquad  + \frac{C}{2 \Gamma_T}  \Big\| \frac{\partial E}{\partial t} \Big\|_{L^{\infty} ((0,T) \times B_R)} \int_0^T \left(  \| \nabla \Phi_k (\rho_t^{(k)}) \|_{L^2(B_R)}^2 + \| \rho_t \|_{L^2 (B_R)}^2 \right) \\
		& \qquad - \frac{C}{\Gamma_T}\int_{B_R} \zeta_k( \rho_0^{(k)} ) \nabla \cdot  E_0  + C\int_{B_R} \zeta_k (\rho_T^{(k)}) \nabla  \cdot  E_T .
	\end{align*}
	From Step 1 we already know $\rho^{(k)}$ is uniformly bounded in $L^2 (0,T ; H^1(B_R))$, from Step 2 $\nabla \Phi_k (\rho^{(k)})$ is uniformly bounded in $L^2 ((0,T) \times B_R )$ and from the assumptions, $E \in W^{1, \infty} ([0,T] \times \overline{B_R})$. The function $\zeta_k$ is such that $\zeta_k (\rho_t^{(k)}) = \int_{B_R} \rho_t^{(k)} \nabla \Phi_k (\rho_t^{(k)})$, then the last two terms of the RHS are uniformly bounded from a combination of all the previously discussed above. Therefore, the sequence $\frac{\partial \rho^{(k)}}{\partial t}$ is uniformly bounded in $L^2 ((0,T) \times B_R)$.

	\textbf{Step 4. A subsequence converges to the strong solution}
	
	\textbf{Step 4a. Convergence by compactness.} 
	The sequence $\rho^{(k)}$ is uniformly bounded in $H^1((0,T) \times B_R)$ by Steps 1 and 3. Therefore, by the Sobolev compactness embedding Theorems, there exists a subsequence $\rho^{(k)}$ such that
	\begin{equation}\label{ec:L2_Aubin_Lionsrho}
		\rho^{(k)} \rightarrow \rho \quad \text{in} \, L^2 ( (0, T) \times B_R) \quad \text{and} \quad \rho^{(k)} \rightarrow \rho \quad \text{in} \, C ( [0, T]; L^1 (B_R)).
	\end{equation} 
	This also implies that, up to further a subsequence, the weak convergence in $H^1((0,T) \times B_R)$ and
	\begin{equation}\label{ec:ae_Aubin_Lionsrho}
		\rho^{(k)} \rightarrow \rho \quad \text{a.e.}
	\end{equation}

	\textbf{Step 4b. $\rho \in L^2 (0, T; H^1(B_R))$.} 
	$\nabla \rho^{(k)}$ is uniformly bounded on $L^2 ((0,T) \times B_R)$, then, taking into account \eqref{ec:L2_Aubin_Lionsrho}
	\begin{equation*}
		\nabla \rho^{(k)} \rightharpoonup \nabla \rho \quad \text{weakly in } L^2 ( (0, T) \times B_R),
	\end{equation*}
	and $\rho \in L^2 (0, T; H^1 (B_R))$.
	
	\textbf{Step 4c. $\frac{\partial \rho}{\partial t} \in L^2 ((0,T) \times B_R)$.} The sequence $\frac{\partial \rho^{(k)}}{\partial t}$ is uniformly bounded in $L^2((0,T) \times B_R)$. Then, up to a subsequence, $\frac{\partial \rho^{(k)}}{\partial t} \rightharpoonup \xi$ weakly in $L^2 ((0,T) \times B_R)$. From \eqref{ec:L2_Aubin_Lionsrho} we recover $\xi = \frac{\partial \rho}{\partial t}$ and $\frac{\partial \rho}{\partial t} \in L^2 ((0,T) \times B_R)$.
	
	\textbf{Step 4d. $\Phi (\rho ) \in L^2 (0, T; H^1(B_R))$.} From Step 2 we know that the sequence $\nabla \Phi_k (\rho^{(k)})$ is uniformly bounded on $L^2((0,T) \times B_R)$. Then, up to a subsequence, 
	\begin{equation*}
		\nabla \Phi_k ( \rho^{(k)} ) \rightharpoonup \xi \quad \text{weakly in } L^2 ((0,T) \times B_R).
	\end{equation*}
	Recall that by \eqref{ec:L2_Aubin_Lionsrho}, $\rho^{(k)} \rightarrow \rho$ in $L^2 ((0,T) \times B_R)$. Furthermore, $\Phi_k \rightarrow \Phi$ uniformly in $[0,\Lambda_T]$. Then,
	\begin{align*}
		\| \Phi_k (\rho^{(k)}) - \Phi (\rho ) \|_{L^2 ((0,T) \times B_R)} & \leq \| \Phi_k (\rho^{(k)}) - \Phi (\rho^{(k)} ) \|_{L^2 ((0,T) \times B_R)} + \| \Phi (\rho^{(k)}) - \Phi (\rho ) \|_{L^2 ((0,T) \times B_R)} \\
		& \leq \| \Phi_k - \Phi \|_{L^{\infty}(0, \Lambda_T)} |B_R|^{1/2} + [\Phi]_{C^{m}} \| | \rho^{(k)} - \rho |^{m} \|_{L^2((0,T) \times B_R)} \rightarrow 0,
	\end{align*}
	where $[\Phi]_{C^{m}}$ is the $m$-Hölder semi-norm of $\Phi$. The convergence to $0$ follows from the Dominated Convergence Theorem: we have almost everywhere convergence and they are bounded functions. Therefore, this implies,
	\begin{equation*}
		\Phi_k (\rho^{(k)} ) \rightarrow \Phi (\rho ) \quad \text{strongly in} \, L^2 ((0,T) \times B_R).
	\end{equation*}
	Then, we also obtain $\nabla \Phi_k (\rho^{(k)} ) \rightharpoonup \nabla \Phi( \rho )$ weakly in $L^2((0,T) \times B_R)$, and $\Phi (\rho) \in L^2 (0, T; H^1(B_R))$. 
	
	\textbf{Step 4e. $\Delta \Phi (\rho) \in L^2((0,T) \times B_R)$.} Due to the smoothness of $E^{(k)}$, $\rho^{(k)}$ are classical solutions. Then, from \textit{a posteriori} computation we get 
	\begin{align*}
		\|\Delta \Phi_k (\rho^{(k)}) \|_{L^2((0,T) \times B_R)}
		& \leq \left\| \frac{\partial \rho^{(k)}}{\partial t} - \diver \left( \rho^{(k)} E^{(k)} \right) \right\|_{L^2 ((0,T) \times B_R)} \\
		& \le \left\| \frac{\partial \rho^{(k)}}{\partial t}  \right\|_{L^2 ((0,T) \times B_R)} + C\| \rho \|_{L^2(0,T; H^1 (B_R))} \| E \|_{L^{\infty}(0,T; W^{1, \infty}(B_R))}.
	\end{align*}
	Therefore, we conclude that $\Delta \Phi_k (\rho^{(k)} ) \rightharpoonup \Delta \Phi( \rho )   \text{ weakly in }  L^2((0,T) \times B_R)$.

	\textbf{Step 4f. Distributional solution.} We have now all the ingredients to see that for all $\varphi \in X = \left\lbrace \varphi \in C^{\infty} ([0,T] \times \overline{B_R}) \, : \, \varphi(T)=0 \right\rbrace$,
	\begin{equation*}
		\int_{B_R} \rho^{(k)}_0 \varphi (0)   + \int_0^T \int_{B_R} \rho^{(k)}_t \frac{\partial \varphi}{\partial t} = \int_0^T \int_{B_R} \nabla\Phi_k (\rho^{(k)}_t) \nabla \varphi + \int_0^T \int_{B_R} \rho^{(k)}_t E_t \nabla \varphi  
	\end{equation*}
	converges to
	\begin{equation*}
		\int_{B_R} \rho_0 \varphi (0) + \int_0^T \int_{B_R} \rho_t \frac{\partial \varphi}{\partial t} = \int_0^T \int_{B_R} \nabla \Phi( \rho ) \nabla \varphi + \int_0^T \int_{B_R} \rho_t E_t \nabla \varphi ,
	\end{equation*}
	Therefore, $\rho$ is a strong solution of \eqref{eq:FDE_Problem_+_known_transport}. 
	
	\noindent \textbf{Step 5. Convergence of the whole sequence.} Since there is a unique strong solution of the limit problem, every subsequence has a further subsequence converging to the same limit, and hence the whole sequence converges.
\end{proof}
Before finishing this section we want to remark that $\rho$, the strong solution obtained in \Cref{th:Existence_FD_known_transport_(0-T)xBR}, preserves all the \textit{a priori} estimates discussed during subsection \ref{sec:The_equation_in_BR}, which are, the $L^p$ estimate \eqref{ec:Lp_estimate} and the \textit{a priori} estimates of the $L^2((0,T) \times B_R)$ norm of $\nabla \Phi (\rho)$, $\nabla \rho$ and $\frac{\partial \rho}{\partial t}$ which are presented in \eqref{ec:a_priori_estimate_Phi}, \eqref{ec:a_priori_estimate_rho}, and \eqref{ec:a_priori_estimate_drho/dt_L2} respectively.

\subsection{Proof of \texorpdfstring{\Cref{th:Banach_fixed_point}}{Theorem \ref{th:Banach_fixed_point}}}\label{sec:FD+V+K} 
	We divide the proof in several steps:
	
	\textbf{Step 1. Define the contraction mapping.} The goal is to prove existence and uniqueness of the solution using Banach fixed point Theorem to an operator $\mathcal K$ defined as follows. First we define
	\begin{equation*}
		\mathcal{T} : \, X  \longrightarrow C ( [0,T]; L^1 (B_R)),
	\end{equation*}
	as the functional that maps the vector field $E$ with the solution of \eqref{eq:FDE_Problem_+_known_transport} for initial data $\rho_0$ and vector field $E$, where due to \Cref{th:Existence_FD_known_transport_(0-T)xBR} we can take the domain as
	\begin{equation*}
		X = \{ E \in W^{1, \infty} ( [0,T] \times B_R) \text{ that satisfies \eqref{ec:extra_boundary_condition}}\}.
	\end{equation*}
	As a second step of the construction of our fixed-point operator, we define the map from $\rho$ to $E$ by
	\[
	\mathcal {E} :
		C ( [0,T]; L^1 (B_R))  \longrightarrow X,  \qquad 
		\mathcal E[\rho] (x)  = \nabla V(x) +  \nabla \int_{B_R} K(x,y) \cdot \rho (y) \, dy .
		\]
		Finally, we define the operator we will use for a fixed point argument as
		\begin{equation}\label{eq:mapping FPT}
			\mathcal{K} = \mathcal T \circ \mathcal E :
				C ( [0,T]; L^1 (B_R))  \longrightarrow C ( [0,T]; L^1 (B_R)) .         
		\end{equation}
		It is easy to see that the ball
		\begin{equation*}
			A := \lbrace u \in C ( [0,T]; L^1 (B_R)) : \sup_{t \in (0,T)} \| u_t \|_{L^1 (B_R)} \leq \| \rho_0 \|_{L^1 (B_R)} \rbrace ,
		\end{equation*}
		is such that $\mathcal{K} (A) \subseteq A$ because of the $L^p$ estimate \eqref{ec:Lp_estimate} for $p=1$, independently of $T$. 
		
		\textbf{Step 2. Prove that our mapping is contractive.} Let us now prove that the operator $\mathcal K : A \to A$ is contracting for small enough $T$. We look at the Lipschitz constant of $\mathcal T$ and $\mathcal E$.
		We know that, due to the estimate \eqref{ec:Lp_estimate} for $p = \infty$, we have the bound 
		\begin{equation}\label{eq:T[E] is bdd in Linfty}
			\|\mathcal T[E] \|_{L^\infty ((0,T) \times B_R)} \leq  \exp \left(\int_0^T \| \nabla \cdot E_t \|_{L^{\infty}(B_R)}  \right) \| \rho_0 \|_{L^\infty (B_R)}.
		\end{equation}
		Then, due to the $L^1$ continuous dependence result from \Cref{th:L1_continuous_dependence} we have that 
		\begin{equation*}
			\|\mathcal{T} [E]_t - \mathcal{T}[\overline{E}]_t \|_{L^1(B_R)} \leq C  \left( T^{1/2} \sup_{t \in (0,T)} \| E_t- \overline{E}_t \|_{L^2(B_R)} + T \sup_{t \in (0,T)} \| \diver (E_t - \overline{E}_t) \|_{L^{\infty}(B_R)} \right),
		\end{equation*}
		where $C$ is equal to the sum of the constant $C_1$ from \Cref{th:L1_continuous_dependence} and $\|\rho_0 \|_{L^1(B_R)}$, i.e.
		\begin{equation*}
			C= \|\rho_0 \|_{L^1(B_R)} + \left( \int_{B_R} G_{\Phi}(\rho_0 ) + \frac{1}{2} \int_0^T \left\| {\mathcal T[E]_t} \right\|_{L^\infty}^{2(2-m)} \| E_t \|^2_{L^{\infty} (B_R)} \right)^{1/2},
		\end{equation*}
		with $G_{\Phi}$ defined in \eqref{def:G_Phi}.
		
		Let us plug in the potential $\mathcal{E}[\rho]$. We start by realizing that if we replace $E$ by $\mathcal{E}[\rho]$ in \eqref{eq:T[E] is bdd in Linfty}, everything is bounded. This is because the term $\int_0^T \| \nabla \cdot \mathcal{E}[\rho]_t \|_{L^{\infty} (B_R)}$ can be bounded as follows, 
		\begin{align*}
			\left\| \nabla \cdot \mathcal E [\rho] \right\|_{L^{\infty} (B_R)}  & \leq \| \Delta V \|_{L^{\infty}(B_R)} + \| \Delta_x K \|_{L^{ \infty}(B_R \times B_R)}  \| \rho \|_{L^1 (B_R)} \\
			& \leq \| \Delta V \|_{L^{\infty}(B_R)} + \| \Delta_x K \|_{L^{ \infty}(B_R \times B_R)}  \| \rho_0 \|_{L^1 (B_R)}.
		\end{align*}
		Notice that 
		\begin{align*}
			\| \mathcal{E} [\rho]_t -\mathcal{E} [\overline{\rho}]_t \|_{L^2 (B_R)} &\leq    \| \nabla_x K \|_{L^{2 } (B_R \times B_R)}  \| \rho_t - \overline{\rho}_t \|_{L^1 (B_R)} \\
			\| \nabla \cdot (\mathcal{E} [\rho]_t - \mathcal{E} [\overline{\rho}]_t )\|_{L^{\infty} (B_R)} &\leq   \| \Delta_x K \|_{L^{ \infty} (B_R \times B_R)}    \| \rho_t - \overline{\rho}_t \|_{L^1 (B_R)}.
		\end{align*}
		Finally, getting everything together, we have that
		\begin{align*}
			\|  \mathcal{K}[\rho]_t - \mathcal{K} [\overline{\rho}]_t \|_{L^1 (B_R)}  &\leq  C  \left( T^{1/2} \sup_{t \in (0,T)} \| \mathcal{E}[\rho]_t - \mathcal{E}[\bar{\rho}]_t \|_{L^2(B_R)} + T \sup_{t \in (0,T)} \| \diver (\mathcal{E}[\rho]_t - \mathcal{E}[\overline{\rho}]_t) \|_{L^{\infty}(B_R)} \right) \\ 
			&\leq C  (T^{1/2} + T ) \sup_{t \in (0,T)} \| \rho_t - \overline{\rho}_t \|_{L^1 (B_R)}, 
		\end{align*}
		for some bounded constant $C(T)$, which can be taken non-decreasing in $T$. Taking $T$ small enough so that $C (T) (T^{1/2} + T ) <1$ and applying Banach fixed point Theorem we get existence and uniqueness of solutions for the problem \eqref{eq:Fast-Diffusion_Problem_BR}.
\qed

\section{Long time asymptotic behaviour}\label{sec:Long time asymptotic}
In this Section we want to study the asymptotic behaviour of the solutions of the problem \eqref{eq:Fast-Diffusion_Problem_BR} when $t \rightarrow \infty$. In order to be able do that, we will need some \textit{a priori} estimates that come from the dissipation of the free energy. We obtain them in subsection \ref{sec:Free_energy}. After that, in subsection \ref{sec:Convergence_to_stationary_state} we study the asymptotic behaviour of the solution using these estimates. 

\subsection{Free energy and its dissipation}\label{sec:Free_energy}

We can derive a variational interpretation of the problem \eqref{eq:Fast-Diffusion_Problem_BR} that leads to additional {\em a priori} estimates using its gradient flow structure. The equation \eqref{eq:Fast-Diffusion_Problem_BR} can be rewritten as
\begin{align}\label{ec:Gradient_Flow_Ball_deduction}
	\begin{split}
		\frac{\partial \rho}{\partial t} & = \nabla \cdot \left( \rho \nabla \left\lbrace U' (\rho) + V + \int_{B_R} K(\cdot , y) \rho (y) \, dy \right\rbrace \right),
	\end{split}
\end{align} 
where, $U (\rho) = \frac{\rho^m}{m-1}$. Formulation \eqref{ec:Gradient_Flow_Ball_deduction} shows that the equation \eqref{eq:Fast-Diffusion_Problem_BR} is formally the gradient flow of the free energy
\begin{equation*}
	\mathcal{F}_R [ \mu ] :=  \int_{B_R} U (\mu(x)) \, dx + \int_{B_R} V(x)\,  d \mu(x)  + \frac{1}{2} \int_{B_R} \int_{B_R } K(x,y) \, d \mu (y) d \mu (x) .
\end{equation*}
Thanks to the free energy we can get some useful {\em a priori} estimates. We can prove \Cref{prop:Convergence_of_the_free_energies}.

\begin{proof}[Proof of \Cref{prop:Convergence_of_the_free_energies}]
	For $\varepsilon \in (0,1)$, define $U_{\varepsilon}$ by $U_\varepsilon(\varepsilon) = U(\varepsilon)$ such that 
	\begin{equation}\label{ec:Theta_epsilon}
		U'_{\varepsilon} (s) \sim \left\lbrace
		\begin{array}{ll}
			U' (\varepsilon) & s < \varepsilon, \\
			U' (s) & s \geq \varepsilon ,
		\end{array} \right.
	\end{equation}
	up to a smoothing. We define the free energy, 
	\begin{equation*}
		\mathcal{F}_{R, \varepsilon} [\rho] = \int_{B_R} U_{\varepsilon} (\rho (x)) \, dx + \int_{B_R} V(x) \rho (x) \, dx + \frac{1}{2} \int_{B_R} \int_{B_R} K(x,y) \rho (y) \rho (x) \, dy \, dx.
	\end{equation*}
	Because $\frac{\partial \rho}{\partial t} \in L^2((0,T) \times B_R)$,
	\begin{equation*}
		\frac{d}{dt} \mathcal{F}_{R, \varepsilon} [\rho_t] =  - \int_{B_R} \rho_t \left[ U' (\rho_t) \nabla \rho_t + E \right] \left[ U_{\varepsilon}' (\rho_t) \nabla \rho_t + E \right],
	\end{equation*}
	where $E$ denotes $\nabla V + \nabla \int_{B_R} K( \cdot , y ) \rho (y) \, dy$. Therefore,
	\begin{equation}\label{ec:F_epsilon_along_trajectories}
		\int_{t_1}^{t_2} \int_{B_R} \rho_t \left[ U' (\rho_t) \nabla \rho_t + E \right] \left[ U_{\varepsilon}' (\rho_t) \nabla \rho_t + E \right] = \mathcal{F}_{R, \varepsilon} [\rho_{t_1}] - \mathcal{F}_{R, \varepsilon} [\rho_{t_2}].
	\end{equation}
	First, let us prove the convergence of the RHS. It is immediate to see that $U_\varepsilon \to U$ uniformly in $[0,\infty)$. Hence, we have that
	\begin{equation*}
		\int_{B_R} | U_\ee (\rho_t) - U(\rho_t) | \le |B_R| \|U_\varepsilon - U \|_{L^\infty (0,\infty)} \to 0. 
	\end{equation*}
	This proves that $\mathcal F_{R, \varepsilon}[\rho_t] \to \mathcal F_R[\rho_t]$ for every $t \ge 0$.
 
	Let us discuss about the convergence of the LHS. 
	We expand 
	\begin{align}\label{ec:Proof of dissipation expansion}
		\int_{t_1}^{t_2} \int_{B_R} \rho_t \left[ U' (\rho_t) \nabla \rho_t + E \right] \cdot \left[ U_{\varepsilon}' (\rho_t) \nabla \rho_t + E \right] & = \int_{t_1}^{t_2} \int_{B_R} \rho_t U' (\rho_t) U'_{\varepsilon} (\rho_t) | \nabla \rho_t |^2 \\ 
		& \quad + \int_{t_1}^{t_2} \int_{B_R} \rho_t U'_{\varepsilon} (\rho_t) \nabla \rho_t \cdot E + \rho_t U' (\rho_t ) \nabla \rho_t \cdot E + \rho_t  |E|^2 \nonumber
	\end{align}
	We start by proving the convergence of the first term of the RHS of \eqref{ec:Proof of dissipation expansion}. In order to do that we define the
	functions $\Lambda_{\varepsilon}, \Lambda : [0, \infty) \rightarrow \mathbb{R}$ so that,
	\begin{align*}
		\Lambda_{\varepsilon}' (s) := \left( s U ' (s) U_{\varepsilon}' (s) \right)^{\frac{1}{2}}, \quad
		\Lambda ' (s) := \left( s U ' (s) U' (s) \right)^{\frac{1}{2}}.
	\end{align*}
	We prove that $\nabla \Lambda_{\varepsilon} (\rho) \rightarrow \nabla \Lambda (\rho )$ strongly in $L^2 (B_R)$, indeed
	\begin{align*}
		\int_{B_R} \left| | \nabla \Lambda_{\varepsilon} (\rho ) |^2 - | \nabla \Lambda (\rho ) |^2 \right| & = \int_{B_R} \rho | \nabla \rho |^2 \left| U' (\rho ) U_{\varepsilon}' (\rho ) - U ' (\rho ) U ' (\rho ) \right| \\
		& = \int_{B_R \cap \left\lbrace \rho (x) < \varepsilon \right\rbrace} \rho | \nabla \rho |^2 \left| U' (\rho ) U_{\varepsilon}' (\rho ) - U ' (\rho ) U ' (\rho ) \right| \\
		& = \left( \frac{m}{1-m} \right)^2 \int_{B_R \cap \left\lbrace \rho (x) < \varepsilon \right\rbrace} \rho^m | \nabla \rho |^2  \left(      \rho^{m-1} - \varepsilon^{m-1}  \right) \\
		& \leq C(m) \int_{B_R \cap \left\lbrace \rho (x) < \varepsilon \right\rbrace} \rho | \nabla \rho^m |^2      \le  C(m)  \varepsilon \int_{B_R}  | \nabla \rho^m |^2 \rightarrow 0,
	\end{align*}
	as $\varepsilon \rightarrow 0$. The convergence of the remaining term of the RHS of \eqref{ec:Proof of dissipation expansion} is satisfied due to the strong convergence $\rho U'_{\varepsilon} (\rho) \rightarrow \rho U' (\rho)$ in $L^2((0,T) \times B_R)$. 
\end{proof}

\begin{corollary}
	\label{cor:FR rhot decreasing}
	Assume the basic hypotheses  \eqref{hyp:Basic BR} on the initial data $\rho_0$ and the potentials $V$ and $K$, then $\mathcal F_R[\rho_t]$ is non-increasing, bounded below, and therefore has a finite limit as $t \to \infty$.
\end{corollary}
\begin{proof} 
	\Cref{prop:Convergence_of_the_free_energies} gives us that $\mathcal{F}_R [\rho_t]$ is non-increasing along solutions of \eqref{eq:Fast-Diffusion_Problem_BR}. Furthermore, we can prove that $\mathcal{F}_R [\rho_t]$ is bounded below for solutions of \eqref{eq:Fast-Diffusion_Problem_BR} because 
	\begin{equation}\label{ec:FR_is_bounded_below}
		\mathcal{F}_R [ \rho] \geq - \frac{1}{1-m} \|\rho\|_{L^1 (B_R)} |B_R|^{1-m} - \| V_- \|_{L^{\infty}(B_R)} \| \rho \|_{L^1(B_R)} - \frac{1}{2} \| K_-\|_{L^{\infty}(B_R \times B_R)}  \| \rho \|_{L^1 (B_R)}^2  ,
	\end{equation}
	so $\mathcal{F}_R [\rho_t]$ is uniformly bounded in time and it has a finite limit as $t\to\infty$.
\end{proof}

We are also able to deduce using these energy estimates an $L^1$ bound of $\nabla \rho^m$. We will do it on a bounded subdomain $\Omega$ to profit from it when $R \rightarrow \infty$.
\begin{corollary}\label{prop:a_priori_estimate_grad_rhom_(0,T)_K}
	Assume the basic hypotheses  \eqref{hyp:Basic BR} on the initial data $\rho_0$ and the potentials $V$ and $K$, then we have that
	\begin{equation*}
		\int_{t_1}^{t_2} \! \int_{\Omega} | \nabla \rho_t^m | \leq \| \rho_0 \|_{L^1 (B_R)} \left( \mathcal{F}_{R} [\rho_{t_1}] - \mathcal{F}_{R} [\rho_{t_2}] + \int_{t_1}^{t_2} \! \int_{\Omega} \rho_t \left| \nabla V + \nabla \int_{B_R} K(\cdot ,y) \rho_t (y) \, dy \right|^2 \right)^{\frac{1}{2}}, \, \forall \Omega \subseteq \overline{B_R}.
	\end{equation*}
\end{corollary}
\begin{proof}
	We have that
	\begin{equation}\label{ec:estimate_on_nabla_rho^m}
		\int_{t_1}^{t_2} \int_{\Omega} | \nabla \rho_t^m | = \int_{t_1}^{t_2} \int_{\Omega} \rho \left| \frac{m}{m-1} \nabla \rho^{m-1} \right| \leq \int_{t_1}^{t_2} \| \rho_t \|_{L^1 (B_R)} \left( \int_{\Omega} \rho_t \left| \nabla \frac{m}{m-1} \rho^{m-1} \right|^2 \right)^{\frac{1}{2}} \, dt.
	\end{equation}
	Hence, we can conclude the result using \Cref{prop:Convergence_of_the_free_energies}, Jensen's inequalities and the conservation of the $L^1$ norm.
\end{proof}

\subsection{Asymptotics in time. Proof of \texorpdfstring{\Cref{th:Convergence_stationary_state}}{Theorem \ref{th:Convergence_stationary_state}}}\label{sec:Convergence_to_stationary_state}

\begin{proof}[Proof of \Cref{th:Convergence_stationary_state}] 

    We have that $\| \rho^{[n]}_t \|_{L^1 (B_R)} =  \| \rho_0 \|_{L^1 (B_R)}$,	and, in particular, the sequence $ \rho^{[n]}_t$ is uniformly bounded on $L^1 (B_R)$. Next, we check that the sequence $\frac{\partial \rho^{[n]}}{\partial t}$ is uniformly bounded in $L^2(0,1; W^{-1, 1} (B_R))$. Using the characterisation of $W^{-1, 1}$ we get \eqref{ec:inequality_about_drho/dt_W-1,1}. Due to \Cref{cor:FR rhot decreasing} we have that $\mathcal{F}_R [\rho_{t_n}] - \mathcal{F}_R [\rho_{t_n+1} ]$ is uniformly bounded by a constant $C$. By properties of the Bochner integral we can take advantage of \eqref{ec:asymptotics_equicontinuity}. In particular, this last result implies that the sequence $\rho^{[n]} $ is equicontinuous. Furthermore, because $L^1(B_R)$ is compactly embedded in $W^{-1,1} (B_R)$ we know by Ascoli-Arzela Theorem that, up to a subsequence, we have the convergence
	\begin{equation*}
		\rho^{[n_k]} \rightarrow \widehat \mu \quad \mathrm{in} \, C([0,1]; W^{-1,1} (B_R)).
	\end{equation*}
	For any $t,s \in [0,1]$, using \eqref{ec:asymptotics_equicontinuity} we also have that the distribution $\widehat{\mu}_t$ does not depend on time because
	\begin{align*}
		\| \widehat{\mu}_t - \widehat{\mu}_s \|_{W^{-1,1}(B_R)} & = \lim_{k \rightarrow \infty} \| \rho_t^{[n_k]} - \rho_s^{[n_k]} \|_{W^{-1,1}(B_R)} \\
		& \leq \| \rho_0 \|^{\frac{1}{2}}_{L^1(B_R)} |t-s|^{\frac{1}{2}}  \lim_{k \rightarrow \infty}  \left( \mathcal{F}_R [\rho_{t_{n_k}}] - \mathcal{F}_R [ \rho_{t_{n_k}+1} ] \right)^{\frac{1}{2}} = 0,
	\end{align*}
	where the last step is again due to \Cref{cor:FR rhot decreasing}.
\end{proof}

Notice that in the setting of strong solutions we cannot characterise $\widehat{\mu}$ as a stationary solution of the problem \eqref{eq:Fast-Diffusion_Problem_BR} because the $W^{-1,1}$ convergence is not sufficient to pass to the limit under the non-linearity of the fast diffusion term. Therefore, we aim to characterise this limit in the mass equation.

\section{An equation for the mass}\label{sec:Mass_eq}
The aim of this section is to develop the well-posedness theory for the mass equation \eqref{eq:General_Mass_equation} in order to characterise the stationary state. We will show that the natural notion of solution in this setting is the notion of viscosity solution.  

In order to study the asymptotic behaviour we will rely on the construction of $\widehat{\mu}$ in subsection \ref{sec:Convergence_to_stationary_state} as a limit of the densities $\rho^{[n]}$ defined in \eqref{seq:time_steps} and on the results from DiBenedetto \cite[Chapter III]{DiB93}.

\subsection{Mass equation for the regularised problem}\label{sec:Mass equation regularized problem}
Assuming further to the basic hypotheses \eqref{hyp:Basic BR} on the initial data $\rho_0$ and the potentials $V$ and $K$ and that they are radially symmetric, it suffices to study the mass variable
\begin{equation}\label{def:M}
	M(t, v) = \int_{\widetilde{B_v}} \rho_t (x) \, dx. 
\end{equation}
We will denote by $r$ the radial variable $r=|x|$, by $v=|B_1| r^d$, the volumetric variable and by $\widetilde{B_v}$, the ball centred
in the origin and with radius $r=( v{|B_1|}^{-1} )^{1/d}$. It was shown in \cite{CGV22} that $M$ satisfies a Hamilton-Jacobi type equation which is better written in the volume variable $v$ as  
\begin{equation}\label{eq:General_Mass_equation}
	\frac{\partial M}{\partial t} = \kappa (v)^2 \frac{\partial }{\partial v} \Phi \left( \frac{\partial M}{\partial v} \right) + \kappa (v)^2 \frac{\partial M}{\partial v} \diffEv [\rho], \quad \kappa (v) = d \omega_d^{\frac{1}{d}} v^{\frac{d-1}{d}}, \quad \text{in } (0, \infty) \times (0,R_v),
\end{equation}
where  $R_v = |B_R|$, and $\diffEv$ is an operator written in volumetric coordinates that substitutes the previous operator $E$ that we have in earlier sections. It makes sense to write this operator in volumetric coordinates because if $\eta$ and $\mu$ are radially symmetric, so it is $(\eta \mu )$, and, since $V$ and $W$ are also radially symmetric we also have radial symmetry for $V+\eta (W \ast (\eta \mu)) = V + \int_{B_R} K(\cdot, y) \mu(y) \, dy$ (the problem in which we are the most interested). Therefore, we can define $\diffEv$ like:
\begin{align}\label{def:diffEv}
\begin{split}
		\diffEv: \mathcal M_{\mathrm{rad}} (\overline{B_R}) & \longrightarrow C([0,R_v])
		\\
		\mu &\longmapsto \frac{\partial V}{\partial v} +   \frac{\partial}{\partial v} \left( \int_{B_R} K( \cdot , y) \diff \mu (y) \right),
		\end{split}
	\end{align}
	which is such that $\diffEv[\mu]$ is the volumetric derivative of $\Ev[\mu]$, where $\Ev$ is the operator:
	\begin{align}\label{def:Ev}
        \begin{split}
		    \Ev: \mathcal M_{\mathrm{rad}} (\overline{B_R}) & \longrightarrow C([0,R_v])
		\\
		        \mu &\longmapsto V +    \int_{B_R} K( \cdot , y) \diff \mu (y) .
		\end{split}
	\end{align}
We now need to show that we can extend $\mathfrak E$ in a suitable way such that
$$ 
\mathfrak E : W^{-1,1}_{\mathrm{rad}} (B_R) \longrightarrow C([0,R_v]) 
$$
is continuous. Here $W^{-1,1}_{\mathrm{rad}}(B_R)$ is the closure of $L^1_{\mathrm{rad}}(B_R)$ in $W^{-1,1}(B_R)$.

We are denoting by $\mathcal M_{\mathrm{rad}}(\overline{B_R})$ the measures invariant by rotation with respect to the origin, and similarly for $L^p_{\mathrm{rad}}(B_R)$ and $W^{-1,1}_{\mathrm{rad}}(B_R)$. In the following results we first point out that the change to volumetric variable is precisely mapping $L^p$ to $L^p$ with no weights and after that we explain how to extend the map $\diffEv$ from $\mathcal M_{\mathrm{rad}} (\overline{B_R})$ to $W^{-1,1}_{\mathrm{rad}} (B_R)$.

\begin{lemma}\label{lem:change of variables map}
	The change of variable map
	\begin{align*}
		\changevariables : L^p_{\mathrm{ rad}} \left ({ B_{R} } \right) 
		&\longrightarrow   L^p (0,R_v), && \text{where  } R_v = |B_1| R^d 
		\\
		f &\longmapsto \changevariables[f]:=g , && \text{where  } g(w) = f( (w/|B_1|)^{\frac 1 d} e_1 ).
	\end{align*}
	is an $L^p$ isometry. 
	The $L^1$ isometry is extended uniquely to an isometry in the total variation distance
	\begin{equation*}
		\changevariables : \mathcal{M}_{\mathrm{ rad}} \left (\overline{ B_{R} } \right)  \rightarrow  \mathcal{M} ([0,R_v]), \qquad \text{ where  } R_v = |B_1| R^d.
	\end{equation*}
\end{lemma}
This lemma is a simple exercise, and we omit the proof. We now prove a lemma of differentiation in volumetric variables.

\begin{lemma}\label{lem:radial}
	Assume the basic hypotheses \eqref{hyp:Basic BR}, and that the initial datum $\rho_0$ and potentials $V$ and $K$ are radially symmetric. Then, the map
	\begin{align*}
		\mathfrak E [\mu] (v) = \frac{v^{\frac{d-1}{d}}}{d |B_1|^{\frac{1}{d}}}  e_1 \cdot \Bigg(  ( \nabla V ) \left( r e_1 \right) + \Big\langle  ( \nabla_x K ) \left( r e_1 , \cdot \right)  , \mu \Big\rangle_{W_0^{1,\infty} (B_R) \times W^{-1,1} (B_R) } \Bigg) , 
	\end{align*}
    where $r=( v{|B_1|}^{-1} )^{1/d}$,
	extends \eqref{def:diffEv} and it is continuous
	\begin{equation*}
	    \mathfrak E : W^{-1,1}_{\mathrm{ rad}} (B_R) \to C([0,R_v]). 
	\end{equation*}
\end{lemma}
\begin{proof}
	We have that $V(r e_1) + \int_{B_R} K(re_1 , y) \mu (y) \, dy$ is radially symmetric. 
	We can compute its derivative as
	\begin{align*}
	    \diffEv[\mu] & = \left(\frac{v}{|B_1|}\right)^{\frac{1}{d}} e_1 \cdot \left( (\nabla V) \left(\left(\frac{v}{|B_1|}\right)^{\frac{1}{d}} e_1 \right) + \Big\langle  ( \nabla_x K ) \left( r e_1 , \cdot \right)  , \mu \Big\rangle_{W_0^{1,\infty} (B_R) \times W^{-1,1} (B_R) }   \right) \frac{1}{vd} \\
	    & = \frac{v^{\frac{d-1}{d}}}{d |B_1|^{\frac{1}{d}}} e_1 \cdot \left( (\nabla V )(re_1) + \Big\langle  ( \nabla_x K ) \left( r e_1 , \cdot \right)  , \mu \Big\rangle_{W_0^{1,\infty} (B_R) \times W^{-1,1} (B_R) } \right)
	\end{align*}
	where $r=( v{|B_1|}^{-1} )^{1/d}$. The details on the change of variables can be found, for example, in \cite{BCLR13, BBCV20}. Finally, since $\nabla_x K \in W^{1, \infty}_0 (B_R\times B_R)$
	and $\diffEv [\mu]$ is affine on $\mu$ the continuity follows.
\end{proof}

Once we have introduced the notation we recall some {\em a priori} estimates for the solutions of equation \eqref{eq:General_Mass_equation} using \Cref{lem:a_priori_estimates_Phi}. Their proofs can be found in \cite{CGV22}. 
\begin{lemma}
	If $\rho_t \in L^q (B_R)$ 
	for some $q \in [1, \infty )$ then
	\begin{equation}\label{ec:Mass_spatial_regularity}
		[M(t, \cdot )]_{C^{\frac{q-1}{q}}([0, R_v])} \leq \| \rho_t \|_{L^q (B_R)}.
	\end{equation}
	If $q= \infty$ the same holds in $W^{1, \infty} (0, R_v )$.
\end{lemma}

\begin{lemma}\label{lem:M a priori time}
	There exists a constant $C > 0$, independent of $\rho$ or $\Phi$, such that
	\begin{equation}\label{ec:bound_on_dM/dt_L2}
		\int_0^T \int_0^{R_v} \left| \frac{\partial M}{\partial t} \right|^2 \, dv \, dt \leq C \left( \int_{B_R} \Psi (\rho_0) + \| E \|_{L^{\infty} ((0,T) \times B_R)}^2 \int_0^T \int_{B_R} \rho_t (x)^2 \, dx \, dt \right).
	\end{equation}
	In particular, if $\rho_0 \in L^2 (B_R)$ and $\Psi (\rho_0 ) \in L^1 (B_R)$ then $M \in C^{\frac{1}{2}} (0, T ; L^1(0, R_v))$.
\end{lemma}

Similarly to the argument in \cite[Appendix A]{CGV22}, we can adapt DiBenedetto's theory  \cite{DiB93} so that we get the following a priori estimate

\begin{theorem}\label{th:viscosity_solutions}
	Assume the basic hypotheses \eqref{hyp:Basic BR}, and that the initial datum $\rho_0$ and potentials $V$ and $K$ are radially symmetric. Let $\rho$ be the strong solution of \eqref{eq:Fast-Diffusion_Problem_BR} and $M$ its mass. Then, for every $\varepsilon \in (0, R_v)$ we have the following interior regularity estimate: for any $T_1 > 0$ and $\varepsilon \leq v_1 <v_2 \leq R_v$ there exists $\gamma > 0$ and $\alpha \in (0,1)$ depending only on $d$, $m, \|\rho_0\|_{L^\infty(B_R)},\|\nabla V\|_{L^\infty (B_R)}, \| \nabla_x K \|_{L^{\infty}(B_R \times B_R)}, \varepsilon , T_1$, such that
	\begin{equation}\label{ec:C_alpha_estimate_in_time_and_space_Mass}
		|M (t_1, v_1) - M (t_2, v_2) | \leq \gamma \left( \frac{|v_1 - v_2 | + \|\rho_0 \|_{L^1(B_R)}^{\frac{m-1}{m+1}} |t_1 - t_2|^{\frac{1}{m+1}}}{ \varepsilon + \| \rho_0 \|_{L^1 (B_R)}^{\frac{m-1}{m+1}} T_1^{\frac{1}{m+1}} } \right)^{\alpha}
	\end{equation}
	for all $(t_i, v_i ) \in [T_1, + \infty ) \times [\varepsilon , R_v ]$. 
\end{theorem}

\subsection{Viscosity solutions}\label{sec:viscosity solutions}

We start by providing a self-contained definition of 
\eqref{eq:General_Mass_equation}. Since we cannot guarantee that $\partial M / \partial v $ is positive, it is better to simplify \eqref{eq:General_Mass_equation} expanding the second derivative and dividing by $\Phi'(\partial M / \partial v)$. Let us state what is a viscosity solution for our problem \eqref{eq:General_Mass_equation}.

\begin{definition}[Viscosity solution]
	\label{def:mass viscosity BR}
	A function $M \in C([0,T]; C((0,R_V)) \cap BV([0,R_v]))$ is a viscosity supersolution of \eqref{eq:General_Mass_equation} if, for every $t_0 > 0$, $v_0 \in (0, R_v)$ and for every $\varphi \in C^2 ((t_0 - \varepsilon , t_0 + \varepsilon ) \times (v_0 - \varepsilon, v_0 + \varepsilon ))$ such that $M \geq \varphi$, $ M(v_0) = \varphi (v_0 )$ and $\frac{\partial \varphi}{\partial v} (v) \neq 0$ for all $v \neq v_0$ it holds that
	\begin{equation*}
		\frac 1 { \Phi' \left( \frac{\partial \varphi}{\partial v} (t_0,v_0)\right)}\frac{\partial \varphi}{\partial t} (t_0, v_0 ) \ge \kappa (v_0)^2 \frac{\partial^2 \varphi}{\partial v^2} (t_0,v_0) + \kappa(v_0)^2 \frac {\frac{\partial \varphi}{\partial v} (t_0,v_0)} { \Phi' \left( \frac{\partial \varphi}{\partial v} (t_0,v_0)\right)} \diffEv\left[\changevariables^{-1}\left[\frac{\partial M}{\partial v}\right]\right] (t_0 , v_0) .
	\end{equation*}
	The corresponding definition of subsolution is made by inverting the inequalities. A viscosity solution is a function that is a viscosity sub and supersolution.
\end{definition}

\begin{remark}
	Notice that, in the viscosity formulation, we replace $\partial M / \partial v$ by the test function in the non-linear term, but we preserve it in the non-local one.
\end{remark}
In the following we prove that $M$ is a viscosity solution for the problem \eqref{eq:General_Mass_equation}.
\begin{theorem}\label{th:M is a viscosity solution}
	Assume the basic hypotheses \eqref{hyp:Basic BR}, and that the initial datum $\rho_0$ and potentials $V$ and $K$ are radially symmetric. Let $\rho$ be the strong solution of \eqref{eq:Fast-Diffusion_Problem_BR} and take $0<T< \infty$. Then, $M$ defined in \eqref{def:M} is a viscosity solution of \eqref{eq:General_Mass_equation} for $\Phi (s) = s^m$, $0<m<1$, and the drift $\diffEv\left[\changevariables^{-1}\left[\frac{\partial M}{\partial v}\right]\right]$.
\end{theorem}
\begin{proof} We proceed analogously to the constructive argument used for the existence of strong solutions to \eqref{eq:Fast-Diffusion_Problem_BR}. 
	
	\textbf{Step 1. Uniformly elliptic diffusion with fixed time-dependent drift.} Given a fixed time-dependent drift $E\in W^{1, \infty} ([0,T] \times B_R)$. We take a sequence of smooth drifts $E^{(k)}$ such that $E^{(k)} \rightarrow E$ in $W^{1, \infty} ([0,T]\times B_R)$ as in \Cref{th:Existence_FD_known_transport_(0-T)xBR}. Let us take $\rho^{(k)}$ the strong solution of \eqref{eq:FDE_Problem_+_known_transport_3} with the drift $E^{(k)}$. From the regularity theory we know that $\rho^{(k)}\in C^1 \left( (0,T); C (\overline{B_R}) \right) \cap C \left( (0, \infty) ; C^2 (\overline{B_R}) \right) \cap C \left( [0, \infty ) \times \overline{B_R} \right) $. In particular,
	\begin{equation*}
		M^{(k)} (t,v) = \int_{\widetilde{B_v}} \rho^{(k)}_t (x) \, dx,
	\end{equation*}
	is a classical solution of \eqref{eq:General_Mass_equation} and, hence, a viscosity solution thanks to the regularity of $\rho^{(k)}$.

	\textbf{Step 2. Fast diffusion with time-dependent drift.} If we take the sequence $\Phi_k$ as in \eqref{ec:Phik}, we have by \Cref{th:Existence_FD_known_transport_(0-T)xBR} that 
	\begin{equation}\label{eq:L2 convergence Mass proof}
		\rho^{(k)} \rightarrow \rho \quad \text{in} \, C ( [0, T]; L^1 (B_R))
	\end{equation}
	with $\rho$ the unique strong solution of \eqref{eq:FDE_Problem_+_known_transport} for a fixed potential $E$. We define
	\begin{equation*}
		M(t,v) = \int_{\widetilde{B_v}} \rho_t (x) \, dx.
	\end{equation*}
	Then, we have 
	\begin{equation}\label{eq:Mass convergence M^(k) to M}
		M^{(k)} \rightarrow M \quad \text{in} \, C([0,T] \times [0,R_v]). 
	\end{equation}
	Indeed, it follows from \eqref{eq:L2 convergence Mass proof} combined with
	\begin{equation*}
		|M^{(k)}(t, v) - M(t,v)| \leq \|\rho^{(k)}_t- \rho_t \|_{L^1(\widetilde{B_v})} \leq  \|\rho^{(k)}_t- \rho_t \|_{L^1(B_R)}, \quad \text{for all } 0 < t \leq T, 
	\end{equation*}
	which implies,
	\begin{equation*}
		\|M^{(k)} - M\|_{C([0,T] \times [0,R_v])} \leq  \|\rho^{(k)}_t- \rho_t \|_{L^\infty(0,T;L^1(B_R))}.
	\end{equation*}
	Next, we use the stability result 
	\cite[Chapter 3 Theorem 2]{Kat14}. The sequence $M^{(k)}$ satisfy the assumptions of the Theorem since we have \eqref{eq:Mass convergence M^(k) to M} and $1/\Phi_k' \rightarrow 1/\Phi'$ locally uniformly in $[0, Q]$ for every $Q>0$, which is easy to check using \eqref{ec:Phik}. 
	Therefore, due to the stability of viscosity solutions, $M$ is also a viscosity solution of \eqref{eq:General_Mass_equation} for $\Phi (s) = s^m$, $0<m<1$, and the drift $E\in W^{1, \infty} ([0,T] \times B_R)$.
	
	\textbf{Step 3. Fast diffusion with confinement and aggregation.} Let $\rho$ be the strong solution of the problem \eqref{eq:Fast-Diffusion_Problem_BR}. From \Cref{th:Banach_fixed_point}, $\rho$ is the unique fixed point of the mapping \eqref{eq:mapping FPT}. Then, there exists a sequence $\rho^{\{k\}}$ by Banach fixed-point iteration such that
	\begin{equation}\label{ec:rhok converges strongly Step 3 viscosity solutions}
		\rho^{\{k\}} \rightarrow \rho \quad \text{strongly in} \, C([0,T];L^1 (B_R)).
	\end{equation}
	Define
	\begin{equation*}
		M(t,v)= \int_{\widetilde{B_v}} \rho_t(x) \, dx.
	\end{equation*}
	Similarly as in Step 2, one can show
	\begin{equation}\label{eq:Mass convergence M^E to M}
		M^{\{k\}} \rightarrow M \quad \text{in} \, C([0,T] \times [0,R_v]).
	\end{equation}
	Since $\changevariables [\rho_t^{\{k\}}] = \frac{\partial M^{\{k\}} (t, \cdot )}{\partial v}$ and $\changevariables$ is an isometry from $L^1_{\mathrm{rad}} (B_R)$ to $L^1(0 , R_v)$, considering \eqref{ec:rhok converges strongly Step 3 viscosity solutions} and \eqref{eq:Mass convergence M^E to M} we also get,
	\begin{equation*}
		\frac{\partial M^{\{k\}}}{\partial v} \rightarrow \frac{\partial M}{\partial v} \quad \text{in} \, C ([0,T] ; L^1 (0, R_v)),
	\end{equation*}
	and that $\changevariables[ \rho_t ] = \frac{\partial M (t, \cdot )}{\partial v}$. 
	Taking into account \eqref{ec:rhok converges strongly Step 3 viscosity solutions} and \Cref{lem:radial}, one obtains
	\begin{equation}\label{eq:convergence derivative of the volumetric kernels FPT step}
		\diffEv [\rho_t^{\{k\}} ]  \rightarrow \diffEv [ \rho ]  \quad \text{in} \, C([0,T]\times [0, R_v]).
	\end{equation}
	Once more, from \eqref{eq:Mass convergence M^E to M} and \eqref{eq:convergence derivative of the volumetric kernels FPT step}, combined with the stability result \cite[Chapter 3 Theorem 2]{Kat14} we have that $M$ is a viscosity solution of the problem \eqref{eq:General_Mass_equation} for $\Phi (s) = s^m$, $0<m<1$, and the drift $\diffEv\left[\changevariables^{-1}\left[\frac{\partial M}{\partial v}\right]\right]$. 
\end{proof}

\begin{remark}
	In the following, every time we refer to $M$ as a viscosity solution of \eqref{eq:General_Mass_equation} we will assume that we are taking $\Phi (s) = s^m$, $0<m<1$, and the drift $\diffEv\left[\changevariables^{-1}\left[\frac{\partial M}{\partial v}\right]\right]$ unless it is stated otherwise.
\end{remark}

\subsection{Convergence of the mass solution as \texorpdfstring{$t\to \infty$}{t to infty}. Proof of 
\texorpdfstring{\Cref{prop:M^[n]_converge_uniformly}}{Theorem \ref{prop:M^[n]_converge_uniformly}}}
\label{sec:Convergence of the mass}

We already  know that the asymptotic in time of $\rho^{[n]}$ is  a distribution $\widehat{\mu}\in W^{-1,1}(B_R)$. In this subsection, we intend to link this distribution with the asymptotic of $M^{[n]}$, defined in \eqref{ec:Mass_tn_tn+1}, as $n \rightarrow \infty$. In fact, we can study the limit $M^{[n]} \rightarrow \widehat{M}$ and discuss its relationship with the limit $\widehat{\mu}$ obtained in \Cref{th:Convergence_stationary_state}.

We will start by proving that $M^{[n]} \rightarrow \widehat{M}$ converges point-wise in $[0,1] \times [0,R_v]$ uniformly over compact subsets on $[0,1] \times (0,R_v]$. We also identify $\frac{\partial \widehat{M}}{\partial v}$ with $\widehat{\mu}$.

\begin{remark}\label{rem:dM/dv identify with mu}
    We identify $\frac{\partial \widehat{M}}{\partial v}$ with $\widehat{\mu}$ in the sense of the isometry $\changevariables$ that we discussed in \Cref{lem:change of variables map}. In this way, we prove that $\frac{\partial \widehat{M}}{\partial v} = \changevariables[\widehat{\mu}]$ in $L^{\infty} (\mathcal{M} ([0,R_v]))$.
\end{remark}

We are now able to prove \Cref{prop:M^[n]_converge_uniformly}. We divide the proof in several steps. 

    \textbf{Step 1. Ascoli-Arzela over compacts of $[0,1] \times (0,R_v]$.}
	From \eqref{ec:C_alpha_estimate_in_time_and_space_Mass} we know that for all $v_1, v_2 \in [\frac{1}{k}, R_v]$ and $t_1, t_2 \in [0,1]$ for any $k \in \mathbb{Z}_{>0}$, we have the estimate,
	\begin{equation}\label{ec:Uniform_convergence_Mn_to_Minfty}
		\left|M^{[n]} (t_1, v_1) - M^{[n]} (t_2, v_2) \right| \leq C_{k} \left( |v_1 - v_2 | + |t_1- t_2 |^{\frac{1}{m+1}} \right)^{\alpha}.
	\end{equation}
Notice that we can use Di-Benedetto's regularity result \eqref{ec:C_alpha_estimate_in_time_and_space_Mass} for $M$ on the strip $[T,\infty)\times [\frac{1}{k}, R_v]$ for a fixed $T>0$ with $n$ large enough.
	Taking successive subsequences $n(k,j)$ we get that for every $k \in \mathbb{Z}_{>0}$
	\begin{equation*}
	    M^{[n(k,j)]} \rightarrow \widehat{M} \quad \text{in }  C\left ([0,1] \times \left[\tfrac 1 k , R_v\right] \right) \text{ as } j \to \infty.
	\end{equation*}
	The diagonal satisfies
	\begin{equation}\label{eq:Strong convergence mass}
		M^{[n(k,k)]} \rightarrow \widehat{M} \quad \text{in }  C_{loc}([0,1] \times (0 , R_v]).
	\end{equation}
	From here on we re-label this sequence as $M^{[n]}$.
	
	\textbf{Step 2. Extension of $\widehat{M}$ and some properties.}
	Based on the fact that $M^{[n]}(t,0)=0$ for all $n$, the extension of $\widehat{M}$ so that $\widehat{M}(t,0) = 0$ is natural. 
	Notice here that the limit of $\widehat{M}(t,v)$ as $v\to 0^+$ does not need to be $0$. 
	In any case, due to \eqref{eq:Strong convergence mass}, we have that  
	\begin{equation*}
        M^{[n]} (t,v) \rightarrow \widehat{M} (t,v) \quad \text{point-wise in } [0,1] \times [0,R_v].
    \end{equation*}
    Furthermore, because $M^{[n]}$ are non-decreasing functions on $v$, $M^{[n]}(t,0)=0$ and $M^{[n]}(t,R_v)=\| \rho_0 \|_{L^1(B_R)}$, thus we get that $M^{[n]}$ are functions with bounded variation given by $\| \rho_0 \|_{L^1(B_R)}$. As a conclusion their derivatives $\frac{\partial M^{[n]}}{\partial v}\in \mathcal{M} ([0, R_v])$ have total mass $\| \rho_0 \|_{L^1(B_R)}$. Then, from Banach-Alaoglu Theorem and \eqref{eq:Strong convergence mass}, up to a subsequence,
    \begin{equation}\label{eq:weak star convergence of the derivatives sec4}
	    \frac{\partial M^{[n]}}{\partial v}  \rightharpoonup \frac{\partial \widehat{M}}{\partial v} \quad \text{weak}-\ast \text{ in } L^{\infty}(0,1; \mathcal{M} ([0, R_v])).
	\end{equation}

	\textbf{Step 3. Identification between $\frac{\partial \widehat{M}}{\partial v}$ and $\changevariables [\widehat{\mu} ]$.}
	From \Cref{th:Convergence_stationary_state}, since $\rho_t^{[n]}$ are measures with total mass $\| \rho_0 \|_{L^1 (B_R)}$, we can use Banach-Alaouglu Theorem to deduce that, up to a subsequence
	\begin{equation*}
	    \rho^{[n]}_t \rightharpoonup \widehat{\mu} \quad \text{weak}-\ast \text{ in } L^{\infty}(0,1; \mathcal{M} (\overline{B_R})). 
	\end{equation*}
	Since $\changevariables$ is a linear isometry it follows that,
	\begin{equation}\label{eq:weak star convergence of the isometries}
	    \changevariables [\rho_t^{[n]}] \rightharpoonup \changevariables [\widehat{\mu} ] \quad \text{weak}-\ast \text{ in }  L^{\infty}(0,1;\mathcal{M} ([0, R_v])). 
	\end{equation}
	Furthermore, because $\changevariables [\rho^{[n]}_t] = \frac{\partial M^{[n]}}{\partial v} (t , \cdot)$, combining \eqref{eq:weak star convergence of the derivatives sec4} and \eqref{eq:weak star convergence of the isometries} we get that $\frac{\partial  \widehat{M} }{\partial v} = \changevariables [\widehat{\mu}]$ in $L^\infty ( \mathcal M ([0,R_v]))$ and in particular $\overline{M}$ does not depend on time. Finally, combinig \Cref{th:Convergence_stationary_state} with \Cref{lem:radial} we get that
	\begin{equation}\label{eq:Erhon to Emu}
	    \diffEv[\rho^{[n]}] \rightarrow \diffEv[\widehat{\mu}] \quad \text{in } C([0,1] \times [0,R_v]).
	\end{equation}
	Finally, from \eqref{ec:Uniform_convergence_Mn_to_Minfty} and \eqref{eq:Erhon to Emu}, combined with the stability result \cite[Chapter 3 Theorem 2]{Kat14} we have that $\widehat{M}$ is a viscosity solution of \eqref{eq:General_Mass_equation}.
	\qed

\subsection{Some examples where \texorpdfstring{$\widehat{M}$}{widehat M} is regular}\label{sec:Regularity M}

\subsubsection{When the flux \texorpdfstring{$\mathfrak E$}{E} has non-negative sign. Proof of \texorpdfstring{\Cref{th:Concentration_Phenomena_ball_viscosity_solutions}}{Theorem \ref{th:Concentration_Phenomena_ball_viscosity_solutions}}}\label{sec:Regularity M 1}

We proceed similarly to the Step 4 of the proof of \cite[Theorem 5.4]{CGV22}. We consider the notion of viscosity solution for $\widehat{M}$.
    At a point of contact $v_0\in (0,R_v)$ of a viscosity test function touching from below, we deduce
	\begin{equation*}
		- \frac{\partial^2 \varphi}{\partial v^2} (v_0) \geq \frac{1}{m} \left( \frac{\partial \varphi}{\partial v} (v_0) \right)^{2-m} \diffEv[\widehat{\mu}]  (v_0) \geq 0,
	\end{equation*}
	which implies that $\widehat{M}$ is a viscosity super-solution of $- \Delta M = 0$. Due to \cite{Ish95}, $\widehat M$ is a distributional super-solution of $- \Delta M = 0$ as well. In particular, distributional super-solutions are concave which implies $\widehat{M} \in W^{1, \infty} ([\overline{\varepsilon}, R_v - \overline{\varepsilon}])$ for all $\overline{\varepsilon} > 0$.
	
	Therefore, we can think of the right-hand side as a datum
	\begin{equation*}
		f = \frac{1}{m} \left( \frac{\partial \widehat{M}}{\partial v} (v_0) \right)^{2-m} \diffEv[\widehat{\mu}] (v_0) \in L^{\infty} (\overline{\varepsilon}, R_v - \overline{\varepsilon}).
	\end{equation*}
	Using the regularisation results from \cite{CC95} we get that $\widehat{M} \in C^{1, \alpha} (2 \overline{\varepsilon}, R_v - 2 \overline{\varepsilon})$. Furthermore, since $\diffEv [\widehat{\mu}] \in W^{1, \infty} ([0,R_v])$, $f \in C^{0, \beta} (2 \overline{\varepsilon}, R_v - 2 \overline{\varepsilon})$ for some $\beta > 0$, which guarantees $\widehat{M} \in C^{2 , \beta} ( 4 \overline{\varepsilon}, R_v - 4 \overline{\varepsilon})$.
\qed

\subsubsection{The regularity of \texorpdfstring{$\widehat{M}$}{widehat M} when \texorpdfstring{$V, K$}{V, K} are compactly supported}\label{subsec:Regularity_Minfty}
In this subsection, we will focus on discussing about the regularity of $\widehat{M}$. Here, we assume  the basic hypotheses \eqref{hyp:Basic BR}, that the initial data $\rho_0$ and the potentials $V$ and $K$ are radially symmetric, and furthermore, that  $V$ and $K$ have compact support. In order to do that, let us fix $a> 0$ and define the infimal convolution of a function $M$ as
\begin{equation*}
	M_{\varepsilon} (v) := \inf_{\bar{v} \in ( a , R_v ) } \left( M ( \bar{v} ) + \frac{| v - \bar{v} |^2}{2 \varepsilon} \right),
\end{equation*}
where $\varepsilon>0$, 
and 
\begin{equation*}
	M^{\varepsilon} (v) := \sup_{\bar{v} \in ( a , R_v ) } \left( M ( \bar{v} ) - \frac{| v - \bar{v} |^2}{2 \varepsilon} \right).
\end{equation*}
It is not difficult to show that there exists $r(\ee) \to 0$ such that
is equivalent to:
\begin{align}\label{ec:M^infty_inf-convolution}
	M_{\varepsilon} (v) = \inf_{\bar{v} \in ( a , R_v  ) \cap B_{r (\varepsilon)} (v)} \left( M ( \bar{v} ) + \frac{| v - \bar{v} |^2}{2 \varepsilon} \right),\\
	\label{ec:Minfty_sup-convolution}
	M^{\varepsilon} (v) = \sup_{\bar{v} \in ( a , R_v ) \cap B_{r (\varepsilon)} (v)} \left( M ( \bar{v} ) - \frac{| v - \bar{v} |^2}{2 \varepsilon} \right).
\end{align}
We start recalling some useful properties of $M_{\varepsilon}$. Let $M: (0 , R_v ) \rightarrow \mathbb{R}$ be bounded and lower semicontinuous in $(0 , R_v )$. It is well known that $M_{\varepsilon}$ is an increasing sequence of semiconcave functions in $(0 , R_v )$ which converges pointwise to $M$. From the fact that all the $M_{\varepsilon}$ are semiconcave, we obtain several well-known properties that we discuss in \Cref{rem:Evans Gariepy}.

For these, and some further properties of the infimal convolution, see \cite{CIL92, MO19} and \cite[Lemma A.1]{JJ12}. The next lemma is the counterpart of \cite[Lemma A.1 (iii)]{JJ12} for our setting.

\begin{lemma}\label{lem:Look_at_inf-convolution}
	Assume that $M: (a, R_v) \rightarrow \mathbb{R}$ is bounded and lower semicontinuous in $(a , R_v)$. If $M$ is a viscosity solution in the sense 
	\begin{equation*}
		- \frac{\partial^2 M}{\partial v^2} \geq  f \left( v , M, \frac{\partial M}{\partial v} \right)
	\end{equation*}
	in $(a , R_v)$, then $M_{\varepsilon}$ is a viscosity solution to 
	\begin{equation*}
		- \frac{\partial^2 M_{\varepsilon}}{\partial v^2} \geq  f_{\varepsilon} \left( v , M_{\varepsilon}, \frac{\partial M_{\varepsilon}}{\partial v} \right)
	\end{equation*}
	in $(a + r (\varepsilon), R_v - r (\varepsilon) )$, where
	\begin{equation*}
		f_{\varepsilon} (v, s , q) := \inf_{\bar{v} \in B_{r (\varepsilon)} (v)} f ( \bar{v}, s , q).
	\end{equation*}
	Respectively, if $M$ bounded and upper semicontinuous in $(a,R_v)$ is a viscosity subsolution, $M^{\varepsilon}$ is a viscosity solution in the sense,
	\begin{equation*}
		- \frac{\partial^2 M^{\varepsilon}}{\partial v^2} \leq  f^{\varepsilon} \left( v , M_{\varepsilon}, \frac{\partial M_{\varepsilon}}{\partial v} \right)
	\end{equation*}
	in $(a + r (\varepsilon), R_v - r (\varepsilon) )$, where
	\begin{equation*}
		f^{\varepsilon} (v, s , q ) := \sup_{\bar{v} \in B_{r (\varepsilon)}} f ( \bar{v}, s , q).
	\end{equation*}
\end{lemma}

\begin{remark}
    From \Cref{prop:M^[n]_converge_uniformly} we know that $\widehat{M}$ is continuous in $(0,R_v]$ and, in particular, $\widehat{M}$ is lower and upper semicontinuous in $[a, R_v]$ for every $a>0$.
\end{remark}

Another well-known result, see for example \cite[Chapter 4 Theorem 7(f)]{Kat14}, gives us convergence of the infimal and supremal convolution to the original solution.
\begin{lemma}
	Assume $M \in C^0 ([a , R_v])$, then we have that
	\begin{equation}\label{ec:uniform_convergence_inf-convolution}
		M_{\varepsilon} \nearrow M \quad \mathrm{in} \, C^0 ([a , R_v]) \quad \mathrm{as} \, \varepsilon \rightarrow 0,
	\end{equation}
	and, respectively
	\begin{equation}\label{ec:uniform_convergence_sup-convolution}
		M^{\varepsilon} \searrow M \quad \mathrm{in} \, C^0 ([a , R_v]) \quad \mathrm{as} \, \varepsilon \rightarrow 0.
	\end{equation}
\end{lemma}
Focusing on our problem, observe that $\widehat{M}$ is a viscosity solution of the equation
\begin{equation}\label{ec:Mass_equation_viscosity}
	- \frac{\partial^2 \widehat{M}}{\partial v^2} (v_0) = \frac{1}{m}  \diffEv[\widehat{\mu}] (v_0)  \left( \frac{\partial \widehat{M}}{\partial v} (v_0 ) \right)^{2-m}.
\end{equation}
Therefore, from \Cref{lem:Look_at_inf-convolution} we have that $\widehat{M}_{\varepsilon}$ is a viscosity solution in the sense
\begin{equation}\label{ec:Minftyepsilon_viscosity_supersolution}
	- \frac{\partial^2 \widehat{M}_{\varepsilon}}{\partial v^2} (v_0) \geq \frac{1}{m}  \diffEv_{\varepsilon}(v_0)  \left( \frac{\partial \widehat{M}_{\varepsilon}}{\partial v} (v_0 ) \right)^{2-m}
\end{equation}
in $(a + r (\varepsilon ), R_v - r (\varepsilon ))$, where $r (\varepsilon ) \rightarrow 0$ when $\varepsilon \rightarrow 0$ and with
\begin{equation}\label{def:diffEv_epsilon}
	\diffEv_{\varepsilon} (v_0) = \min_{v \in \overline{B_{r (\varepsilon )} (v_0)}} \diffEv[\widehat{\mu}]  (v).
\end{equation}
And respectively, $\widehat{M}^{\varepsilon}$ is a viscosity solution in the sense 
\begin{equation}\label{ec:Minftyepsilon_viscosity_subsolution}
	- \frac{\partial^2 \widehat{M}^{\varepsilon}}{\partial v^2} (v_0) \leq \frac{1}{m}  \diffEv^{\varepsilon}(v_0)  \left( \frac{\partial \widehat{M}^{\varepsilon}}{\partial v} (v_0 ) \right)^{2-m}
\end{equation}
in $(a + r (\varepsilon ), R_v - r (\varepsilon ))$ and with
\begin{equation}\label{def:diffEv^epsilon}
	\diffEv^{\varepsilon} (v_0) = \max_{v \in \overline{B_{r (\varepsilon )} (v_0)}} \diffEv[\widehat{\mu}] (v).
\end{equation}
Notice that $\diffEv_{\varepsilon}$ and $\diffEv^{\varepsilon}$ converge uniformly to $\diffEv [ \widehat{\mu} ]$.
\begin{remark}
    Let $\widetilde{\omega}$ be the modulus of continuity of $\diffEv [\widehat \mu]$. Since $\diffEv_\varepsilon(v_0) = \diffEv [\widehat \mu] (v_\varepsilon)$ where $|v_0 - v_\ee| \le r(\varepsilon)$, then 
	\begin{equation}\label{ec:Minft_Tech_assumption_1_2}
	    |\diffEv_\varepsilon(v_0) - \diffEv [\widehat \mu] (v_0)| \le \tilde{\omega}(r(\ee)).
	\end{equation}
	Notice that $\omega(\ee) = \widetilde{\omega}(r(\ee))$ is a modulus of continuity. Respectively for $\diffEv^{\varepsilon}$.
\end{remark}

Furthermore, since we are assuming $V$ and $K$ compactly supported and non-negative, then the function
\begin{equation*}
	    \Ev[\widehat{\mu}](v) = - \int_{v}^{R_v} \diffEv[\widehat \mu] (s) \diff s =  
	    \Big( V(x) + \int_{B_R} K(x,y) \diff \widehat \mu (y) \Big) \Bigg|_{   v = |B_1| |x|^d  }.
	\end{equation*}
is also compactly supported and non-negative in $[0,R_v)$. In particular, there must exist $b > 0$ such that
\begin{align}\label{ec:Minft_Tech_assumption_3}
		\begin{split}
			\Ev[\widehat{\mu}] \equiv 0 & \quad \mathrm{in} \, [R_v - b, R_v ),  \\
			\Ev[\widehat{\mu}] \geq 0  & \quad \mathrm{in} \, (0, R_v - b ). 
		\end{split}
	\end{align}
We present some extra results that follows from the remark we have just noticed.
\begin{lemma}\label{lem:Minft_Tech_remark_1}
	Let $\widehat{M}$ be a viscosity solution of \eqref{ec:Mass_equation_viscosity}. Assume $V$ and $K$ compactly supported so that $\Ev [\widehat{\mu}]$ satisfies \eqref{ec:Minft_Tech_assumption_3}, then $\widehat{M}$ is linear in the interval $[R_v - b , R_v]$. Moreover, if $\widehat{M} (0^+ ) < \widehat{M} (R_v)$, $\widehat{M}$ is non-constant in that interval.
\end{lemma}

\begin{proof}
	The result follows as a consequence of the maximum principle for semicontinuous viscosity solutions of nonlinear elliptic partial differential equations derived by Kawohl and Kutev in \cite{KK98}, see \Cref{ap:Maximum principle viscosity solutions}, where we state this version of the maximum principle.
	Let us define
	\begin{align*}
		F^{\infty}\colon& \mathbb{R} \times \mathbb{R} \times \mathbb{R}^d \times \mathbb{R}^{d \times d}  \longrightarrow  \mathbb{R} \\
		& \qquad (z,r,p,X) \qquad \mapsto \, \, \, - X - \frac{1}{m} \diffEv[\widehat{\mu}]   (z) \, p^{2-m}\,,
    \end{align*}
	We are interested on studying the fully nonlinear equation 
	\begin{equation*}
		F^{\infty}(v, \widehat{M}, D \widehat{M} , D^2 \widehat{M} ) = 0.
	\end{equation*}
	Since $\diffEv[\widehat{\mu}] \equiv 0$ in $(R_v - b , R_v)$ it follows that $\widehat{M}$ is linear on that interval. Let us now prove the assumptions \eqref{ec:Maximum_Principle_viscosity_Tech_assumption_(5)}-\eqref{ec:Maximum_Principle_viscosity_Tech_assumption_(11)} for $F^{\infty}$ to see that $\widehat{M}$ is also non-constant. \eqref{ec:Maximum_Principle_viscosity_Tech_assumption_(5)} and \eqref{ec:Maximum_Principle_viscosity_Tech_assumption_(11)} follow immediately for $F^{\infty}$. Furthermore,
	\begin{equation*}
		F^{\infty} (z, r, p, Y) - F^{\infty} (z, r, p, X) = -(Y-X),
	\end{equation*}
	and \eqref{ec:Maximum_Principle_viscosity_Tech_assumption_(8)} also follows for the modulus of ellipticity $\omega_1 (s) = s$. Finally, 
	\begin{align*}
		\left| F^{\infty} (z, r, p, X) - F^{\infty} (z, r, q, Y) \right| & = \left| Y-X + \frac{1}{m} \diffEv[\widehat{\mu}] (z) \left(   p^{2-m} -  q^{2-m} \right) \right| \\
		& \leq \left| Y-X \right| + \frac{1}{m} \Big\| \diffEv[\widehat{\mu}]  \Big\|_{L^{\infty}(0,R_v)} (2-m) \,| \xi |^{1-m}  |p-q|, 
	\end{align*}
	for some $\xi \in (0, K)$ between $p$ and $q$. Then, \eqref{ec:Maximum_Principle_viscosity_Tech_assumption_(10)} is satisfied for 
	\begin{equation*}
	    A_K=1, \qquad B_K = \frac{2-m}{m} \Big\| \diffEv[\widehat{\mu}]  \Big\|_{L^{\infty}(0,R_v)} K^{1-m},
	\end{equation*}
	with $\omega_2$ the identity once more. Therefore, because $\omega_1 \equiv \omega_2$, \eqref{ec:Maximum_Principle_viscosity_Tech_assumption_(extra)} is also trivially satisfied. As a consequence, we can apply \Cref{th:KK98} to $F^{\infty}$ and the result follows immediately from it.
\end{proof}

\begin{remark}\label{rem:monotonicity inf sup convolution}
    If $M$ is non-decreasing, then $M_{\varepsilon}$ and $M^{\varepsilon}$ are both non-decreasing. Furthermore, if we combine this with the maximum principle stated on \Cref{th:KK98} it implies that if $M$ is non-constant, then $M_{\varepsilon}$ and $M^{\varepsilon}$ are also strictly increasing.
\end{remark}

Because, $\widehat{M}$ is increasing and linear at the end of its domain we can take advantage of the following result.
\begin{lemma}\label{lem:sup- super-viscosity_solutions_linear_case}
	Assume $M$ is linear on an interval, $M(v) = A v + B$ on $(\alpha, \beta)$. Then, the infimal convolution  is $M_{\varepsilon} (v) = A v + B - \frac \ee 2 A^2$ on $(\alpha + \varepsilon A , \beta - \varepsilon A )$ (resp. $M^{\varepsilon} (v) = A v + B + \frac \ee 2 A^2$).
\end{lemma}
\begin{proof}
	We are doing the proof for the infimal convolution (analogue for the supremal case). It is easy to compute that if $v \in (\alpha + \varepsilon A , \beta - \varepsilon A )$, then, due to the convexity of $M(\bar{v}) + \frac{|v- \bar{v}|^2}{2 \varepsilon}$ the infimum is achieved at the point $\bar{v} = v- \varepsilon A$. Therefore,
	\begin{equation*}
		M_{\varepsilon} (v) = A v + B - \frac{\varepsilon}{2} A^2.
		\qedhere
	\end{equation*}
\end{proof}

Finally, before presenting the regularity result we need to see that we can integrate both, the infimal and supremal convolution. The theory we are exposing now is devoted to that goal and follows from the fact that the infimal and supremal convolution are semiconcave and semiconvex respectively. We will only discuss the case of the infimal convolution, the supremal convolution is analogous.

\begin{remark}\label{rem:Evans Gariepy}
    Semiconcavity implies that $\widehat{M}_{\varepsilon}$  is locally Lipschitz in $(a + r(\varepsilon), R_v - r(\varepsilon))$ \cite[Theorem 6.7]{EG18}. Hence, by Rademacher's Theorem \cite[Theorem 6.6]{EG18}, $D\widehat{M}_{\varepsilon}$ is an a.e. defined derivative in $(a + r(\varepsilon), R_v - r(\varepsilon))$ which coincides with the distributional derivative. Alexandrov's Theorem states that $\widehat{M}_\varepsilon$ is twice differentiable a.e., i.e. for a.e. $v_0 \in (a + r(\varepsilon), R_v - r(\varepsilon))$, there exists $\widehat{X}_{\varepsilon }(v_0)$ such that
    \begin{equation*}
        \widehat{M}_{\varepsilon} (v_0 + w) = \widehat{M}_{\varepsilon} (v_0) + D\widehat{M}_{\varepsilon} (v_0)w + \frac{1}{2} \widehat{X}_{\varepsilon }(v_0) w^2 + \Re_{v_0}(|w|^2),
    \end{equation*}
    where
    \begin{equation*}
        |\Re_{v_0} (\gamma ) | \leq \omega_{v_0} (\gamma ) \gamma^2.
    \end{equation*}
    For a proof see \cite[Theorem 14.1]{Vil09}. To avoid confusion with the distributional and measure derivative we avoid the notation $D^2 \widehat{M}_{\varepsilon}$ and we mantain the notation $\widehat{X}_{\varepsilon }$. Furthermore, there exists a sequence $\widehat{M}_{\varepsilon , j}$ regular such that
    \begin{equation}\label{eq:Minfty_epsilon a.e. convergence}
        (\widehat{M}_{\varepsilon , j} , D \widehat{M}_{\varepsilon , j}, D^2 \widehat{M}_{\varepsilon , j} ) \rightarrow (\widehat{M}_{\varepsilon }, D \widehat{M}_{\varepsilon }, \widehat{X}_{\varepsilon }) \quad \text{a.e. in } (a + r(\varepsilon), R_v - r(\varepsilon)).
    \end{equation}
    \end{remark}
    
    We can also prove the following result,
    
    \begin{lemma}\label{lem:Second derivative a.e.}
        Let $\widehat{M}_{\varepsilon}$ be a viscosity solution in the sense \eqref{ec:Minftyepsilon_viscosity_supersolution}. Then,
        \begin{equation*}
            - \widehat{X}_{\varepsilon } (v_0) \geq \frac{1}{m} \diffEv_{\varepsilon} (v_0) \left( \frac{\partial \widehat{M}_{\varepsilon}}{\partial v} (v_0) \right)^{2-m} \quad \text{a.e. in } (a+r(\varepsilon), R_v - r(\varepsilon)).
        \end{equation*}
        The result is analogous for $\widehat{M}^{\varepsilon}$, a viscosity solution in the sense \eqref{ec:Minftyepsilon_viscosity_subsolution}.
    \end{lemma}
    
    \begin{proof}
        Let us take $v_0 \in (a+r(\varepsilon), R_v - r(\varepsilon))$ and take a test function $\varphi$,
        \begin{equation*}
            \varphi (v) = \widehat{M}_{\varepsilon} (v_0) + D\widehat{M}_{\varepsilon} (v_0)(v-v_0) + \frac{1}{2} (\widehat{X}_{\varepsilon }(v_0) - \gamma )(v-v_0)^2.
        \end{equation*}
        For $|v-v_0|$ small enough
        \begin{equation*}
            \frac{1}{2} \gamma |v-v_0|^2 \geq |\Re_{v_0} (|v-v_0|^2)|,
        \end{equation*}
        and hence, in a small enough neighbourhood of $v_0$, we have that $\varphi (v) \leq \widehat{M}_{\varepsilon} (v)$. Thus, $\varphi$ touches $\widehat{M}_{\varepsilon}$ from below at $v_0$. Since $\widehat{M}_{\varepsilon}$ is a viscosity solution in the sense \eqref{ec:Minftyepsilon_viscosity_supersolution}
        \begin{equation*}
            -D^2 \varphi (v_0 ) \geq \frac{1}{m} \diffEv_{\varepsilon} (v_0) \left( D\widehat{M}_{\varepsilon} (v_0) \right)^{2-m}.
        \end{equation*}
        Computing precisely $D^2 \varphi (v_0)$ we deduce
        \begin{equation*}
            -(\widehat{X}_{\varepsilon }(v_0) - \gamma ) \geq \frac{1}{m} \diffEv_{\varepsilon} (v_0) \left( D\widehat{M}_{\varepsilon} (v_0) \right)^{2-m}.
        \end{equation*}
        Since this holds for any $\gamma > 0$ and a.e. $v_0 \in (a + r(\varepsilon ), R_v - r(\varepsilon))$, the result is proven.
    \end{proof}

Once we have presented the problem, we proceed to state the main result of the subsection about the regularity of $\widehat{M}$. We prove $\widehat{M}\in W^{2,\infty}_{loc} ((0,R_v])$ using the infimal and the supremal convolutions and all the properties we have discussed through this subsection. The first step of the proof of the Theorem where we show that $\widehat{M}_{\varepsilon}$ is also a distributional supersolution is based on \cite[Lemma 4.1]{Sil21}.

We are now able to prove \Cref{th:Regularity_of_Minfty}. Notice that 
\eqref{eq:mass ball stationary characterisation}
is written in our notation as
	\begin{equation*}
	\frac{\partial \widehat{M}}{\partial v} (v) =  \left(   \left( \frac{\partial \widehat{M}}{\partial v} (R_v ) \right)^{m-1} - \frac{1-m}{m}  \int_v^{R_v}  \diffEv (\sigma) \, d\sigma \right)^{-\frac 1 {1-m} }
\end{equation*}

\begin{proof}[Proof of \Cref{th:Regularity_of_Minfty}] We divide the proof in several steps.

\textbf{Step 1. Weak formulation.} 
    Let us take $\Theta: \mathbb{R} \rightarrow \mathbb{R}$ smooth and non-decreasing.
    Let us take the regular sequence $\widehat M_{\varepsilon, j}$ like in \eqref{eq:Minfty_epsilon a.e. convergence}. 
    Let us now take a non-negative test function $\varphi \in C_c^{\infty} ((a + r( \varepsilon ), R_v - r ( \varepsilon ))$. 
    Then, we can integrate by parts in order to get
    \begin{align*}
        &\int_{a+ r( \varepsilon)}^{R_v - r(\varepsilon)}  \left( - \frac{\partial^2 \widehat{M}_{\varepsilon,j}}{\partial v^2} (\sigma) - \frac{1}{m} \diffEv_{\varepsilon} (\sigma) \left( \frac{\partial \widehat{M}_{\varepsilon,j}}{\partial v} (\sigma) \right)^{2-m} \right)  \Theta' \left( \frac{\partial \widehat{M}_{\varepsilon, j}}{\partial v} \right) \varphi \, d\sigma  \\
        &\qquad  = 
        \int_{a+ r( \varepsilon)}^{R_v - r(\varepsilon)}  \left( \Theta \left( \frac{\partial \widehat{M}_{\varepsilon, j}}{\partial v} \right) \frac{\diff \varphi}{\diff v} - \frac{1}{m} \diffEv_{\varepsilon} (\sigma) \left( \frac{\partial \widehat{M}_{\varepsilon,j}}{\partial v} (\sigma) \right)^{2-m} \Theta' \left( \frac{\partial \widehat{M}_{\varepsilon, j}}{\partial v} \right) \varphi\right)  \, d\sigma.
    \end{align*}
    From \Cref{rem:Evans Gariepy} we have that $\left| \frac{\partial \widehat{M}_{\varepsilon , j}}{\partial v} \right|$ is uniformly bounded in the support of $\varphi$. Hence, by the Dominated Convergence Theorem.
    \begin{align*}
        \lim_{j \rightarrow \infty} &
        \int_{a+ r( \varepsilon)}^{R_v - r(\varepsilon)}  \Theta \left( \frac{\partial \widehat{M}_{\varepsilon, j}}{\partial v} \right) \frac{\diff \varphi}{\diff v} - \frac{1}{m} \diffEv_{\varepsilon} (\sigma) \left( \frac{\partial \widehat{M}_{\varepsilon,j}}{\partial v} (\sigma) \right)^{2-m} \Theta' \left( \frac{\partial \widehat{M}_{\varepsilon, j}}{\partial v} \right) \varphi  \, d\sigma \\
        & = 
        \int_{a+ r( \varepsilon)}^{R_v - r(\varepsilon)}  \Theta \left( \frac{\partial \widehat{M}_{\varepsilon}}{\partial v} \right) \frac{\diff \varphi}{\diff v} - \frac{1}{m} \diffEv_{\varepsilon} (\sigma) \left( \frac{\partial \widehat{M}_{\varepsilon}}{\partial v} (\sigma) \right)^{2-m} \Theta' \left( \frac{\partial \widehat{M}_{\varepsilon}}{\partial v} \right) \varphi  \, d\sigma.
    \end{align*}
    Notice that $\frac{\partial^2 \widehat{M}_{\ee,j}}{\partial v^2} \ge C > -\infty$, $\left| \frac{\partial \widehat{M}_{\varepsilon , j}}{\partial v} \right|$ is uniformly bounded in the support of $\varphi$ and the first and second derivatives converge a.e. Then, we use Fatou's lemma and the Dominated Convergence Theorem to deduce
    \begin{align}
        \liminf_{j \rightarrow \infty} 
         &\int_{a+ r( \varepsilon)}^{R_v - r(\varepsilon)}  \left( - \frac{\partial^2 \widehat{M}_{\varepsilon,j}}{\partial v^2} (\sigma) - \frac{1}{m} \diffEv_{\varepsilon} (\sigma) \left( \frac{\partial \widehat{M}_{\varepsilon,j}}{\partial v} (\sigma) \right)^{2-m} \right)  \Theta' \left( \frac{\partial \widehat{M}_{\varepsilon, j}}{\partial v} \right) \varphi \, d\sigma  
 \nonumber
        \\
        & \geq 
         \int_{a+ r( \varepsilon)}^{R_v - r(\varepsilon)} 
         \lim_{j \rightarrow \infty} \left( - \frac{\partial^2 \widehat{M}_{\varepsilon,j}}{\partial v^2} (\sigma) - \frac{1}{m} \diffEv_{\varepsilon} (\sigma) \left( \frac{\partial \widehat{M}_{\varepsilon,j}}{\partial v} (\sigma) \right)^{2-m} \right)  \Theta' \left( \frac{\partial \widehat{M}_{\varepsilon, j}}{\partial v} \right)  \varphi \, d\sigma \nonumber \\
         & = \int_{a+ r( \varepsilon)}^{R_v - r(\varepsilon)} 
          \left( - \widehat{X}_{\varepsilon } (\sigma) - \frac{1}{m} \diffEv_{\varepsilon} (\sigma) \left( \frac{\partial \widehat{M}_{\varepsilon}}{\partial v} (\sigma) \right)^{2-m} \right)  \Theta' \left( \frac{\partial \widehat{M}_{\varepsilon}}{\partial v} \right)  \varphi \, d\sigma \ge 0.\label{eq:Fatou viscosity}
    \end{align}
    The final inequality follows from \Cref{lem:Second derivative a.e.}.
    We have deduced that
    \begin{equation*}
        \int_{a+ r( \varepsilon)}^{R_v - r(\varepsilon)} 
          \left( \Theta \left( \frac{\partial \widehat{M}_{\varepsilon}}{\partial v} (\sigma) \right)  \frac{\diff \psi}{\diff v} - \frac{1}{m} \diffEv_{\varepsilon} (\sigma) \left( \frac{\partial \widehat{M}_{\varepsilon}}{\partial v} (\sigma) \right)^{2-m} \Theta' \left( \frac{\partial \widehat{M}_{\varepsilon}}{\partial v} (\sigma) \right) \psi \right)   \, d\sigma \geq 0.
    \end{equation*}

    \textbf{Step 2. A point-wise estimate.} 
    Since $\Ev[\widehat{\mu}]$ satisfies \eqref{ec:Minft_Tech_assumption_3} we are on the assumptions of \Cref{lem:Minft_Tech_remark_1} and therefore $\widehat{M}$ is linear and non-constant on the interval $[R_v - b , R_v]$. Then, by \Cref{lem:sup- super-viscosity_solutions_linear_case}, $\widehat{M}_{\varepsilon}$ is also linear on $[R_v - b + r(\varepsilon), R_v-r(\varepsilon)]$ and it is such that 
	$$
	    \frac{\partial \widehat{M}_{\varepsilon}}{\partial v} (v) = \frac{\partial \widehat{M}}{\partial v} (R_v) 
	\qquad \text{on }[R_v - b + r(\varepsilon), R_v - r(\varepsilon)],
	$$
    and $\mathfrak{E}_{\varepsilon} = 0$ on $[R_v-b+r(\varepsilon), R_v-r(\varepsilon)]$.
    Due to the Lebesgue differentiation Theorem, for almost every $v \in (a+r(\varepsilon) , R_v -2 r(\varepsilon))$ if we take $\psi \rightarrow \chi_{[v,R_v - 2 r(\varepsilon)]}$ we have,
    \begin{equation*}
        \Theta \left( \frac{\partial \widehat{M}_{\varepsilon}}{\partial v} (v) \right) - \Theta \left( \frac{\partial \widehat{M}_{\varepsilon}}{\partial v} (R_v) \right) \geq \int_v^{R_v} \frac{1}{m} \diffEv_{\varepsilon} (\sigma) \left( \frac{\partial \widehat{M}_{\varepsilon}}{\partial v} (\sigma) \right)^{2-m} \Theta' \left( \frac{\partial \widehat{M}_{\varepsilon}}{\partial v} (\sigma) \right) \, d\sigma. 
    \end{equation*}
    Moreover, for $\Theta$ non-decreasing,
    \begin{equation*}
        \frac{\partial \widehat{M}_{\varepsilon}}{\partial v} (v) \ge  \Theta^{-1} \left(  \Theta \left( \frac{\partial \widehat{M}_{\varepsilon}}{\partial v} (R_v) \right) + \int_v^{R_v} \frac{1}{m}  \diffEv_{\varepsilon} (\sigma) \left( \frac{\partial \widehat{M}_{\varepsilon}}{\partial v} (\sigma) \right)^{2-m}\Theta'\left( \frac{\partial \widehat{M}_{\varepsilon}}{\partial v} (\sigma) \right) \, d\sigma \right).
    \end{equation*}
     We take a sequence $\Theta_m$ such that $s^{2-m}\Theta_m'(s) \to 1-m$ uniformly over compacts and $\Theta_m(1) = -1$.
    This corresponds to $\Theta (s) = -s^{m-1}$ so $\Theta^{-1} (s) = (-s)^{-\frac 1 {1-m}}$. Therefore, we get the estimate
    \begin{align}
    \label{eq:M_eps derivative bound}
         \frac{\partial \widehat{M}_{\varepsilon}}{\partial v} (v) 
         &\ge \left(   \left( \frac{\partial \widehat{M}}{\partial v} (R_v ) \right)^{m-1} - \frac{1-m}{m}  \int_v^{R_v}  \diffEv_{\varepsilon} (\sigma) \, d\sigma \right)^{-\frac 1 {1-m} } .
    \end{align}
    Proceeding analogously for $\widehat{M}^{\varepsilon}$ we deduce the converse formula
    \begin{align}
    \label{eq:M^eps derivative bound}
         \frac{\partial \widehat{M}^{\varepsilon}}{\partial v} (v) 
         &\le \left(   \left( \frac{\partial \widehat{M}}{\partial v} (R_v ) \right)^{m-1} - \frac{1-m}{m}  \int_v^{R_v}  \diffEv^{\varepsilon} (\sigma) \, d\sigma \right)^{-\frac 1 {1-m} } .
    \end{align}
    
\textbf{Step 3. Formula for $\widehat{M}$.}
Let $a < v_1 \le v_2 < R_v$. Then, due to \eqref{eq:M_eps derivative bound} we have that for $\varepsilon$ small enough
	$$
	    \widehat{M}_\varepsilon(v_2) - \widehat{M}_\varepsilon(v_1) \ge \int_{v_1}^{v_2} \left(   \left( \frac{\partial \widehat{M}}{\partial v} (R_v ) \right)^{m-1} - \frac{1-m}{m}  \int_v^{R_v}  \diffEv_{\varepsilon} (\sigma) \, d\sigma \right)^{-\frac 1 {1-m} } \diff v 
	$$
	Due to the uniform convergence of $\widehat M_\varepsilon$ and $\mathfrak E_\ee$ we deduce that
	$$
	    \widehat{M}(v_2) - \widehat{M}(v_1) \ge \int_{v_1}^{v_2} \left(   \left( \frac{\partial \overline{M}}{\partial v} (R_v ) \right)^{m-1} - \frac{1-m}{m}  \int_v^{R_v}  \diffEv (\sigma) \, d\sigma \right)^{-\frac 1 {1-m} } \diff v 
	$$
	Arguing conversely for \eqref{eq:M^eps derivative bound} we deduce the other inequality. 
	Due to the bounds and known regularity of $\mathfrak E$, it follows that $\widehat{M} \in W^{2,\infty}_{loc} ((0, R_v])$ and 
	$$
	    \frac{\partial \widehat{M}}{\partial v} (v) =  \left(   \left( \frac{\partial \widehat{M}}{\partial v} (R_v ) \right)^{m-1} - \frac{1-m}{m}  \int_v^{R_v}  \diffEv (\sigma) \, d\sigma \right)^{-\frac 1 {1-m} }.
	$$
    This concludes the proof.
\end{proof}

\section{The problem in \texorpdfstring{$\mathbb{R}^d$}{Rd}}\label{sec:Rd}
In this section, we will focus on the problem \eqref{ec:The_equation}. The strategy we follow to study \eqref{ec:The_equation} consists in extending the results of the problem in the ball $B_R$ considering the kernel $K_{\eta}(x,y) = \eta(x) W(x-y) \eta (y)$, with $\eta \in C^{\infty}_c (\mathbb{R}^d)$ a cut-off. 

We recall that in order to to be able to use the results we have obtained for the ball $B_R$, the potential $V$ needs to fulfil the hypothesis \eqref{eq:V has no flux}. Therefore, for each one of the different problems on the ball we need to modify slightly the potential $V$. We take this modification $V_R$ such that 
\begin{align*} 
&|V_R| \leq C |V|, |\nabla V_R | \leq C | \nabla V |,\text{ and } | D^2  V_R | \leq C |D^2 V| \text{ in } B_R, \\
&\nabla V_R \cdot \nu = 0 \text{ on }\partial B_R,\\
&V_R = V \text{ in } B_{R-1}
\end{align*} 
for some constant $C$ that does not depend on $R$.

\subsection{Short-time well-posedness. Proof of \texorpdfstring{\Cref{th:locally_strong_solution_Rn}}{Theorem \ref{th:locally_strong_solution_Rn}}}\label{sec:short-time well-posedness Rd}
First we are going to construct locally-strong solutions in the whole space and study some of their properties. Let us start defining the notion of locally-strong solution.
\begin{definition}[Locally-strong solution]\label{def:locally-strong solution}
	A measure $\rho$ is said to be a locally-strong solution of the problem \eqref{ec:The_equation} in $(0,T) \times \mathbb{R}^d$ if it is a distributional solution such that
	\begin{enumerate}
		\item $\rho^m \in L^2 (0, T; H^2_{loc} (\mathbb{R}^d))$;
		\item $\partial_t \rho \in L^2(0,T; L^2_{loc} (\mathbb{R}^d))$. 
	\end{enumerate}
\end{definition}

Next, we make same remarks about the problem.

\begin{remark}
	From assumption \eqref{eq:V_quadratic} one gets immediately by the mean value theorem that
	\begin{equation}\label{ec:Assumptions_on_V_whole_space}
		V(x) \leq C (1+|x|^{2 }) .
	\end{equation}
	From assumption \eqref{eq:W_quadratic} we also get in a similar way that,
	\begin{equation}\label{ec:Assumptions_on_W_whole_space}
		W(x-y) \leq C (1+|x|^{2 })(1+|y|^{2 }) \quad \text{for every} \,  x, y \in \mathbb{R}^d. 
	\end{equation}
\end{remark}
\begin{remark}\label{rem:V_R unif bdd}
    From the construction we have done of $V_R$, \eqref{eq:V_quadratic}, \eqref{eq:Delta V and W bounded}, and \eqref{ec:Assumptions_on_V_whole_space} also hold for $V_R$ with a constant that does not depend on $R$.
\end{remark}
Before starting with the proof of \Cref{th:locally_strong_solution_Rn} we need to show some \textit{a priori} estimates. First we deal with the $p$-th order moment
$$
    \mathbf{m_p} (\rho ) := \int_{\mathbb{R}^d} |x|^p \rho
$$
\begin{remark}
    From assumption \eqref{eq:p moment bdd for rho0} one gets immediately that $\mathbf{m_2} (\rho_0)$ is bounded. Using Hölder inequality we obtain that,
    \begin{align*}
        \mathbf{m_2} (\rho_0) \leq \left( \int_{\Rd} |x|^{\frac{2}{1-m}} \rho_0 \right)^{1-m} \| \rho_0 \|^m_{L^1 (\Rd)}.
    \end{align*}
\end{remark}
\begin{lemma}[The second moment is bounded for finite time]
	Assume \eqref{ec:Free_energy_is_bounded_below} and \eqref{eq:p moment bdd for rho0}. Then, the second moment is uniformly bounded in the sense that
	\begin{equation}\label{ec:Second_moment_uniformly_bounded_above}
		\mathbf{m_2} (\rho^{R, \eta, \ast}_t ) \leq  \left(\mathbf{m_2} (\rho_0 ) + \sup_R \mathcal{F}_{R, 1} [\rho_0] - \underline{\mathcal{F}_0} \right)  e^t.
	\end{equation}
\end{lemma}
\begin{proof}
	Let us take the time derivative. We apply integration by parts, Holder and Young's inequality successively in order to get, 
	\begin{align*}
		\frac{d}{dt} \mathbf{m_2} (\rho^{R, \eta, \ast}_t ) & = \frac{d}{dt} \int_{\mathbb{R}^d} |x|^2 \rho^{R, \eta, \ast}_t = \frac{d}{dt} \int_{B_R} |x|^2 \rho^{R, \eta}_t = \int_{B_R} |x|^2 \partial_t \rho^{R, \eta}_t \\
		& = \int_{B_R} |x|^2 \nabla \cdot \left( \rho^{R, \eta}_t \left[ \nabla \left( \frac{m}{m-1} (\rho^{R, \eta}_t)^{m-1} + V_R + \int_{B_R} K_{\eta}(x,y) \rho^{R, \eta}_t(y) \, dy \right) \right] \right) \, dx \\
		& = -2 \int_{B_R} x \rho^{R, \eta}_t \left[ \nabla \left( \frac{m}{m-1} (\rho^{R, \eta}_t)^{m-1} + V_R + \int_{B_R} K_{\eta}(x,y) \rho^{R, \eta}_t(y) \, dy \right) \right] \, dx \\
		& \leq 2 \left( \int_{B_R} |x|^2 \rho^{R, \eta}_t \right)^{\frac{1}{2}} \\
		&\qquad \times \left( \int_{B_R} \rho^{R, \eta}_t \left| \nabla \left( \frac{m}{m-1} (\rho^{R, \eta}_t)^{m-1} + V_R + \int_{B_R} K_{\eta}(x,y) \rho^{R, \eta}_t(y) \, dy \right) \right|^2 \, dx \right)^{\frac{1}{2}} \\
		& \leq  \int_{B_R} |x|^2 \rho^{R, \eta}_t  \\
		&\qquad + \left( \int_{B_R} \rho^{R, \eta}_t \left| \nabla \left( \frac{m}{m-1} (\rho^{R, \eta}_t)^{m-1} + V_R + \int_{B_R} K_{\eta}(x,y) \rho^{R, \eta}_t(y) \, dy \right) \right|^2 \, dx \right).
	\end{align*}
	Integrating in time between $0$ and $t$ and taking into account \eqref{eq:flow of the free energy} we have that, 
	\begin{equation*}
	    \mathbf{m_2} (\rho^{R, \eta, \ast}_t ) - \mathbf{m_2} (\rho^{R, \eta, \ast}_0 ) \leq \int_0^t \mathbf{m_2} (\rho^{R, \eta, \ast}_s ) \, ds + \mathcal{F}_{R, \eta} [\rho_0^{R, \eta, \ast}] - \mathcal{F}_{R, \eta} [\rho_t^{R, \eta, \ast}],
	\end{equation*}
	which gives us the bound,
	\begin{equation*}
		\mathbf{m_2} (\rho^{R, \eta, \ast}_t ) \leq \int_0^t \mathbf{m_2} (\rho^{R, \eta, \ast}_s ) \, ds +   \left( \mathbf{m_2} (\rho_0 ) + \sup_{\eta} \sup_R \mathcal{F}_{R, \eta} [\rho_0] - \underline{\mathcal{F}_0}  \right) .
	\end{equation*}
	Finally, applying Grönwall's inequality we recover \eqref{ec:Second_moment_uniformly_bounded_above}.
\end{proof}
The previous Lemma only ensures boundedness of $\mathbf{m_2} (\rho)$. However, taking advantage of this result we can also get an \textit{a priori} estimate for the $p$-order moment
for $p = \frac{2}{1-m} > 2$.
\begin{lemma}\label{lem:p-order moment}
	Assume \eqref{ec:Free_energy_is_bounded_below} and \eqref{eq:p moment bdd for rho0}-\eqref{eq:W_quadratic}. Then, the $p$-moment for $p = \frac{2}{1-m}$ is uniformly bounded from above if we pick the sequence $\eta = \eta_j$
	\begin{equation}\label{ec:p_moment_uniformly_bounded_above}
		\mathbf{m_p} (\rho^{R, \eta_j, \ast}_t) \leq A \exp{ (B t ) }, \qquad \forall t \in [0,T]
	\end{equation}
	where $A$ and $B$ are constants that depend only on $m$, $T$, the constants in \eqref{ec:Free_energy_is_bounded_below}-\eqref{eq:W_quadratic}, $\| \rho_0\|_{L^1 (\Rd)}$, $\|\rho_0 \|_{L^{\frac{1}{1-m}}(\Rd)}$, $\mathbf{m_p} (\rho_0)$, $\| \nabla \eta_1 \|_{L^{\infty} (\Rd)}$,
    and the right-hand side of \eqref{ec:Second_moment_uniformly_bounded_above}.
\end{lemma}

The proof of this result is on \Cref{ap:a priori estimates Rd}.
Once we have an \textit{a priori} estimate for a $p > 2$ order moment we are ready to prove \Cref{th:locally_strong_solution_Rn}.

\begin{proof}[Proof of \Cref{th:locally_strong_solution_Rn}]
We divide the proof of the Theorem into 5 steps.

	\textbf{Step 1. Compactness in $R$.} 
	If we use the $L^p$ estimate \eqref{ec:Lp_estimate}, we get that,
	\begin{equation}\label{ec:Lp_uniform_in_Rn}
		 \| \rho^{R, \eta_j}_t \|_{L^p (\Rd)} \leq e^{ C(p,T) } \| \rho_0 \|_{L^p (\Rd)}, \qquad \forall t \in [0,T].
	\end{equation}
    where
    $$
    C(p,T) = \sup_{R,j} \frac{p-1}{p} \int_0^T \left(\| \Delta V_{R} \|_{L^{\infty} (\mathbb{R}^d)} + \left\| \Delta \int_{\Rd}K_{\eta_j} (\cdot,y) \rho_t^{R , \eta_j} (y) \, dy \right\|_{L^{\infty}( \mathbb{R}^d)} \right) \, dt .
    $$ 
	The term $\| \Delta V_R \|_{L^{\infty}(\Rd)}$ is uniformly bounded in $R$ by the assumption \eqref{eq:Delta V and W bounded} as we point out in \Cref{rem:V_R unif bdd}. Furthermore, since 
	$$
		| \nabla \eta_j (x) | \leq j^{-1} \| \nabla \eta_1 \|_{L^{\infty} (\Rd)} \chi_{B_j} (x) 
		\qquad \text{and} \qquad | \Delta \eta_j (x) | \leq j^{-2} \| \Delta \eta_1 \|_{L^{\infty}(\Rd)} \chi_{B_j} (x),
	$$ we can bound $\| \Delta \int_{\Rd}K_{\eta_j} (\cdot,y) \rho^{R , \eta_j} (y) \, dy \|_{L^{\infty}( \mathbb{R}^d)}$ uniformly in $j$.
	\begin{align*}
	    & \left| \Delta \int_{\Rd}K_{\eta_j} (x,y) \rho_t^{R , \eta_j} (y) \, dy \right| \\
	    & \qquad \leq \left| \int_{\Rd} \left( \Delta \eta_j (x) W(x-y)  + 2 \nabla \eta_j (x) \cdot \nabla W(x-y)  + \eta_j(x) \Delta W(x-y)   \right) \rho_t^{R, \eta_j} (y) \, dy \right|  \\
	    & \qquad \leq C \Big( \int_{\Rd} \| \Delta \eta_1 \|_{L^{\infty}(\Rd)}(j^{-2} + 1) (1 + |y|^2) \rho_t^{R, \eta_j} (y) \, dy  \\
	    & \qquad \qquad + \int_{\Rd} \| \nabla \eta_1 \|_{L^{\infty}(\Rd)} ( j^{-1} + 1 )  \rho_t^{R, \eta_j} (y) \, dy + \int_{\Rd} \| \eta_1 \|_{L^{\infty} (\Rd)}  \rho_t^{R, \eta_j} (y) \, dy \Big) \\
	    & \qquad \leq C(\| \eta_1 \|_{W^{2, \infty} (\Rd)} ) \left( \| \rho_0 \|_{L^1 (\Rd)} + \mathbf{m_2} (\rho_t^{R, \eta_j}) \right).
	\end{align*}
	Then, since the second moment is uniformly bounded due to \eqref{ec:Second_moment_uniformly_bounded_above}, we have that the sequence $\rho^{R, \eta_j, \ast}$ is uniformly bounded in $L^p ( (0,T) \times \mathbb{R}^d)$ for every $p \geq 1$. 
 
  Due to the uniform bound in $L^\infty(0,T; L^p(\Rd))$ we can do a diagonal argument in time to prove that for every $j$ fixed there exists a sequence $R_{ij} \nearrow \infty$ as $i \to \infty$ such that
	\begin{equation*}
	    \rho^{R_{ij} , \eta_{j}, \ast} \rightharpoonup \rho^{\infty , \eta_j} \quad \text{in } L^p_{loc}(0, \infty ; L^p( \mathbb{R}^d)).
	\end{equation*}
	In \Cref{ap:a priori estimates Rd} we prove uniform \textit{a priori} estimates. 
	For every $\omega > 0$ and $i$ such that $R_{ij} > \omega, j$, we deduce:
    \begin{itemize} 
    \item from \eqref{eq:rho H1 Rd}, $\nabla \rho^{R_{ij}, \eta_j}$ is uniformly bounded in $L^2((0,T) \times B_{\omega})$,
    \item from \eqref{eq:rhom H1 Rd}, $(\rho^{R_{ij}, \eta_j} )^m$ is uniformly bounded in $L^2 ((0,T) ; H^1(B_{\omega}))$,
    \item from \eqref{eq:d/dt rho L2 Rd}, $\partial_t \rho^{R_{ij}, \eta_j}$ is uniformly bounded in $L^2((0,T) \times B_{\omega} )$, and 
    \item from \eqref{eq:Delta rhom L2 Rd}, $\Delta (\rho^{R_{ij} , \eta})^m$ is uniformly bounded in $L^2 ((0,T) \times B_{\omega})$.
    \end{itemize}
	
	\textbf{Step 2. A subsequence converges to a locally-strong solution for a fixed cut-off $\eta_j$.}
    Thanks to all the compactness results obtained in Step 1 we can prove convergence to a locally-strong solution.
	
	\textbf{Step 2a. Convergence by compactness.}
	The sequence $\rho^{R_{ij}, \eta_j}$ is uniformly bounded in $H^1((0,T) \times B_{\omega})$ by Step 1. Since $H^1 ((0,T) \times B_{\omega})$ is compactly embedded in $C^{\frac{1}{2}} (0,T ; H^1(B_{\omega}))$, using Ascoli-Arzela Theorem and a diagonal argument in $T$ and $\omega$, for $j$ fixed up to a further subsequence $R_{ij}$, the convergence is such that
	\begin{equation}\label{eq:Cloc strong convergence}
		\rho^{R_{ij}, \eta_j, \ast} \rightarrow \rho^{\infty, \eta_j} \quad \text{in }  C_{loc} ([0,\infty) ; L^2_{loc} (\Rd)) \qquad \text{ as } i \to \infty,  \text{ and for each $t$ locally }  a.e. \text{ in } \Rd.
	\end{equation}
	In particular, we also have, $\rho^{\infty , \eta_j} \in H^1_{loc} ((0,\infty) \times \Rd)$.  Moreover, by the Dominated Convergence Theorem we have
	\begin{equation*}
	    (\rho^{R_{ij}, \eta_j, \ast})^m \rightarrow (\rho^{\infty , \eta_j})^m \quad \text{strongly in } C_{loc} ((0,\infty) ;L^2_{loc}( \Rd )).
	\end{equation*}

	Furthermore, since $(\rho^{R_{ij}, \eta_j})^m$ is uniformly bounded in $L^2 (0,T; H^2 (B_{\omega}))$ for any $T$ and $\omega$, due to  up to a further subsequence in $i$ on $R_{ij}$; by Banach-Alaoglu Theorem and the previous characterisation of the limit
	\begin{equation*}
	    ((\rho^{R_{ij}, \eta_j, \ast})^m) \rightharpoonup (( \rho^{\infty , \eta_j})^m) \quad \text{weakly in } L^2_{loc} (0,\infty; H^2_{loc} (\Rd)).
	\end{equation*}
	Therefore, $(\rho^{\infty, \eta_j})^m \in L^2_{loc} (0,\infty;H^2_{loc} (\Rd))$.
 
 	\textbf{Step 2b. Distributional solution.}  
	We now have all the ingredients to see that for all $\phi \in X = \left\lbrace \phi \in C_c^{\infty} ([0,T] \times \Rd ) : \phi (T) = 0 \right\rbrace$,
	\begin{equation*}
		\int_{\Rd} \rho_0 \phi (0)    + \int_0^T \int_{\Rd} \rho^{R_{ij}, \eta_j}_t \frac{\partial \varphi}{\partial t}  = \int_0^T \int_{\Rd} \nabla( (\rho^{R_{ij}, \eta_j}_t)^m ) \nabla \phi  + \int_0^T \int_{\Rd} \rho^{R_{ij}, \eta_j}_t E_t^{R_{ij}, \eta_j} \nabla \phi  
	\end{equation*}
	converges to
	\begin{equation*}
		\int_{\Rd} \rho_0 \phi (0)  + \int_0^T \int_{\Rd} \rho^{\infty, \eta_j}_t \frac{\partial \phi}{\partial t}  = \int_0^T \int_{\Rd} \nabla ( ( \rho^{\infty, \eta_j} )^m ) \nabla \phi  + \int_0^T \int_{\Rd} \rho^{\infty, \eta_j}_t E_t^{\infty, \eta_j} \nabla \phi  .
	\end{equation*}

	\textbf{Step 3. Extension $\eta_j \rightarrow 1$.} All the bounds discussed in the Step 1 are uniform on $\supp \eta_j$. Therefore, in the same way we have just done in Step 2, taking successive subsequences $ \eta_{j}$ we get that there exists $\rho^{\infty , 1} \in L^{\infty}(0,T; L^p (\mathbb{R}^d))$ such that, 
	\begin{equation}\label{eq:Local convergence proof}
		\rho^{\infty , \eta_{j}} \rightarrow \rho^{\infty , 1} \quad \text{in } C_{loc}([0, \infty ) ; L^2_{loc}(\mathbb{R}^d)).
	\end{equation}
	Additionally, $\rho^{\infty , 1}$ is a locally-strong solution of the problem \eqref{ec:The_equation} in $(0,T) \times \mathbb{R}^d$ in the case $\eta \equiv 1$ for every $T >0$.

	\textbf{Step 4. Convergence of the free energy with a fixed cut-off $\eta_j$.} In this step we claim that we have the convergence stated at \eqref{eq:Convergence free energy ball to Rd}. Our strategy to prove that the limit holds is to control the tails of each one of the three terms.
	
	\textbf{Step 4a. Convergence of the diffusive term of the energy.} 
	First we control the tails. Let $\sigma > 0$.  We can use Hölder inequality to obtain, if $m > \frac{1+d - \sqrt{2d+1}}{d}$ 
	\begin{align*}
	    \int_{\Rd \backslash B_{\sigma}} \left( \rho_t^{R_{ij} , \eta_j, \ast} \right)^m &
     \leq  \left( \int_{\Rd \backslash B_{\sigma}}  |x|^{\frac{2}{1-m}} \rho_t^{R_{ij}, \eta_j, \ast} \right)^m \left( \int_{\Rd \backslash B_{\sigma}} \frac{1}{|x|^\frac{2m}{(1-m)^2}} \right)^{1-m} \\
	    & \le C_m  \mathbf{m_{ \frac{2}{1-m} }} \left( \rho_t^{R_{ij}, \eta_j, \ast}  \right)^m \sigma^{- \frac{2m}{(1-m)} + d(1-m)}.
	\end{align*}
	The $\frac{2}{1-m}$-th moment is bounded due to \eqref{ec:p_moment_uniformly_bounded_above} and the exponent on $\sigma$ is negative with our choice of $m$ (notice also this range is connect with the reverse Hardy-Littlewood inequality \Cref{rem:sharp m})
    Thus, by the Fatou's lemma the limit also satisfies, for every $\sigma < \kappa < \infty$ and $t \in [0,T]$ due to local-a.e. convergence
	\begin{equation*}
		 \int_{ B_\kappa \setminus B_\sigma }   \left( \rho^{\infty, \eta_{j}}_t \right)^m \leq \omega_{T,m} (\sigma^{-1}) \rightarrow 0,
	\end{equation*}
    Now letting $\kappa \to \infty$ we obtain  that the formula also holds replacing $B_\kappa$ by $\Rd$.
	Eventually we recover that, for each $\sigma > 0$ we have
	\begin{align*}
		 \limsup_{i \rightarrow \infty} &\sup_{t \in [0,T]} \left| \int_{\Rd} \left( \rho^{R_{ij}, \eta_{j}, \ast}_t \right)^m - \int_{\mathbb{R}^d} \left( \rho_t^{\infty , \eta_j} \right)^m \right| \\
            &\le \limsup_{i \rightarrow \infty}
        \sup_{t \in [0,T]}\left|  \int_{B_{\sigma}} \left( \rho^{R_{ij}, \eta_{j}, \ast}_t \right)^m  - \int_{B_{\sigma}} \left( \rho^{\infty , \eta_j}_t \right)^m\right| +  2 \omega_{T,m} (\sigma^{-1})
	\end{align*}
    The first term is $0$ because we have proven that the limit $i \rightarrow \infty$ is such that
	\begin{equation*}
	    \left( \rho^{R_{ij}, \eta_{j}, \ast} \right)^m  \rightarrow \left( \rho^{\infty , \eta_j} \right)^m \quad \text{in }  L_{loc}^2(0,\infty; H^1_{loc} (\mathbb{R}^d)).
	\end{equation*}
	The second term vanishes as $\sigma \to \infty$ we deduce that the $\limsup$ is actually a limit, and its value is $0$.
	
	\textbf{Step 4b. Convergence of the confinement term of the energy.} 
    Analogously as before we write
	\begin{equation}\label{ec:Free_energy_convergence_b1}
    \begin{aligned} 
		\limsup_{i \rightarrow \infty} &\sup_{t \in [0,T]}\left|   \int_{\Rd} V_{R_{ij}} \rho^{R_{ij}, \eta_{j}, \ast}_t  - \int_{\Rd} V \rho^{\infty , \eta_j} \right| \\
  &\le   \limsup_{i \rightarrow \infty} \sup_{t \in [0,T]} \left |   \int_{B_{\sigma}} V_{R_{ij}} \rho^{R_{ij}, \eta_{j}, \ast}_t  - \int_{B_{\sigma}} V \rho^{\infty , \eta_j} \right | \\
  &\qquad + \limsup_{i \to \infty} \sup_{t \in [0,T]}\int_{\mathbb{R}^d \backslash B_{\sigma}} V_{R_{ij}} \rho^{R_{ij}, \eta_{j}, \ast}_t + \sup_{t \in [0,T]} \int_{\mathbb{R}^d \backslash B_{\sigma}} V \rho^{\infty, \eta_{j}}_t
    \end{aligned}
	\end{equation}
    The first term is $0$ since we know that taking a sequence $i \rightarrow \infty$ we have that
	\begin{equation*}
		\rho^{R_{ij}, \eta_{j}, \ast} \rightarrow \rho^{\infty , \eta_j} \quad \text{in }  L^2 (0, T; H^1_{loc} (\mathbb{R}^d)).
	\end{equation*}
	and from \eqref{ec:Assumptions_on_V_whole_space} we also know that $V \in L^{\infty}_{loc} (\mathbb{R}^d)$. 
	From \eqref{ec:Assumptions_on_V_whole_space}, for $\sigma$ big enough, $V_{R_{ij}}$ is such that $V_{R_{ij}}(x) \leq C (1 + |x|^{2})$, and, in particular, we get that,
	\begin{align*}
		\int_{\mathbb{R}^d \backslash B_{\sigma}} V_{R_{ij}} \rho^{R_{ij}, \eta_{j}, \ast}_t & \leq  \frac{C}{\sigma^{p - 2}} \left( \| \rho_0 \|_{L^1} + \mathbf{m_p} ( \rho^{R_{ij}, \eta_{j}, \ast}_t ) \right).
	\end{align*}
    In particular we use $p = \frac{2}{1-m}>2$, where we can control the moment.
    Using Fatou this uniform estimate for $t\in [0,T]$ also holds for the last term in \eqref{ec:Free_energy_convergence_b1}.
    Letting $\sigma \to \infty$ we deduce that the the first $\limsup$ in \eqref{ec:Free_energy_convergence_b1} vanishes.
	
	\textbf{Step 4c. Convergence of the aggregation term of the free energy.} Due to the local convergence \eqref{eq:Cloc strong convergence}, since the function $\eta_j(x) W(x-y) \eta_j (y)$ is supported on a compact set in $\Rd \times \Rd$ we observe that,
	\begin{align*}
	 \lim_{i \rightarrow \infty} & \sup_{t \in [0,T]}	\Bigg| \int_{\mathbb{R}^d} \rho^{R_{ij}, \eta_j, \ast}_t (x) \eta_j (x) \int_{\mathbb{R}^d}  W(x-y) \eta_j(y) \rho^{R_{ij}, \eta_j, \ast}_t (y) \, dy \, dx \\
	&\qquad  - \int_{\mathbb{R}^d} \rho^{\infty, \eta_j}_t (x) \eta_j (x) \int_{\mathbb{R}^d}  W(x-y) \eta_j(y) \rho^{\infty, \eta_j}_t (y) \, dy \, dx \Bigg|= 0.
	\end{align*}

	\textbf{Step 5. Convergence of the free energy when $\eta_j \rightarrow 1$.} 
 We now devote this step to prove \eqref{eq:Convergence free energy ball to Rd 2}. Steps 4a and 4b are analogous and therefore, up to a subsequence in $j$,
	\begin{equation*}
	    \lim_{j \rightarrow \infty} \int_{\mathbb{R}^d} \left( \rho^{\infty, \eta_{j}}_t \right)^m = \int_{\mathbb{R}^d} \left( \rho^{\infty, 1}_t \right)^m
	\qquad \text{ and } \qquad 
	    \lim_{j \rightarrow \infty} \int_{\Rd} V \rho_t^{\infty, \eta_j} = \int_{\Rd} V \rho^{\infty , 1}.
	\end{equation*}
    In order to study the convergence of the aggregation term of the free energy, let us define the symmetric bilinear form
    $$
        \mathfrak W ( u , v ) = \int_\Rd \int_\Rd u(x) W(x-y) v(y) dy dx .
    $$
    Notice that the aggregation term on $\mathcal F_{\infty, \eta} [\rho]$ is $ \frac 1 2  \mathfrak W ( \rho \eta_j , \rho \eta_j )$ (even for $\eta = 1$).
	Now, we would like to prove that the limit
	\begin{equation}\label{eq:Free energy aggregation convergence}
	    \lim_{j \rightarrow \infty} \mathfrak W (\rho^{\infty , \eta_j}_t \eta_j , \rho^{\infty , \eta_j}_t \eta_j ) = \mathfrak W (\rho^{\infty , 1}_t , \rho^{\infty , 1}_t) 
	\end{equation}
	holds. Let us proceed to check the claim. 
	\begin{equation*}
	    \left| \mathfrak W (\rho^{\infty , \eta_j}_t \eta_j , \rho^{\infty , \eta_j}_t \eta_j ) - \mathfrak W (\rho^{\infty , 1}_t , \rho^{\infty , 1}_t) \right| \leq \left| \mathfrak W (\rho^{\infty , \eta_j}_t \eta_j , \rho^{\infty , \eta_j}_t \eta_j - \rho^{\infty, 1}_t) \right| + \left| \mathfrak W (\rho^{\infty, \eta_j}_t \eta_j - \rho^{\infty , 1}_t , \rho^{\infty , 1}_t) \right|
	\end{equation*}
	We focus on the first term of the RHS since the computations for the second one work in an analogous way. 
    From \eqref{ec:Assumptions_on_W_whole_space} and taking $p=\frac{2}{1-m}$, we have that,
	\begin{align*}
	    & \left|\mathfrak W \left ( \rho^{\infty, \eta_j}_t  \eta_j , \rho^{\infty, \eta_j}_t  \eta_j  - \rho^{\infty, 1}_t \right) \right|  \\
	    & \qquad \leq C \left| \int_{\Rd}  (1+|x|^2) \rho^{\infty, \eta_j}_t (x) \eta_j (x) dx \right| \left|  \int_{\Rd} (1+|y|^2) \left[ \rho^{\infty, \eta_j}_t (y) \eta_j (y) - \rho^{\infty, 1}_t (y) \right] \, dy \right| 
     \end{align*} 
     The first integral on the RHS is bound due to \eqref{ec:Second_moment_uniformly_bounded_above}. We will prove that the second integral tends to zero as $j \to \infty$, with a similar argument as before. For any $\sigma > 0$ we have that
     \begin{equation} 
     \label{eq:Free energy convergence W last step}
     \begin{aligned} 
	     \limsup_{j\to \infty} &\sup_{t \in [0,T]} \left|  \int_{\Rd} (1+|y|^2) \left[ \rho^{\infty, \eta_j}_t (y) \eta_j (y) - \rho^{\infty, 1}_t (y) \right] \, dy \right| \\
        & \le \limsup_{j \to \infty } 
        \sup_{t \in [0,T]}
        \int_{B_{\sigma}} (1+|y|^2) \left| \rho^{\infty, \eta_j}_t (y) \eta_j (y) - \rho^{\infty,1}_t (y) \right| \, dy \\
	    & \qquad + \limsup_{j\to \infty} 
        \sup_{t \in [0,T]}
        \int_{\Rd \backslash B_{\sigma}} (1 + |y|^2)  \rho^{\infty, \eta_j}_t (y) \eta_j (y) +  
        \sup_{t \in [0,T]}
        \int_{\Rd \backslash B_{\sigma}} (1 + |y|^2)  \rho^{\infty,1}_t (y)  \, dy.
	\end{aligned}
    \end{equation}
    The first $\limsup$ on the RHS is $0$ due to the local convergences. For the second term we bound
    $$
	    \int_{\Rd \backslash B_{\sigma}} (1 + |y|^2)  \rho^{\infty, \eta_j}_t (y) \eta_j (y) \, dy
     \le  \frac{1+\sigma^2}{1+ \sigma^p} \left( \| \rho_0 \|_{L^1} + \mathbf{m_p} (\rho^{\infty, \eta_j}_t) \right).
    $$
    And we use $p = \frac{2}{1-m}>2 $ because this moment is uniformly bounded.
    The last can also be bounded similarly by Fatou. Letting $\sigma \to \infty$ we recover that the RHS of \eqref{eq:Free energy convergence W last step} is $0$. With this we conclude with the proof of \eqref{eq:Free energy aggregation convergence}.
	\end{proof}

\subsection{Long-time asymptotics. Proof of \texorpdfstring{\Cref{th:mu infty Rd}}{Theorem \ref{th:mu infty Rd}}}\label{sec:Long time asymptotics Rd}
	
	Let us now take $t_n$ a sequence of times such that $t_n \rightarrow \infty$. We define the sequence of functions
	\begin{align}\label{def:rho infty 1 n}
		\begin{split}
			\rho^{\infty, 1, [n]}\colon& [0,1] \times \Rd  \longrightarrow  \mathbb{R} \\
			& \qquad (t,x) \quad \mapsto \, \, \rho^{\infty , 1} (t + t_n, x).
		\end{split}
	\end{align}
        In this subsection, we are going to study the limit when $n \rightarrow \infty$ of this sequence. Furthermore, we would also like to prove that there is no mass escaping through infinity when we take the limit in time. In $\mathbb{R}^d$, the free-energy of the FDE, $\partial_t u = \Delta u^m$ in the range $0 < m < 1$, is not bounded from below and is such that $u(t) \rightarrow 0$ as $t \rightarrow \infty$. This implies the necessity of requiring further assumptions. In fact, it suffices $V$ to not be critical in the sense of constants, i.e. for some $\varepsilon \in (0,1)$ 
	\begin{align}
	\begin{split}\label{ec:Free_energy_bounded_below}
		& \inf_{\substack{\rho \in \mathcal{M}_{ac} ( \mathbb{R}^d), \\ \| \rho \|_{L^1(\Rd)}=\|\rho_0 \|_{L^1(\Rd)}}} \left( \frac{1}{m-1} \int_{\mathbb{R}^d} \rho^m + \int_{\mathbb{R}^d} V \rho + \frac{1}{2} \int_{\mathbb{R}^d} \int_{\mathbb{R}^d} K(x,y) \rho (y) \rho(x) \, dy \, dx \right) \\
		& \qquad \geq  \inf_{\substack{\rho \in \mathcal{M}_{ac} ( \mathbb{R}^d), \\ \| \rho \|_{L^1(\Rd)}=\|\rho_0 \|_{L^1(\Rd)}}} \left( \frac{1}{m-1} \int_{\mathbb{R}^d} \rho^m + (1- \varepsilon ) \int_{\mathbb{R}^d} V \rho  \right) \\
		& \qquad > - \infty.
	\end{split}
\end{align}
From this assumption we can deduce the following result stated in \cite{CGV22},
\begin{lemma}[Boundedness of $\|V\rho\|_{L^1 (\mathbb{R}^d)}$, \cite{CGV22}]
	Assume \eqref{ec:Free_energy_bounded_below}. Then, there exists a constant $C>0$ such that
	\begin{equation}\label{ec:sup_V_bounded}
		\int_{\mathbb{R}^d} \rho^m + \int_{\mathbb{R}^d} V \rho \leq C (1 + \mathcal{F} [\rho] ),
	\end{equation}
\end{lemma}
Furthermore, by an argument from \cite{CD18} extended in \cite{CGV22}, we can give a sufficient condition on $V$ so that \eqref{ec:Free_energy_bounded_below} holds.

\begin{theorem}[Family of admissible potentials, \cite{CD18, CGV22}]
	Assume that, for some $\alpha \in (0, m)$ we have that
	\begin{equation}\label{eq:A way to bound the free energy}
		\chi_V = \sum_{j=1}^{\infty} 2^{jn} V(2^j)^{- \frac{\alpha}{1-m}} < \infty.
	\end{equation}
	Then, \eqref{ec:Free_energy_bounded_below} holds for any $\varepsilon \in (0,1)$.
\end{theorem}
	Then, we can prove \Cref{th:mu infty Rd} about the asymptotic in time of $\rho^{\infty,1}$.

	\begin{proof}[Proof of \Cref{th:mu infty Rd}]
	We divide the proof in several steps.
	
	\textbf{Step 1. Convergence to the limit.}
	Let us take the sequence of functions defined in \eqref{def:rho infty 1 n}. Fix $R_{ij}$ and $\eta_j$ and take $\Omega \subseteq B_{R_{ij}}$ bounded. From the definition of the space $W^{-1,1}$ we have the \textit{a priori} estimate,
	\begin{align*}
		\| u \|_{W^{-1,1}( \Omega )}  
		= \inf_{ \substack{ u|_\Omega = \diver F \\ F \in L^1(\Omega) }} \|F \|_{L^1(\Omega)} 
		=\inf_{ \substack{ {u}|_\Omega = \diver F \\ F \in  L^1(\Omega) }} \|F \|_{L^1(\Omega)}  
		\leq \inf_{ \substack{ u = \diver F \\ F \in  L^1(B_{R_{ij}}) }} \|F \|_{L^1(B_{R_{ij}})}  
		\leq \| u \|_{W^{-1,1}(B_{R_{ij}})}.
	\end{align*} 
	Therefore, using the linearity of the extension and a similar strategy as in the proof of \Cref{th:Convergence_stationary_state}, we have that,
	\begin{align*}
		\| \rho^{R_{ij}, \eta_j, [n], \ast}_t - \rho^{R_{ij}, \eta_j, [n], \ast}_s \|_{W^{-1,1}( \Omega )} & \leq \| \rho^{R_{ij}, \eta_j, [n]}_t - \rho^{R_{ij}, \eta_j, [n]}_s \|_{W^{-1,1}(B_{R_{ij}})} \\
        &\leq \int_s^t \left\| \frac{\partial \rho^{R_{ij}, \eta_j, [n]}_{\sigma}}{\partial \sigma} \right\|_{W^{-1,1}(B_{R_{ij}})} \, d \sigma \\
		& \leq \| \rho_0 \|_{L^1 (B_{R_{ij}})}^{\frac{1}{2}} \left( \mathcal{F}_{R_{ij}, \eta_j} [ \rho^{R_{ij}, \eta_j }_{t_n} ] - \mathcal{F}_{R_{ij}, \eta_j} [\rho^{R_{ij}, \eta_j}_{t_n +1}] \right)^{\frac{1}{2}} | t - s |^{\frac{1}{2}}.
	\end{align*}
	Furthermore, from \eqref{eq:Convergence free energy ball to Rd} taking the limit  $i\rightarrow \infty$ we get that,
	\begin{align*}
		\| \rho^{\infty, \eta_j, [n]}_t - \rho^{\infty,\eta_j, [n]}_s \|_{W^{-1,1}( \Omega )} 
		& \leq \| \rho_0 \|_{L^1 (B_{R_{ij}})}^{\frac{1}{2}} \left( \mathcal{F}_{\infty, \eta_j} [ \rho^{\infty, \eta_j }_{t_n} ] - \mathcal{F}_{\infty, \eta_j} [\rho^{\infty,\eta_j}_{t_n +1}] \right)^{\frac{1}{2}} | t - s |^{\frac{1}{2}},
	\end{align*}
	and, from the further convergence \eqref{eq:Convergence free energy ball to Rd 2}, taking $j \rightarrow \infty$ we obtain,
	\begin{align*}
		\| \rho^{\infty, 1, [n]}_t - \rho^{\infty,1, [n]}_s \|_{W^{-1,1}( \Omega )} 
		& \leq \| \rho_0 \|_{L^1 (B_{R_{ij}})}^{\frac{1}{2}} \left( \mathcal{F}_{\infty, 1} [ \rho^{\infty, 1 }_{t_n} ] - \mathcal{F}_{\infty, 1} [\rho^{\infty,1}_{t_n +1}] \right)^{\frac{1}{2}} | t - s |^{\frac{1}{2}}.
	\end{align*}
	Take $\Omega = B_1$. By the Ascoli-Arzelá Theorem and the $L^p$ bounds, there exists a subsequence in $n$ such that $\rho^{\infty, 1, [n]} \to \widehat\mu^{\{1\} }$ in $C([0,1];W^{-1,1}(\Omega))$. Hence, 
	\begin{equation*}
		\|\widehat\mu_t^{\{1\}} - \widehat\mu_s^{\{1\}} \|_{W^{-1,1} (\Omega)} \le \liminf \| \rho^{\infty, 1, [n]}_t - \rho^{\infty, 1, [n]}_s \|_{W^{-1,1}( \Omega )} = 0,
	\end{equation*}
	and $\widehat\mu^{\{1\}}$ does not depend on time in $\Omega$.
	
	Similarly, we construct a further subsequence that converges in $B_2$, and so on $\widehat\mu^{\{2\}} \in W^{-1,1} (B_2) $ (i.e. time independent). Using test functions in $W^{1,\infty}_0(B_1)$ we observe that $\widehat\mu^{\{2\}}|_{B_1} = \widehat\mu^{\{1\}}$. We proceed inductively for all $B_k$.  
    Lastly, by a diagonal argument we get a distribution $\widehat\mu$ and a final subsequence $\rho^{\infty, 1, [n]}$ that converges to it in $C([0,1];W^{-1,1}_{loc}(\Rd))$.
    
    \textbf{Step 2. The limit is also weak-$\ast$.} Since, $\rho_t^{\infty , 1 , [n]}$ are measures with total mass $\| \rho_0 \|_{L^1(\mathbb{R}^d)}$, from Banach-Alaoglu Theorem, up to a subsequence, we have that,
	\begin{equation*}
	    \rho_t^{\infty , 1, [n]} \rightharpoonup \widehat\mu \quad \text{weak}-\ast \text{ in } L^{\infty}(0,1; \mathcal{M}(\mathbb{R}^d)). 
	\end{equation*}

	\textbf{Step 3. There is no mass escaping through infinity.} Let us compute. Taking advantage of \eqref{ec:sup_V_bounded}  we have that,
    \begin{align*}
        \| \rho_0 \|_{L^1 (\Rd)} - \int_{B_R} \widehat\mu (x) \, dx  = \int_{\Rd \backslash B_R} d\widehat\mu \leq \frac{1}{\inf_{ x \in \Rd \backslash B_R} V (x)} \int_{\Rd} V \, d\widehat\mu \leq \frac{C}{\inf_{x \in \Rd \backslash B_R} V (x)} = \omega (R^{-1}) \rightarrow 0
    \end{align*}
    since $\inf_{x \in \Rd \backslash B_{\sigma}} V (x) \rightarrow \infty$ when $\sigma \rightarrow \infty$.
\end{proof}

\subsection{Viscosity solutions} 
The goal of this subsection is to study the mass problem in the whole space $\mathbb{R}^d$ taking advantage from viscosity solutions and all the theory developed on \cref{sec:Mass_eq}. In order to do that, we define analogously $\diffEv_{\infty, \eta}$, $\Ev_{\infty, \eta}$ and $\changevariables_{\infty}$ for the whole space $\mathbb{R}^d$ and we change the notation to $\diffEv_{R, \eta}$, $\Ev_{R, \eta}$ and $\changevariables_{R}$ for the case we studied on \cref{sec:Mass_eq}, where we extend by $0$ if we need to. Furthermore, in the same way we did on \cref{sec:Mass_eq}, we assume $V$, $W$ and $\eta$ radially symmetric.

We study the problem,
\begin{equation}\label{eq:General_Mass_equation_whole_space}
	\begin{dcases}
		\frac{\partial M}{\partial t}  = \kappa (v)^2 \frac{\partial }{\partial v} \left( \frac{\partial M}{\partial v} \right)^m + \kappa (v)^2 \frac{\partial M}{\partial v} \diffEv_{\infty,\eta} \left[ \changevariables^{-1}_{\infty} \left[ \frac{\partial M}{\partial v} \right] \right] (t, v)   , & t,v > 0,\\
        M(t,0) = 0, & t > 0 \\
		M(0,v)  = \int_{\widetilde{B_v}} \rho_0 (x) \, dx, & v > 0
	\end{dcases}
\end{equation}
for which we need a notion of viscosity solution. Since at this point we cannot guarantee that $\partial M / \partial v $ is positive, it is better to simplify \eqref{eq:General_Mass_equation_whole_space} expanding the second derivative and multiplying by $\frac{1}{m} \left( \frac{\partial M}{\partial v} \right)^{1-m}$.
\begin{definition}
	A function $M \in C([0,T];C((0, \infty)) \cap BV([0,\infty))$ is a viscosity supersolution of \eqref{eq:General_Mass_equation_whole_space} if, for every $t_0 > 0$, $v_0 \in (0, \infty)$ and for every $\varphi \in C^2 ((t_0 - \varepsilon , t_0 + \varepsilon ) \times (v_0 - \varepsilon, v_0 + \varepsilon ))$ such that $M \geq \varphi$, $ M(v_0) = \varphi (v_0 )$ and $\frac{\partial \varphi}{\partial v} (v) \neq 0$ for all $v \neq v_0$ it holds that
	\begin{equation*}
		\tfrac 1 { m} \left( \tfrac{\partial \varphi}{\partial v} (t_0,v_0)\right)^{1-m} \tfrac{\partial \varphi}{\partial t} (t_0, v_0 ) \ge \kappa (v_0)^2 \left( \tfrac{\partial^2 \varphi}{\partial v^2} (t_0,v_0) + \tfrac {1}{m}\left(\tfrac{\partial \varphi}{\partial v} (t_0,v_0) \right)^{2-m}  \diffEv_{\infty,\eta} \left[ \changevariables^{-1}_{\infty} \left[ \tfrac{\partial M}{\partial v} \right] \right] (t_0, v_0) \right) .
	\end{equation*}
	The corresponding definition of subsolution is made by inverting the inequalities. A viscosity solution is a function that is a viscosity sub and supersolution.
\end{definition}
\begin{remark}
	Notice that, in the viscosity formulation we do not replace $\partial M / \partial v$ by the test function in the non-local term.
\end{remark}

In order to obtain viscosity solutions of the problem \eqref{eq:General_Mass_equation_whole_space}, we base our strategy on extending the solutions of the problem in the ball \eqref{eq:General_Mass_equation}. We define,
\begin{equation*}
	M^{R, \eta} (t, v) = \int_{\widetilde{B_v}} \rho_t^{R, \eta},
\end{equation*}
which we now extend to the whole space in the natural way,
\begin{equation}\label{ec:Tilde_M[n]R}
	M^{R, \eta , \ast} (t,v)     =\int_{\widetilde{B_v}} \rho_t^{R, \eta, \ast}
    =
    \begin{dcases}
		M^{R, \eta} (t, v) & \text{if } v \leq R_v, \\
		\| \rho_0 \|_{L^1(B_R)} & \text{if } v > R_v.
	\end{dcases} 
\end{equation}
Once we have presented the problem we are ready to state the main result of this subsection.

\begin{proposition}\label{prop:extend mass solution to Rd}
    Assume all the hypothesis from \Cref{th:locally_strong_solution_Rn}. Assume furthermore that $\rho_0$, $V$ and $W$ are radially symmetric and \eqref{ec:Free_energy_bounded_below}. Thus, for $j$ fixed there exists a sequence $R_{ij} \rightarrow \infty$  and $M^{\infty , \eta_j} \in C ([0, \infty) \times (0, \infty))$ such that:
    \begin{itemize}
        \item $\diffEv_{R_{ij}, \eta_j}[\rho^{R_{ij}, \eta_j, \ast}] \rightarrow \diffEv_{\infty, \eta_j}[\rho^{\infty , \eta_j}]$ in $C_{loc} ([0, \infty) \times (0, \infty))$,
        \item $\frac{\partial M^{\infty,\eta_j}}{\partial v} = \changevariables_{\infty} [\rho^{\infty,\eta_j}]$ belongs to $C_{loc}([0,\infty); L^2_{loc} [0, \infty))$,
        \item $\sup_{[0,T] \times [0, \infty)} |M^{R_{ij} , \eta_j, \ast} (t,v) - M^{\infty, \eta_j} (t,v)| \rightarrow 0$,
        \item $M^{\infty, \eta_j}$ is a viscosity solution of \eqref{eq:General_Mass_equation_whole_space}.
    \end{itemize}
    Then, if we let $\eta_j \nearrow 1$ there exists $M^{\infty, 1} \in C([0, \infty) \times (0, \infty))$ such that:
    \begin{itemize}
        \item $\diffEv_{\infty, \eta_{j}}[\rho^{\infty, \eta_{j}}] \rightarrow \diffEv_{\infty, 1}[\rho^{\infty , 1}]$ in $C_{loc} ([0, \infty) \times (0, \infty))$,
        \item $\frac{\partial M^{\infty,1}}{\partial v} = \changevariables_{\infty} [\rho^{\infty,1}]$ belongs to $C_{loc}([0,\infty); L^2_{loc} [0, \infty))$,
        \item $\sup_{[0,T] \times [0, \infty)} |M^{\infty , \eta_{j}} (t,v) - M^{\infty, 1} (t,v)| \rightarrow 0$,
        \item $M^{\infty, 1}$ is a viscosity solution of \eqref{eq:General_Mass_equation_whole_space} for $\eta \equiv 1$.
    \end{itemize}
\end{proposition}

\begin{proof}
  	 \textbf{Step 1. Convergence on $R$.} Let us fix $\eta_j$. Due to the $C^{\alpha}$ regularity \eqref{ec:C_alpha_estimate_in_time_and_space_Mass} and the uniform boundedness of $\|  M^{R_{ij}, \eta_j, \ast} \|_{L^{\infty}([0,1] \times \mathbb{R})}$ we know that, up to a further subsequence on $i$, by the Ascoli-Arzela Theorem, for any $T>0$ and $v_1 , v_2 > 0$,
	\begin{equation}\label{eq:M loc convergence R}
		M^{R_{ij}, \eta_j, \ast} \rightarrow  M^{\infty, \eta_j} \quad \text{in }  C([0,T] \times [v_1, v_2]).
	\end{equation}
    On \Cref{th:locally_strong_solution_Rn} we show that, up to a subsequence in $i$, we have that,
	\begin{equation*}
	    \rho^{R_{ij}, \eta_j, \ast} \rightarrow \rho^{\infty , \eta_j} \quad \text{in } C_{loc}([0,\infty); L^2_{loc} (\Rd)).
	\end{equation*}
	From here it follows immediately that,
        \begin{equation}\label{ec:E[rho] loc convergence R}
            \diffEv_{R_{ij}, \eta_j}[\rho^{R_{ij}, \eta_j, \ast}] \rightarrow \diffEv_{\infty, \eta_j} [\rho^{\infty, \eta_j}] \quad \text{in }  C_{loc} ( [0,\infty) \times \mathbb{R}^d).
        \end{equation}
	From \eqref{eq:M loc convergence R} and \eqref{ec:E[rho] loc convergence R}, applying the stability result \cite[Chapter 3 Theorem 2]{Kat14} we recover that $M^{\infty, \eta_j}$ is a viscosity solution of the mass equation \eqref{eq:General_Mass_equation_whole_space} for every cut-off $\eta_j \in C_c^{\infty} (\mathbb{R}^d)$.

	Furthermore, from \Cref{th:locally_strong_solution_Rn} we know that
	\begin{equation*}
	    \rho^{R_{ij} , \eta_j , \ast} \rightarrow \rho^{\infty, \eta_j} \quad \text{in } C_{loc}([0,\infty); L^2_{loc} (\Rd)).
	\end{equation*}
	Then, since $\changevariables_{\infty} [\rho_t^{R_{ij}, \eta_j , \ast}] = \frac{\partial M^{R_{ij}, \eta_j, \ast} (t, \cdot )}{\partial v}$ and $\changevariables_{\infty}$ is an isometry from $L^p_{rad}(\mathbb{R}^d)$ to $L^p (0, \infty )$ we also get that,
	\begin{equation*}
	    \frac{\partial M^{R_{ij}, \eta_j, \ast} (t, \cdot )}{\partial v} \rightarrow \frac{\partial M^{\infty, \eta_j} (t, \cdot )}{\partial v} \quad \text{in } C_{loc}([0,\infty) ;L^2_{loc}( \mathbb{R}^d)),
	\end{equation*}
	and that $\changevariables_{\infty} [\rho_t^{\infty, \eta_j}] = \frac{\partial M^{\infty, \eta_j} (t, \cdot )}{\partial v}$ in $C_{loc} ([0, \infty); L^2_{loc} [0,\infty))$.
	
	\textbf{Step 2. Asymptotics in space.} With all the theory we have already developed we can study the asymptotic behaviour in space. We compute the following,
	\begin{align*}
	    \sup_{[0,T] \times [0, \infty)} &| M^{R_{ij}, \eta_{j}, \ast} - M^{\infty, 1} |  = \sup_{[0,T] \times [0, \infty)} \left| \int_{\widetilde{B_v}} \rho_t^{R_{ij}, \eta_{j}, \ast} - \rho^{\infty ,\eta_j} \right| \\
	    & \leq \sup_{[0,T]} \int_{B_{\sigma}} \left| \rho_t^{R_{ij}, \eta_{j}, \ast} - \rho^{\infty ,\eta_j} \right| +  \frac{1}{\sigma^2}  \int_{\Rd \backslash B_{\sigma}} |x|^2  \rho_t^{R_{ij}, \eta_{j}, \ast} +  \frac{1}{\sigma^2}  \int_{\Rd \backslash B_{\sigma}} |x|^2 \rho^{\infty ,\eta_j} .
	\end{align*}
	Due to the \eqref{eq:Local convergence from BR to Rd}, the first term in the RHS converges to $0$ for every $B_{\sigma}$ when we take the limit in $i$. For the second term we are taking advantage of \eqref{ec:Second_moment_uniformly_bounded_above}. Since $\mathbf{m_2} (\rho^{R_{ij}, \eta_j, \ast})$ is uniformly bounded the second term converges to $0$ when $\sigma \rightarrow \infty$. For the third term we use again the same argument combined with Fatou's Lemma.
	
	\textbf{Step 3. Extension $\eta \rightarrow 1$.} All the bounds discussed in the previous step are uniform on $\mathrm{supp} \, \eta_j$. Then, in the same way we have just done, there exists $M^{\infty , 1}$ such that, up to a subsequence on $j$, 
	\begin{equation*}
		M^{\infty , \eta_j} \rightarrow M^{\infty , 1} \quad \text{in }  C_{loc}([0,T] \times (0, \infty )).
	\end{equation*}
	We now claim that
	\begin{equation*}
	    \diffEv_{\infty , \eta_j}[\rho^{\infty, \eta_j}] \rightarrow \diffEv_{\infty , 1}[\rho^{\infty, 1}] \quad \text{in } C_{loc} ([0,T] \times [0,\infty)).
	\end{equation*}
	In order to prove this, it is enough to show that for a fixed $r>0$, when $j \rightarrow \infty$, we have that,
	\begin{equation*}
	    \eta_j (r e_1) \frac{\partial}{\partial r} \int_{\Rd} \rho_t^{\infty, \eta_j} (y) W( r e_1 - y) \eta_j (y) \, dy \longrightarrow \frac{\partial}{\partial r}  \int_{\Rd} \rho_t^{\infty, 1} (y) W(re_1 - y) \, dy.
	\end{equation*}
	Let us take $j$ big enough so that $\eta_j (r e_1)=1$ and compute,
	\begin{align*}
	    & \left| \int_{\Rd} \nabla W (r e_1 - y) \left[ \rho^{\infty,1}_t (y) -  \rho^{\infty, \eta_j}_t (y) \eta_j (y) \right] \, dy \right| \\
	    & \qquad \leq \left| \int_{\Rd} \nabla W (r e_1 - y) \rho^{\infty, \eta_j}_t (y) \left[ 1 -   \eta_j (y) \right] \, dy \right| + \left| \int_{\Rd} \nabla W (r e_1 - y) \left[ \rho^{\infty,1}_t (y) -  \rho^{\infty, \eta_j}_t (y) \right] \, dy \right|.
	\end{align*}
	Let us control the first term,
	\begin{align*}
	    \left| \int_{\Rd} \nabla W (r e_1 - y) \rho^{\infty, \eta_j}_t (y) \left[ 1 -   \eta_j (y) \right] \, dy \right| & \leq \left| \int_{\Rd \backslash B_{\frac{j}{2}}} \nabla W(r e_1-y) \rho^{\infty , \eta_j}_t (y) \, dy \right| \\
	    & \leq (1+r) \int_{\Rd \backslash B_{\frac{j}{2}}}  (1+|y|) \rho^{\infty , \eta_j}_t (y) \, dy\\
	    & \leq (1+r) \frac{1+ j / 2 }{1+( j / 2)^2} \int_{\Rd \backslash B_{\frac{j}{2}}} (1+|y|^2) \rho^{\infty, \eta_j}_t (y) \, dy.
	\end{align*}
	For the second one, we again introduce an intermediate ball of radius $B_\sigma$
	\begin{align*}
	    & \limsup_{j \rightarrow \infty} \sup_{t \in [0,T]} \left| \int_{\Rd} \nabla W (r e_1 - y) \left[ \rho^{\infty,1}_t (y) -  \rho^{\infty, \eta_j}_t (y) \right] \, dy \right| \\
	    & \qquad \leq (1+r) \limsup_{j \rightarrow \infty} \sup_{t \in [0,T]} \Big( \int_{B_{\sigma}} (1+|y|) \left| \rho_t^{\infty,1} (y) - \rho_t^{\infty, \eta_j} (y) \right| \, dy \\
	    & \qquad \quad + \frac{1+\sigma}{1+ \sigma^2} \int_{\Rd \backslash B_{\sigma}} |y|^2 \rho^{\infty,1}_t (y) \, dy + \frac{1+\sigma}{1+ \sigma^2} \int_{\Rd \backslash B_{\sigma}} (1+|y|^2) \rho^{\infty, \eta_j}_t (y) \, dy \Big) \\
	    & \qquad \leq (1+r) \limsup_{j \rightarrow \infty} \sup_{t \in [0,T]}  \int_{B_{\sigma}} (1+|y|) \left| \rho_t^{\infty,1} (y) - \rho_t^{\infty, \eta_j} (y) \right| \, dy + 2 \omega_T (\sigma^{-1}) \to 0 
	\end{align*}
    as $\sigma \to \infty$ since,
	as above, the first term is zero by local convergence, and the second one converges to zero as $\sigma \to \infty$.

	Then, due to the stability result \cite[Chapter 3 Theorem 2]{Kat14}, for any $T>0$ we have that $M^{\infty ,1}$ is a viscosity solution of the problem \eqref{eq:General_Mass_equation_whole_space} in $[0,T]\times \mathbb{R}^d$ for $\eta \equiv 1$.

	We also know from the previous step that $\changevariables_{\infty} [\rho_t^{\infty, \eta_j}] = \frac{\partial M^{\infty, \eta_j} (t, \cdot )}{\partial v}$. 
	Then, since $\changevariables_{\infty}$ is an isometry, from \eqref{eq:Local convergence from BR to Rd 2} we also get that,
	\begin{equation*}
	    \frac{\partial M^{\infty, \eta_j } (t, \cdot )}{\partial v} \rightarrow \frac{\partial M^{\infty, 1} (t, \cdot )}{\partial v} \quad \text{in } C_{loc}([0,\infty ) ; L^2_{loc}(0,\infty)),
	\end{equation*}
	and that $\changevariables_{\infty} [\rho_t^{\infty, 1}] = \frac{\partial M^{\infty, 1} (t, \cdot )}{\partial v}$ in $C_{loc}([0,\infty); L^2_{loc}([0,\infty))$. 
    Since $M(0,v) = M_0(v)$, with a diagonal argument we can show that, up to a further subsequence, the limit $M^{\infty , 1}$ is such that $M^{\infty , 1} \in  C_{loc}([0, \infty) \times [0, \infty ))$.
	
	Furthermore, through the control of the tails we have uniform convergence in $v$ since, introducing balls of radius $B_\sigma$, we have
	\begin{align*}
	    \limsup_j \sup_{[0,T] \times [0, \infty)}& | M^{\infty, \eta_{j}} - M^{\infty, 1} |  = \sup_{[0,T] \times [0, \infty)} \left| \int_{\widetilde{B_v}} \rho_t^{\infty, \eta_{j}} - \rho^{\infty ,1} \right| \\
	    & \leq \limsup_j \sup_{[0,T]} \int_{B_{\sigma}} \left| \rho_t^{\infty, \eta_{j}} - \rho^{\infty ,1} \right| +  \frac{1}{\sigma^2}  \int_{\Rd \backslash B_{\sigma}} |x|^2  \rho_t^{\infty, \eta_{j}} + \frac{1}{\sigma^2}  \int_{\Rd \backslash B_{\sigma}} |x|^2 \rho^{\infty ,1} \to 0
	\end{align*}
    as $\sigma \to \infty$.
	and we use the local convergence and the boundedness of the second moment in $[0,T]$.
\end{proof}

\subsection{Convergence of the mass equation as \texorpdfstring{$t \rightarrow \infty$}{t to infty} in the whole space}

Let us recall $\rho^{\infty,1, [n]}$ defined in \eqref{def:rho infty 1 n}. Now we define,
\begin{equation}\label{def:M infty 1 n}
    M^{[n]} (t,v) = \int_{\widetilde{B_v}} \rho_t^{\infty,1,[n]} (x) \, dx,
\end{equation}
which are solutions to the mass equation \eqref{eq:General_Mass_equation_whole_space} for the time interval $[t_n, t_n + 1]$ after a translation to $[0,1]$. 

On the assumptions of \Cref{th:locally_strong_solution_Rn} we already know that the asymptotic in time of $\rho^{\infty, 1, [n]}$ is a distribution $\widehat\mu \in W^{-1,1}_{loc} (\mathbb{R}^d)$. We can study the limit $M^{[n]} \rightarrow \widehat M$ and discuss its relationship with the limit $\widehat\mu$.

\begin{proposition}\label{prop:limit Mn properties}
    Assume all the hypothesis from \Cref{th:locally_strong_solution_Rn}. Assume furthermore that $\rho_0$, $V$ and $W$ are radially symmetric, \eqref{ec:Free_energy_bounded_below}, \eqref{eq:Growth of V compare to W}, and \eqref{eq:Control of the tails}. Let the sequence $M^{[n]}$ be defined as in \eqref{def:M infty 1 n}, then there exists $\widehat M \in C([0,1] \times (0, \infty))$, such that, if we define $\widehat M(t,0)=0$ for all $t \in [0,1]$, then the following properties hold:
    \begin{itemize}
        \item The function $\widehat M$ is non-decreasing,
        \item up to a subsequence, $M^{[n]} \rightarrow \widehat M$ in $C_{loc}([0,1] \times (0, \infty))$ and point-wise in $[0,1]\times [0, \infty)$,
        \item $\widehat M$ does not depend on time,
        \item $\frac{\partial \widehat M}{\partial v} = \changevariables_{\infty} [\widehat\mu ]$ in $L^{\infty} (\mathcal{M}([0, \infty)))$,
        \item $\diffEv_{\infty,1}[\rho^{\infty, 1, [n]}] \rightarrow \diffEv_{\infty, 1}[\widehat\mu]$ in $C_{loc} ([0,1] \times [0, \infty))$,
        \item $\widehat M$ is a viscosity solution of \eqref{eq:General_Mass_equation_whole_space} for $\eta \equiv 1$.
    \end{itemize}
    Furthermore, if we also assume that $\inf_{x \in \Rd \backslash B_{\sigma}} V(x) \rightarrow \infty$ when $\sigma \rightarrow \infty$. Then,
    \begin{equation*}\label{eq:no mass escapes through infinty}
        \| \rho_0 \|_{L^1 (\Rd)} - \widehat M(t, v) = \int_{\Rd \backslash \widetilde{B_v}} \widehat\mu = \omega (v^{-1}) \rightarrow 0.
    \end{equation*}
\end{proposition}

The proof is very similar to the one presented in \Cref{prop:M^[n]_converge_uniformly} for the case of the ball $B_R$.

\begin{proof}
    \textbf{Step 1. Ascoli-Arzela over compacts of $[0,1]\times(0, \infty)$.} For any $k \in \mathbb{Z}_{>0}$ we know from \eqref{ec:C_alpha_estimate_in_time_and_space_Mass} that for all $v_1, v_2 \in [\frac{1}{k}, k]$ and $t_1 , t_2 \in [0,1]$ we have the estimate,
    \begin{equation}\label{ec:Uniform_convergence_Mn_to_Minfty Rd}
		\left|M^{[n]} (t_1, v_1) - M^{[n]} (t_2, v_2) \right| \leq C_{k} \left( |v_1 - v_2 | + |t_1- t_2 |^{\frac{1}{m+1}} \right)^{\alpha}.
	\end{equation}
	Furthermore, taking successive subsequences,
	\begin{equation*}
	    M^{[n(k,j)]} \rightarrow \widehat M \quad \text{in }  C\left ([0,1] \times \left[\tfrac 1 k , k \right] \right) \text{ as } j \to \infty,
	\end{equation*}
	and the diagonal satisfies
	\begin{equation}\label{eq:Strong convergence mass Rd}
		M^{[n(k,k)]} \rightarrow \widehat M \quad \text{in }  C_{loc}([0,1] \times (0 , \infty )).
	\end{equation}
	From here on we re-label this sequence again as $M^{[n]}$.
	
	\textbf{Step 2. Extension of $\widehat M$ and some properties.} Since $M^{[n]}(t,0)=0$ for all $n$, the extension of $ \widehat M$ so that $\widehat M (t,0)=0$ is natural. Therefore, due to \eqref{ec:Uniform_convergence_Mn_to_Minfty Rd} we have that
	\begin{equation*}
	    M^{[n]} (t,v) \rightarrow \widehat M (t,v) \quad \text{point-wise in } [0,1] \times [0, \infty ).
	\end{equation*}
	Since $M^{[n]}$ are non-decreasing functions bounded from below by $0$ and from above by $\| \rho_0 \|_{L^1(\mathbb{R}^d)}$, they all have bounded variation $\| \rho_0 \|_{L^1(\mathbb{R}^d)}$. From Banach-Alaoglu Theorem and \eqref{ec:Uniform_convergence_Mn_to_Minfty Rd}, up to a subsequence, we have that,
	\begin{equation}\label{eq:weak star convergence of the derivatives}
	    \frac{\partial M^{[n]}}{\partial v}  \rightharpoonup \frac{\partial \widehat M}{\partial v} \quad \text{weak}-\ast \text{ in } L^{\infty}(0,1; \mathcal{M} ([0, \infty))).
	\end{equation}
	
	\textbf{Step 3. Identification between $\frac{\partial \widehat M}{\partial v}$ and $\changevariables_{\infty} [\widehat\mu ]$.} From \Cref{th:mu infty Rd} we have that,
	\begin{equation*}
	    \rho_t^{\infty , 1, [n]} \rightharpoonup \widehat\mu \quad \text{weak}-\ast \text{ in } L^{\infty}(0,1; \mathcal{M}(\mathbb{R}^d)). 
	\end{equation*}
	Since $\changevariables_{\infty}$ is a linear isometry we also obtain that,
	\begin{equation}\label{eq:I[rho] to I[mu] Rd}
	    \changevariables_{\infty}[ \rho_t^{\infty , 1, [n]} ] \rightharpoonup \changevariables_{\infty}[ \widehat\mu ] \quad \text{weak}-\ast \text{ in } L^{\infty}(0,1; \mathcal{M}([0, \infty )).
	\end{equation}
	Furthermore, since $\changevariables_{\infty} [ \rho_t^{\infty, 1, [n]} ] = \frac{\partial M^{[n]}}{\partial v} (t, \cdot )$, if we combine \eqref{eq:weak star convergence of the derivatives} and \eqref{eq:I[rho] to I[mu] Rd} we obtain that $\frac{\partial \widehat M}{\partial v} = \changevariables_{\infty} [\widehat\mu ]$ in $L^{\infty} (\mathcal{M}([0, \infty)))$. We claim that
	\begin{equation*}
	    \diffEv_{\infty, 1}[\rho^{\infty, 1, [n]}] \rightarrow \diffEv_{\infty,1}[\widehat\mu] \quad \text{in } C_{loc} ([0,1] \times [0, \infty)).
	\end{equation*}
	Let us proceed to prove it. Let us take a cut-off function $\psi \in C_c^{\infty} (\Rd)$ such that $\psi (x) = 1$ if $x \in B_{\sigma}$ and $\psi(x)=0$ if $x \in \Rd \backslash B_{2 \sigma}$.
	\begin{align*}
	    & \limsup_{n \rightarrow \infty} \left| \int_{\Rd} \nabla W (r e_1 - y ) \left[ \rho^{\infty, 1, [n]}_t (y) - \widehat\mu (y)  \right] \, dy \right| \\
	    & \qquad \leq \limsup_{n \rightarrow \infty} \left| \int_{\Rd} \psi(y) \nabla W (r e_1 - y ) \left[ \rho^{\infty, 1, [n]}_t (y) - \widehat\mu (y)  \right] \, dy \right| \\
	    & \qquad \quad + \limsup_{n \rightarrow \infty} \left| \int_{\Rd} (1- \psi (y)) \nabla W (r e_1 - y ) \left[ \rho^{\infty, 1, [n]}_t (y) - \widehat\mu (y)  \right] \, dy \right|.
	\end{align*}
	The first term of the RHS converges to zero since $\psi(\cdot) \nabla W (r e_1 - \cdot ) \in W^{1, \infty}_0 (B_{2 \sigma})$ and $\rho^{\infty, 1, [n]} \rightarrow \widehat \mu $ in $C([0,1]; W^{-1,1}_{loc} (\Rd))$ due to \Cref{th:mu infty Rd}. We focus now on the second term,
	\begin{align*}
	    & \limsup_{n \rightarrow \infty} \left| \int_{\Rd} (1- \psi (y)) \nabla W (r e_1 - y ) \left[ \rho^{\infty, 1, [n]}_t (y) - \widehat\mu (y)  \right] \, dy \right| \\
	    & \qquad \leq \limsup_{n \rightarrow \infty}  \int_{\Rd \backslash B_{\sigma}} |\nabla W (r e_1 - y )| \left[ \rho^{\infty, 1, [n]}_t (y) + \widehat\mu (y)  \right] \, dy \\
	    & \qquad = \limsup_{n \rightarrow \infty}  \int_{\Rd \backslash B_{\sigma}} \frac{|\nabla W (r e_1 - y )|}{V(re_1 - y)} \frac{V(re_1 - y)}{1+V(y)} (1+ V(y))  \left[ \rho^{\infty, 1, [n]}_t (y) + \widehat\mu (y)  \right]  \, dy \\
	    & \qquad \leq C \left\| \frac{|\nabla W|}{V} \right\|_{L^{\infty} (\Rd \backslash re_1 - B_{\sigma})} \left( \limsup_{n \rightarrow \infty} \int_{\Rd \backslash B_{\sigma}}  (1 + V(y)) \rho^{\infty , 1, [n]}_t (y) \, dy + \int_{\Rd \backslash B_{\sigma}}  (1+V(y)) \widehat\mu (y) \, dy  \right) \\
	    & \qquad \rightarrow 0,
	\end{align*}
	as $\sigma \rightarrow \infty$ due to \eqref{eq:Growth of V compare to W}, \eqref{eq:Control of the tails} and \eqref{ec:sup_V_bounded}. 
	
	Finally, using the stability results \cite[Chapter 3 Theorem 2]{Kat14} we get that $\widehat M$ is a viscosity solution of \eqref{eq:General_Mass_equation_whole_space}.
\end{proof}

Let us state a case in which \Cref{rem:Potential concentration Rd} holds and we can characterize $\widehat \mu$.

\begin{proposition}\label{prop:R=infty eta= 1 t=infty regularity in terms of flux}
    Assume all the hypothesis from \Cref{th:locally_strong_solution_Rn}. Assume furthermore that $\rho_0$, $V$ and $W$ are radially symmetric and the technical assumptions \eqref{ec:Free_energy_bounded_below}, \eqref{eq:Growth of V compare to W} and \eqref{eq:Control of the tails}. If $\diffEv_{\infty , 1}[\widehat\mu]$ is such that
    \begin{equation*}\label{ec:Assumption for regularity}
        \diffEv_{\infty,1}[\widehat\mu] (v) \geq 0.
    \end{equation*}
    Then $\widehat M$, the limit in time obtained in \Cref{prop:limit Mn properties}, is $C^2 ((0, \infty))$.
\end{proposition}

The proof is analogous to the one from \Cref{th:Concentration_Phenomena_ball_viscosity_solutions}.

Let us assume that $V$ is a radially increasing and symmetric potential that satisfies the hypothesis presented on \Cref{th:locally_strong_solution_Rn},  \eqref{ec:Free_energy_bounded_below}, \eqref{eq:Growth of V compare to W} and \eqref{eq:Control of the tails}. If we take $V(x) \geq 2 |x|^2 + V_0 (x)$, with $V(0)=0$ and we choose $W(x)=\frac{|x|^2}{2}$. Then $\diffEv_{\infty, 1}[\widehat\mu] \geq 0$ for every $\widehat\mu$ positive and radially symmetric. On \cref{sec:Example of concentration}, we show that there exists concentration for this example.

\section{Examples of concentration phenomena}\label{sec:Example of concentration}

In this section we prove that if $\rho$ solves the Euler-Lagrange equation
\begin{equation}\label{eq:Euler-Lagrange}
	- \frac{m}{1-m} \rho^{m-1} (x) + V(x) + \int_{B_R} \eta (x) W(x-y) \eta (y)  \rho (y) \, dy = -h,
\end{equation}
then, for certain choices of $V$ and $W$, we have concentration phenomena for both: The problem in $B_R$, and the problem in the whole space $\mathbb{R}^d$ with $\eta \equiv 1$. Through this section, we will assume $\eta (x) = 1$ if $|x| \leq \frac{R}{\sqrt{2}}$ whenever we are studying the problem on the ball $B_R$.

\begin{theorem}
	Assume $\rho$ solves the Euler-Lagrange equation \eqref{eq:Euler-Lagrange}, $\rho$ is radial and $\int \rho (x) \, dx \leq 1$. Let us assume that $V$ is a  radially increasing and symmetric potential that satisfies the hypothesis presented on \Cref{th:locally_strong_solution_Rn}, \eqref{ec:Free_energy_bounded_below}, \eqref{eq:Growth of V compare to W} and \eqref{eq:Control of the tails}. If $V(x) \geq 2 |x|^2 + V_0 (x)$, $V(0)=0$ and $W(x) = \frac{|x|^2}{2}$, then
	\begin{equation}\label{eq:Estimate_for_rho_aggregation_example}
		\rho (x) \leq \left( \frac{1-m}{m} V_0 (x) \right)^{- \frac{1}{1-m}}.
	\end{equation}
	In particular, if $\int_{B_R} \left( \frac{1-m}{m} V_0 (x) \right)^{- \frac{1}{1-m}} < 1$ then so is that of $\rho$.
\end{theorem}

\begin{proof}
	Since $W(x) = \frac{|x|^2}{2}$, we can rewrite \eqref{eq:Euler-Lagrange} like,
	\begin{equation*}
		- \frac{m}{1-m} \rho^{m-1} + V + \frac{\eta (x)}{2} |x|^2 A = - h - \frac{\eta (x)}{2} B,
	\end{equation*}
	where
	\begin{equation*}
		A = \int_{B_R} \eta (y) \rho(y) \, dy
	\end{equation*}
	and
	\begin{equation*}
		B = \int_{B_R} \eta(y) |y|^2 \rho (y) \, dy.
	\end{equation*}
	Since $V(0)=0$, if we evaluate on $x=0$ we get that,
	\begin{equation*}
		- \frac{m}{1-m} \rho^{m-1}(0) = - h - \frac{B}{2},
	\end{equation*}
	from where we deduce that 
	\begin{equation*}
		h+ \frac{B}{2} \geq 0.
	\end{equation*}
	Then, we get that $\rho$ is of the form,
	\begin{equation*}
		\rho (x) = \left[ \frac{1-m}{m} \left( h + \frac{\eta (x)}{2} B + \frac{\eta (x)}{2} |x|^2 A + V(x) \right) \right]^{- \frac{1}{1-m}}.
	\end{equation*}
	We distinguish between two cases. First, if $|x| \leq \frac{R}{\sqrt{2}}$, we have that $\eta (x) = 1$ and $h + \frac{\eta(x)}{2} B \geq 0$. Therefore,
	\begin{equation*}
		\rho (x) \leq \left( \frac{1-m}{m} V (x) \right)^{- \frac{1}{1-m}} \leq \left( \frac{1-m}{m} V_0 (x) \right)^{- \frac{1}{1-m}}
	\end{equation*}
	and we recover \eqref{eq:Estimate_for_rho_aggregation_example}. For the second case, if $|x| > \frac{R}{\sqrt{2}}$,
	\begin{equation*}
		0 \leq B = \int_{B_R} \eta (y) |y|^2 \rho (y) \, dy \leq R^2 \int_{B_R} \rho (y) \, dy \leq R^2.
	\end{equation*}
	Thus, $h \geq - \frac{B}{2} \geq - \frac{R^2}{2}$, and, in particular,
	\begin{equation*}
		\rho (x) \leq \left[ \frac{1-m}{m} \left( - \frac{R^2}{2} + V(x) \right)  \right]^{- \frac{1}{1-m}}.
	\end{equation*}
	$V(x) \geq |x|^2 + V_0 (x) \geq \frac{R^2}{2} + V_0 (x)$ when $|x| \geq \frac{R}{\sqrt{2}}$. Then,
	\begin{equation*}
		\rho (x) \leq \left( \frac{1-m}{m}   V_0(x)  \right)^{- \frac{1}{1-m}},
	\end{equation*}
	and we recover \eqref{eq:Estimate_for_rho_aggregation_example} again, finishing with the proof.
\end{proof}

\begin{remark}
	The proof works in the same way when we are in the whole space $\mathbb{R}^d$ and $\eta \equiv 1$. In this case, we just need $V(x) \geq V_0(x)$. In the same way as before, we also obtain examples where aggregation happens.
\end{remark}

\subsection*{Acknowledgements}
The authors were supported by the Advanced Grant Nonlocal-CPD (Nonlocal PDEs for Complex Particle Dynamics: Phase Transitions, Patterns and Synchronization) of the European Research Council Executive Agency (ERC) under the European Union’s Horizon 2020 research and innovation programme (grant agreement No. 883363). JAC was also partially supported by the EPSRC grant numbers EP/T022132/1 and EP/V051121/1. DGC was partially supported by PID2021-127105NB-I00 from the Spanish Government.

\appendix

\section{A maximum principle for viscosity solutions}\label{ap:Maximum principle viscosity solutions}
In \cite{KK98}, Kawohl and Kutev study the subsolutions of the general elliptic equation
\begin{equation}\label{ec:KK98_Elliptic_general_form}
	F(x, u, Du, D^2 u) = 0 \quad \mathrm{in} \, \Omega ,
\end{equation}
where $\Omega$ is a connected domain and $F$ is a Caratheodory function. They assume,
\begin{enumerate}
	\item For every $(x, p, X) \in \Omega \backslash N \times \mathbb{R}^d \times S^n$, where $N \subseteq \Omega$ denotes a nowhere dense set, $F$ satisfies,
	\begin{equation}\label{ec:Maximum_Principle_viscosity_Tech_assumption_(5)}
		r \geq s \quad \Rightarrow \quad 0 \leq F( x, r, p, X ) - F (x, s, p, X).
	\end{equation} 
	\item Equation \eqref{ec:KK98_Elliptic_general_form} is strictly elliptic with modulus of ellipticity $\omega_1$
	\begin{equation}\label{ec:Maximum_Principle_viscosity_Tech_assumption_(8)}
		X \geq Y \quad \Rightarrow \quad \omega_1 (\lambda_K \trace (X-Y) ) \leq F( x, r, p, Y) - F(x, r, p, X),
	\end{equation}
	for every $x \in \Omega \backslash N$, and $|r|, \, |p|, \, \|X \|, \, \| Y \| \leq K$, and for some positive constant $\lambda_K$ depending on $K$. Furthermore, as a modulus of continuity the real function $s \mapsto \omega_1 (s)$ satisfies the usual conditions,
	\begin{equation}\label{ec:Maximum_Principle_viscosity_Tech_assumption_(9)}
		\omega (s) \in C [0, \infty), \quad \omega (s) > \omega (t) \, \, \text{for} \, s > t > 0 \quad \text{and} \quad \omega (0) = 0.
	\end{equation}
	\item For every $x \in \Omega \backslash N$, and $|r|, |p|, |q|, \| X \|, \| Y \| \leq K$ there exists some positive constants $A_K$ and $B_K$ depending on $K$ such that
	\begin{equation}\label{ec:Maximum_Principle_viscosity_Tech_assumption_(10)}
		|F(x, r, p, X) - F(x, r, q, Y)| \leq \omega_2 (A_K \| X-Y\| + B_K |p-q| )
	\end{equation}
	for $\omega_2$ some modulus of ellipticity satisfying \eqref{ec:Maximum_Principle_viscosity_Tech_assumption_(9)}.
	\item $\Omega$ is connected and there exists a positive $s_0$ such that 
	\begin{equation}\label{ec:Maximum_Principle_viscosity_Tech_assumption_(extra)}
		\omega_1 (s) \geq \omega_2 (s) \quad s \in [0, s_0].
	\end{equation}
	\item $F$ is such that 
	\begin{equation}\label{ec:Maximum_Principle_viscosity_Tech_assumption_(11)}
		F(x,0,0,0) \geq 0 \quad \text{for every} \, x \in \Omega \backslash N .
	\end{equation}
\end{enumerate}
Then, we have the following strong interior maximum principle.
\begin{theorem}[Kawohl-Kutev, \cite{KK98}]\label{th:KK98}
	Assume $F(x,z,p,X)$ satisfies the hypothesis \eqref{ec:Maximum_Principle_viscosity_Tech_assumption_(5)}-\eqref{ec:Maximum_Principle_viscosity_Tech_assumption_(11)}. Then, if $u \in USC(\Omega )$ is a viscosity supersolution of \eqref{ec:KK98_Elliptic_general_form} and if $u$ has a maximum $\Theta$ at some interior point $x_0 \in \Omega$, then $u (x)$ must be constant in $\Omega$, $F(x,0,0,0) \equiv 0$ everywhere in $\Omega \backslash N$, and $F(x, z, 0, 0)$ must be independent of $z$ for $z \in [0, \Theta]$ and $x \in \Omega \backslash N$.
\end{theorem}

\section{Uniform estimates for \texorpdfstring{$\rho$}{rho} as \texorpdfstring{$R \to \infty$}{R to infty}}\label{ap:a priori estimates Rd}

In this appendix we include some of the \textit{a priori} estimates that we need in \cref{sec:Rd} to be able to prove \Cref{th:locally_strong_solution_Rn}.

First, we realise that using the estimate for the second order moment \eqref{ec:Second_moment_uniformly_bounded_above} we are able to obtain the \textit{a priori} estimate for a $p$-order moment with $p=\frac{2}{1-m} > 2$.

\begin{proof}[Proof of \Cref{lem:p-order moment}]
	In the same way we did before for the second moment, we take the time derivative and compute,
	\begin{equation} 
    \label{eq:d/dt mp}
    \begin{aligned}
		\frac{1}{p}\frac{d}{dt} \mathbf{m_p} (\rho^{R, \eta_j, \ast}_t ) & = -  \int_{\mathbb{R}^d} |x|^{p-2}x \cdot \nabla (\rho^{R, \eta_j, \ast}_t)^m \\
        &\qquad -  \int_{\mathbb{R}^d} |x|^{p-2} x \cdot \rho^{R, \eta_j, \ast} \nabla \left(V_R + \int_{\Rd} K_{\eta}(x,y) \rho^{R, \eta_j, \ast}_t (y) \, dy \right) \\
		&\le  (p-1) \mathbf{m_{p-2}} ((\rho_t^{R, \eta_j, \ast})^m) \\
        &\qquad +    \left\| \frac{\nabla V_R + \nabla \int_{\Rd} K_{\eta}(\cdot,y) \rho^{R, \eta_j, \ast}_t (y) \, dy}{1+|\cdot|} \right\| _{L^\infty(\mathbb{R}^d)} \int_{\Rd} |x|^{p-1}  (1 + |x|) \rho^{R, \eta_j, \ast}_t \, dx. 
	\end{aligned}
    \end{equation}
Now we bound each of the terms. For the first term
we have that
	\begin{align*}
		\mathbf{m_{p-2}} \left((\rho_t^{R, \eta_j, \ast})^m\right)      & \le  \left( \mathbf{m_{\frac{p-2}m}}  (\rho^{R, \eta_j, \ast}_t)\right)^m  \| \rho^{R, \eta_j, \ast}_t\|_{L^{\frac{1}{1-m}}} \leq m \, \mathbf{m_{p}}  (\rho^{R, \eta_j, \ast}_t) + (1-m)  \| \rho^{R, \eta_j, \ast}_t\|_{L^{\frac{1}{1-m}}}^{\frac{1}{1-m}}.
	\end{align*}
    Notice that we have chosen $p$ such that $p = \frac{p-2}{m}$. The second term is bounded due to the $L^p$ estimate \eqref{ec:Lp_estimate}. 
    The term $\nabla V_R / (1+ |x|) $ is uniformly bounded by hypothesis. We expand the derivative
	\begin{align*}
		\nabla \int_{\Rd} K(x,y) \rho^{R, \eta_j, \ast}_t (y) \, dy&= \int_{B_R} \eta_j (x) \nabla W(x-y) \eta_j (y) \rho^{R, \eta_j, \ast}_t (y) \, dy  \\
  &\qquad +  \int_{\Rd} \nabla \eta_j (x) W(x-y) \eta_j (y) \rho_t^{R, \eta_j, \ast}(y) \, dy. 
    \end{align*}
    Since we assume $|\nabla W(x)| \leq C (1+|x|)$ from \eqref{eq:W_quadratic} and from there we also deduce $W(x-y) \leq C (1+|x|^2) (1 + |y|^2)$ at \eqref{ec:Assumptions_on_W_whole_space}, we obtain that 
    \begin{align*}
        \left| \int_{B_R} \eta_j (x) \nabla W(x-y) \eta_j (y) \rho^{R, \eta_j, \ast}_t (y) \, dy \right| & \le  C (W)  \int_{\mathbb{R}^d} \eta_j (x) (1+ |x| + |y| ) \eta_j (y) \,  \rho_t^{R, \eta_j, \ast} (y)  \, dy \\
		& \leq C (W) \left(  (1+|x|) \| \rho_t^{R, \eta_j, \ast} \|_{L^1 (\mathbb{R}^d)} +   \mathbf{m_1}  (\rho_t^{R, \eta_j, \ast})\right)  \\
  		& \leq C (W) \left( (1+|x|) \| \rho_0 \|_{L^1 (\mathbb{R}^d)} +  \| \rho_0 \|_{L^1 (\mathbb{R}^d)}^{\frac{1}{2}}  +  \mathbf{m_2} ( \rho^{R, \eta_j, \ast}_t )^{\frac{1}{2}} \right) 
    \end{align*} 
    For the next term, we take into account that $ |\nabla \eta_j(x)| \le j^{-1}  \|\nabla \eta_1\|_{L^\infty(\Rd)} \chi_{B_j} (x)  $, and we can therefore estimate 
    \begin{align*}
        \left| \int_{\Rd} \nabla \eta_j (x) W(x-y) \eta_j (y) \rho_t^{R, \eta_j, \ast}(y) \, dy \right| 
        &   \leq  C(W) | \nabla \eta_j (x) |  ( 1 + |x|^2) \left( \| \rho_0 \|_{L^1 (\Rd)} + \mathbf{m_2} (\rho_t^{R, \eta_j, \ast})  \right) \\
      &   \leq  C(W) \| \nabla \eta_1  \|_{L^\infty(\Rd) }  ( j^{-1} + |x|) \left( \| \rho_0 \|_{L^1 (\Rd)} + \mathbf{m_2} (\rho_t^{R, \eta_j, \ast})  \right)
	\end{align*}
	where $C(W)$ is the combination of the constants appearing in \eqref{eq:W_quadratic} and \eqref{ec:Assumptions_on_W_whole_space},  
    and we control $\mathbf{m_2}$ due to \eqref{ec:Second_moment_uniformly_bounded_above}.
    Lastly, using Hölder inequality on the $p-1$ moment and Young's inequality afterwards we get,
	\begin{align*}
		\int |x|^{p-1}  (1 + |x|) \rho^{R, \eta_j, \ast}_t &=  \mathbf{m_{p-1}} (\rho^{R, \eta, \ast}_t) + \mathbf{m_p} (\rho^{R, \eta_j, \ast}_t) \le      \| \rho^{R, \eta_j, \ast}_t \|_{L^1 (\mathbb{R}^d)}^{\frac{1}{p}}    \mathbf{m_p} (\rho^{R, \eta_j, \ast}_t)^{\frac{p-1}{p}} + \mathbf{m_p} (\rho^{R, \eta_j, \ast}_t)     \\
		& \le \frac{ \| \rho^{R, \eta_j, \ast}_t \|_{L^1 (\mathbb{R}^d)}}{p} + \frac{2p-1}{p}    \mathbf{m_p} (\rho^{R, \eta_j, \ast}_t)   .
	\end{align*}
	Integrating \eqref{eq:d/dt mp} in time from $0$ to $t$ and using the previous computations we get that
	\begin{align*}
	    \mathbf{m_p} (\rho^{R, \eta_j, \ast}_t) - \mathbf{m_p} (\rho^{R, \eta_j, \ast}_0) & \leq C \left( 1  +   \int_0^t   \mathbf{m_p} (\rho_s^{R, \eta_j, \ast} ) \, ds \right)
	\end{align*}
	where $C$ is a constant with the same dependencies as $A$ and $B$ in the statement.
	Therefore, by Gronwall's inequality, we recover \eqref{ec:p_moment_uniformly_bounded_above}.
\end{proof}

Once we have obtain an estimate for the $p=\frac{2}{1-m}$ order moment we focus on obtaining a collection of \textit{a priori} results. In order to do that, let us take $\varphi \in C^{\infty}_c (\Rd)$ such that $\varphi (x) = 1$ for all $x \in B_{\omega_1}$ and $\varphi (x) = 0$ for all $x \in \Rd \backslash B_{\omega_2}$. We fix a cut-off function $\varsigma$ such that $\varsigma (x) = \varphi^2(x)$. Take $R$ large enough so that ${\omega_1} < {\omega_2} < {R}$. We present here the following collection of \textit{a priori} results.

\begin{lemma}[\textit{a priori} estimates on $\nabla \rho^{R, \eta}$]\label{lem:rho H1 Rd}
    Let $\rho^{R, \eta}$ be the unique strong solution of \eqref{eq:Fast-Diffusion_Problem_BR} in the bounded domain $(0,T) \times B_R$ for the kernel $K (x,y)= \eta(x) W(x-y)\eta(y)$. We have that
    \begin{equation}
    	\label{eq:rho H1 Rd}
    	\begin{aligned} 
	    \int_0^T& \int_{B_{R}} | \nabla \rho_t^{R, \eta}|^2 \varsigma \\
	    & \leq C \left( m, \| \rho^{R, \eta} \|_{L^{3-m}((0,T) \times \Rd)}, \| \rho^{R, \eta} \|_{L^{\infty}((0,T) \times \Rd)}, \| \varsigma^{\frac{1}{2}}\|_{W^{1, \infty} (\Rd) } , \| E^{R, \eta} \|_{L^2 ((0,T) \times B_{\omega_2})} \right).
	    \end{aligned} 
	\end{equation}
\end{lemma}

\begin{proof}
     We use a similar strategy to the one followed for the \textit{a priori} estimate \eqref{ec:a_priori_estimate_rho}. In this case, the function $G_{\Phi}$ appearing in \eqref{ec:a_priori_estimate_rho} corresponds to $G(s) = \frac{1}{m(2-m)(3-m)} s^{3-m}$, where we define $G$ using \eqref{def:G_Phi} for the function $\Phi(s) = s^m$.  For $R$ large enough we get that,
	\begin{align*}
		\partial_t \int_{B_{R}} ( G (\rho^{R, \eta }_t) ) \varsigma & = - \int_{B_{R}} \nabla (G' (\rho_t^{R, \eta}) \varsigma ) \cdot \left( \nabla (\rho_t^{R, \eta})^m + \rho_t^{R, \eta} E_t^{R, \eta} \right)   \\
		& = - \int_{B_{R}} | \nabla \rho_t^{R, \eta} |^2 \varsigma - \int_{B_{R}} \frac{1}{m} (\rho_t^{R, \eta})^{2-m} \nabla \rho_t^{R, \eta} \cdot E_t^{R, \eta} \varsigma \\
		& \quad - \int_{B_{R}} \frac{1}{m (2-m)} (\rho_t^{R, \eta})^{2-m} \Big ( \nabla (\rho_t^{R, \eta} )^m + \rho_t^{R, \eta} E_t^{R, \eta} \Big) \nabla \varsigma . 
	\end{align*}
	Thus, using Hölder, Young's inequality and integrating in time we get \eqref{eq:rho H1 Rd}.
\end{proof}

\begin{lemma}[\textit{a priori} estimates on $\nabla (\rho^{R, \eta})^m$]
    Let $\rho^{R, \eta}$ be the unique strong solution of \eqref{eq:Fast-Diffusion_Problem_BR} in the bounded domain $(0,T) \times B_R$ for the kernel $K (x,y)= \eta(x) W(x-y)\eta(y)$. We have that
    \begin{equation}\label{eq:rhom H1 Rd}
    	\begin{aligned} 
	    &\int_0^T  \int_{B_{R}}  | \nabla (\rho_t^{R , \eta})^m |^2 \varsigma \\
	    & \leq  C\left(m, \| \rho^{R, \eta} \|_{L^{m+1}((0,T) \times \Rd)}, \|\rho^{R, \eta}\|_{L^\infty((0,T) \times \Rd)} , \| \varsigma^{\frac 1 2} \|_{W^{1,\infty}(\Rd)} , \|E^{R, \eta}\|_{L^\infty ((0,T) \times B_{\omega_2})} \right) .
	    \end{aligned} 
	\end{equation}
\end{lemma}

\begin{proof}
      With a similar strategy to the one followed for the \textit{a priori} estimate \eqref{ec:a_priori_estimate_Phi} we obtain that,
	\begin{align*}
	\begin{split}
		\partial_t \int_{B_{R}} \frac{1}{m+1} (\rho_t^{R , \eta})^{m+1} \varsigma & = \int_{B_{R}} (\rho_t^{R , \eta})^m \partial_t \rho_t^{R , \eta} \varsigma \\
		& = \int_{B_{R}} (\rho_t^{R , \eta})^m \nabla \cdot ( \nabla (\rho_t^{R , \eta})^m + \rho_t^{R , \eta} E_t^{R, \eta}) \varsigma \\
		& = - \int_{B_{R}} | \nabla (\rho_t^{R , \eta})^m |^2 \varsigma - \int_{B_{R}} \rho_t^{R , \eta} \nabla (\rho_t^{R , \eta})^m E_t^{R, \eta} \varsigma \\
		& \quad - \int_{B_{R}} (\rho_t^{R , \eta})^m \nabla (\rho_t^{R , \eta})^m \nabla \varsigma - \int_{B_{R}} (\rho_t^{R , \eta})^{m+1} E_t^{R, \eta} \nabla \varsigma .
		\end{split}
	\end{align*}
	Thus, using Hölder, Young's inequality and integrating in time we get \eqref{eq:rhom H1 Rd}.
\end{proof}

\begin{lemma}[\textit{a priori} estimates on $\partial_t \rho^{R, \eta}$]\label{lem:d/dt rho L2 Rd}
    Let $\rho^{R, \eta}$ be the unique strong solution of \eqref{eq:Fast-Diffusion_Problem_BR} in the bounded domain $(0,T) \times B_R$ for the kernel $K (x,y)= \eta(x) W(x-y)\eta(y)$. We have that
    \begin{equation}\label{eq:d/dt rho L2 Rd}
    	\begin{aligned} 
	    &\int_0^T \int_{B_{R}} \left| \frac{\partial \rho_t^{R, \eta}}{\partial t} \right|^2 \varsigma \\
	    &\leq C \left(m, \| \rho^{R, \eta} \|_{L^{\infty}((0,T) \times \Rd)}, \| (\rho^{R, \eta})^m \|_{L^{\infty}((0,T); H^1(B_{\omega_2}))} \| \varsigma^{\frac{1}{2}} \|_{W^{1, \infty} (\Rd)}, \| E^{R, \eta} \|_{W^{1,2}((0,T); L^2 (B_{\omega_2})} \right).
	    \end{aligned} 
	\end{equation}
\end{lemma}

\begin{proof}
      This time we use a similar strategy to the one from the \textit{a priori} estimate \eqref{ec:a_priori_estimate_drho/dt_L2}. Applying Young's inequality we have that,
	\begin{align*}\label{ec:a_priori_estimate_drho/dt_L2_Cacciopoli}
		\begin{split}
			\int_{B_{R}} m (\rho_t^{R, \eta})^{m-1}  \left| \frac{\partial \rho_t^{R, \eta}}{\partial t} \right|^2 \varsigma & = \int_{B_{R}} \frac{\partial}{\partial t} (\rho_t^{R, \eta} )^m \cdot \left( \Delta (\rho_t^{R, \eta})^m + \nabla \cdot ( \rho_t^{R, \eta} E_t^{R, \eta} )\right) \varsigma \\
			& = - \int_{B_{R}} \nabla \left( \frac{\partial}{\partial t} (\rho_t^{R, \eta})^m \right) \nabla (\rho_t^{R, \eta})^m \varsigma -  \int_{B_{R}} \frac{\partial}{\partial t} (\rho_t^{R, \eta})^m \nabla (\rho_t^{R, \eta})^m \nabla \varsigma \\
			& \quad - \int_{B_{R}} \rho_t^{R, \eta} E_t^{R, \eta} \nabla \left( \frac{\partial}{\partial t} (\rho_t^{R, \eta})^m \right) \varsigma - \int_{B_{R}} \rho_t^{R, \eta} E_t^{R, \eta} \frac{\partial}{\partial t} (\rho_t^{R, \eta})^m \nabla \varsigma.
		\end{split}
	\end{align*}
	Integrating in time from $0$ to $T$ we get that,
	\begin{align*}
	    \begin{split}
	        	\int_0^T \int_{B_{R}} m (\rho_t^{R, \eta})^{m-1}  \left| \frac{\partial \rho_t^{R, \eta}}{\partial t} \right|^2 \varsigma & - \frac{1}{2} \int_{B_{R}} | \nabla \rho_0^m |^2 \varsigma + \frac{1}{2} \int_{B_{R}} | \nabla (\rho_T^{R, \eta})^m |^2 \varsigma \\
	        	& = \int_0^T \int_{B_{R}} \left( \nabla (\rho_t^{R, \eta})^m + (\rho_t^{R, \eta})^m \right) \left( E_t^{R, \eta} \frac{\partial \rho_t^{R, \eta}}{\partial t} + \rho_t^{R, \eta} \frac{\partial E_t^{R, \eta}}{\partial t} \right) \varsigma \\
	        	& \quad - \int_0^T \int_{B_{R}} \nabla (\rho_t^{R, \eta})^m \frac{\partial}{\partial t}  (\rho_t^{R, \eta})^m \nabla \varsigma .
	    \end{split}
	\end{align*}
	Let us bound each one of the terms. The first one can be bounded by,
	\begin{align*}
	    & \int_0^T \int_{B_{R}} \left( \nabla (\rho_t^{R, \eta})^m + (\rho_t^{R, \eta})^m \right) \left( E_t^{R, \eta} \frac{\partial \rho_t^{R, \eta}}{\partial t} \right) \varsigma \\
	        & \qquad \leq \frac{\| \rho_t^{R, \eta} \|_{L^{\infty}((0,T) \times \Rd)}^{1-m}}{ m } \int_0^T  \|  (\rho_t^{R, \eta})^m \|^2_{H^1(B_{\omega_1})} \|E_t^{R, \eta}  \varsigma^{\frac{1}{2}} \|_{L^2(B_{\omega_1})} \\
	        & \qquad \quad + \frac{m}{4 \| \rho_t^{R, \eta} \|_{L^{\infty}((0,T) \times \Rd)}^{1-m}} \int_0^T \int_{B_{R}} \left| \frac{\partial}{\partial t} \rho_t^{R, \eta} \right|^2 \varsigma. 
	    \end{align*}
	    The second one by,
	    \begin{align*}
	    &\int_0^T \int_{B_{R}} \left( \nabla (\rho_t^{R, \eta})^m + (\rho_t^{R, \eta})^m \right) \left(  \rho_t^{R, \eta} \frac{\partial E_t^{R, \eta}}{\partial t} \right) \varsigma \\
	        & \qquad \leq  \int_0^T \|  (\rho_t^{R, \eta})^m \|_{H^1 (B_{\omega_1})} \| \rho_t^{R, \eta} \|_{L^{\infty}(B_{\omega_1})} \left \| \frac{\partial E_t^{R, \eta}}{\partial t} \varsigma  \right\|_{L^2 (B_{R})}. 
	   \end{align*}
	   And the last one by,
	   \begin{align*}
	    & - \int_0^T \int_{B_{R}} \nabla (\rho_t^{R, \eta})^m \frac{\partial}{\partial t}  (\rho_t^{R, \eta})^m \nabla \varsigma \\
	    & \qquad  \leq \int_0^T \| \nabla (\rho_t^{R, \eta})^m   \|_{L^2( B_{\omega_2})} \| \varsigma^{- \frac{1}{2}} \nabla \varsigma \|_{L^{\infty} (B_{\omega_2})}  \Big\| \frac{\partial (\rho_t^{R, \eta})^m}{\partial t}  \varsigma^{\frac{1}{2}} \Big\|_{L^2 (B_{\omega_2})} \\
	        	& \qquad \quad +  \frac{1}{  m } \int_0^T \| \nabla (\rho_t^{R, \eta_j})^m   \|_{L^2( B_{\omega_2})}^2 \| \varsigma^{- \frac{1}{2}} \nabla \varsigma \|_{L^{\infty} (B_{\omega_2})}^2 + \frac{m}{4 \| \rho_t^{R, \eta} \|_{L^{\infty}((0,T) \times \Rd)}^{1-m}} \int_0^T \int_{B_{R}}  \left| \frac{\partial }{\partial t} \rho_t^{R, \eta} \right|^2  \varsigma.
	\end{align*}
	Thus, putting everything together, we obtain \eqref{eq:d/dt rho L2 Rd}.
\end{proof}

\begin{lemma}[\textit{a priori} estimates on $\Delta ( \rho^{R, \eta} )^m$]
    Let $\rho^{R, \eta}$ be the unique strong solution of \eqref{eq:Fast-Diffusion_Problem_BR} in the bounded domain $(0,T) \times B_R$ for the kernel $K (x,y)= \eta(x) W(x-y)\eta(y)$. We have that
    \begin{equation}\label{eq:Delta rhom L2 Rd}
	    	 \int_0^T \int_{B_{R}} \left| \Delta ((\rho^{R, \eta}_t)^m)  \right|^2 \, \varsigma    \leq C,
	\end{equation}
	where $C$ depend on the constants \eqref{eq:rho H1 Rd} and \eqref{eq:d/dt rho L2 Rd}.
\end{lemma}

\begin{proof}
      We have that,
      \begin{align*}
          \int_0^T \int_{B_R} \left| \Delta ((\rho_t^{R, \eta})^m)  \right|^2 \, \varsigma & = \int_0^T \int_{B_R} \left| \frac{\partial \rho^{R, \eta}_t}{\partial t} \right|^2 \varsigma + \int_0^T \int_{B_R} \left| \diver \left( \rho_t^{R, \eta} E_t^{R, \eta} \right) \right|^2 \varsigma\\
          & \quad - 2 \int_0^T \int_{B_R} \frac{\partial \rho_t^{R, \eta}}{\partial t} \diver \left( \rho_t^{R, \eta} E_t^{R, \eta} \right) \varsigma.
      \end{align*}
      From here, using the results from \Cref{lem:rho H1 Rd} and \Cref{lem:d/dt rho L2 Rd} combined with Hölder inequality we recover \eqref{eq:Delta rhom L2 Rd}.
\end{proof}

\bibliography{content/bibliography.bib}
\bibliographystyle{abbrv}

\end{document}